\documentclass[12pt,letterpaper,twoside,leqno]{amsart}

\usepackage[T1]{fontenc}
\usepackage[utf8]{inputenc}
\usepackage{times}

\usepackage[
  letterpaper,
  left=1.2in,
  right=1.2in,
  top=1in,
  textheight=8.68in,
  headheight=8pt,
  headsep=14pt,
  footskip=12pt
]{geometry}

\usepackage{enumitem}
\usepackage{xcolor}
\usepackage{booktabs}
\usepackage{graphicx}
\usepackage{etoolbox}

\usepackage{amsmath,amssymb,amsthm,mathtools}
\numberwithin{equation}{section}

\usepackage[
  colorlinks=true,
  linkcolor=blue,
  citecolor=red,
  urlcolor=blue
]{hyperref}

\usepackage[
  nameinlink,
  noabbrev,
  capitalise
]{cleveref}

\title{%
  Sign-changing solutions for the critical Neumann problem
  in the three-dimensional unit ball
}

\author{Fei Wu}

\address{%
  Fei Wu, Department of Mathematics, Faculty of Science,
  University of Macau, Taipa, Macau
}

\date{October 2, 2026\ (Last Typeset)}



\theoremstyle{plain}
\newtheorem{theorem}{Theorem}[section]
\newtheorem{lemma}[theorem]{Lemma}
\newtheorem{proposition}[theorem]{Proposition}

\theoremstyle{definition}

\theoremstyle{remark}
\newtheorem{remark}[theorem]{Remark}

\AddToHook{env/thebibliography/begin}{%
  \small
}

\makeatletter
\patchcmd{\thebibliography}
  {\labelsep .5em\relax}
  {\labelsep 5pt\relax}
  {}
  {%
    \PackageWarning{paperstyle}{%
      Bibliography label spacing was not patched%
    }%
  }
\makeatother

\allowdisplaybreaks[1]

\setlist[enumerate]{%
  label=\textup{(\roman*)},
  leftmargin=*
}

\setlist[itemize]{%
  leftmargin=*
}

\newcommand{\eps}{\varepsilon}

\DeclareMathOperator{\supp}{supp}

\crefname{equation}{equation}{equations}
\crefname{section}{Section}{Sections}
\crefname{appendix}{Appendix}{Appendices}

\definecolor{draftblue}{RGB}{36,82,128}

\begin{document}

\begin{abstract}
\noindent We consider the existence of nonradial sign-changing solutions to the problem
\[
\Delta u-\mu u+|u|^4u=0
\]
with zero Neumann boundary conditions in the unit ball for an arbitrary fixed parameter $\mu>0$.
We prove that, for every $\mu>0$, this problem admits infinitely many
distinct nonradial sign-changing solutions. More precisely, for every
sufficiently large even integer $K$, we construct a solution by gluing
$K$ symmetrically arranged, suitably transformed copies of a sign-changing
entire solution of the critical equation in $\mathbb{R}^3$.
The construction combines an inner--outer gluing scheme with a
finite-dimensional reduction. A detailed expansion of the reduced energy
allows us to locate an interior critical point in the admissible
parameter region and thereby complete the construction.
\end{abstract}

\maketitle

\tableofcontents
\clearpage

\section{Introduction}
\label{sec:introduction}

Let $B_1$ be the unit ball in $\mathbb{R}^3$. We study sign-changing
solutions of the critical Neumann problem
\begin{equation}
\label{eq:problem}
\begin{cases}
\Delta u-\mu u+u^5=0 & \text{in }B_1,\\
\partial_\nu u=0 & \text{on }\partial B_1,
\end{cases}
\end{equation}
where $\mu>0$ and $\nu$ is the outward unit normal. Our aim is to
construct infinitely many nonradial nodal solutions for each fixed
value of $\mu$.

Problem~\eqref{eq:problem} belongs to the class of equations underlying
the Lin--Ni conjecture. On a smooth bounded domain
$\Omega\subset\mathbb{R}^N$, $N\geq3$, Lin, Ni, and
Takagi~\cite{LinNiTakagi1988} proved that, if
$1<p<(N+2)/(N-2)$ and $\mu$ is sufficiently small, the only positive
solution of $\Delta u-\mu u+u^p=0$ with homogeneous Neumann boundary
conditions is $u\equiv\mu^{1/(p-1)}$. Lin and Ni~\cite{LinNi1988}
asked whether this conclusion remains valid at the critical exponent
$p=(N+2)/(N-2)$.

The answer depends on the dimension and the domain. In dimension three,
Zhu~\cite{Zhu1999} and Wei and Xu~\cite{WeiXu2005} established the
conjecture for smooth bounded convex domains. In dimension five, Rey
and Wei~\cite{ReyWei2005} constructed positive solutions with any
prescribed number of interior concentration points as $\mu\to0$.
The smallness threshold for $\mu$ in this construction depends on the
number of concentration points.

Concentration phenomena for Neumann equations have also been studied
through variational and perturbative methods. Ni, Pan, and
Takagi~\cite{NiPanTakagi1992} analyzed the singular behavior of
least-energy solutions at the critical Sobolev exponent.
For related singularly perturbed Neumann problems, Gui and
Wei~\cite{GuiWei1999} constructed solutions with several interior
peaks by reducing the problem to the choice of their locations.
These studies provide a framework for understanding how concentration
is affected by the boundary and by interactions between peaks.

For a fixed coefficient, Wang, Wei, and Yan~\cite{WangWeiYan2010}
proved the existence of infinitely many positive solutions in suitable
symmetric nonconvex domains in every dimension $N\geq3$. They subsequently
established the same multiplicity result in symmetric domains including
balls for $N\geq4$~\cite{WangWeiYan2011}. These solutions have unbounded
energy, and the number of bubbles serves as the large parameter.
The use of an increasing number of peaks to obtain solutions with
unbounded energy also appears in the work of Wei and
Yan~\cite{WeiYan2010} on prescribed scalar curvature.
In the three-dimensional ball, however, the uniqueness results above
exclude nonconstant positive solutions when $\mu$ is small. This leads
us to consider multiplicity in the nodal class.

Related constructions in dimension three provide further motivation.
Del Pino, Musso, Rom\'an, and Wei~\cite{DelPinoMussoRomanWei2019}
constructed positive interior bubbling solutions, while
Salazar~\cite{Salazar2019} obtained sign-changing bubbling solutions,
including two-bubble solutions in thin annuli. In these works, the
coefficient approaches a distinguished positive value determined by
Green-function data.

For the Dirichlet Brezis--Nirenberg equation
$\Delta v+\lambda v+|v|^{4/(N-2)}v=0$~\cite{BrezisNirenberg1983},
the loss of compactness at the critical exponent has led to several
approaches to multiplicity. In dimensions $N\geq7$, Devillanova and
Solimini~\cite{DevillanovaSolimini2002} obtained infinitely many
solutions using concentration estimates. Schechter and
Zou~\cite{SchechterZou2010} proved the existence of infinitely many
sign-changing solutions for every fixed $\lambda>0$ in the same
dimensional range, using estimates of the Morse indices of nodal
solutions.

Another approach is to construct concentrating solutions and determine
their parameters through Green-function expansions. Rey~\cite{Rey1990} studied the connection between concentration
phenomena and the regular part of the Green function, while
Druet~\cite{Druet2002} investigated the relation between the existence
of positive solutions in dimension three and the associated
Robin function. Musso and Salazar~\cite{MussoSalazar2018} constructed positive
solutions with several interior concentration points as $\lambda$
approaches a suitable value determined by the Green function.
For the three-dimensional ball, Sun, Wei, and
Yang~\cite{SunWeiYangFull} constructed infinitely many nonradial
sign-changing solutions for every $\lambda>0$, resolving the question
recalled by Brezis~\cite{Brezis2023}. Their related
work~\cite{SunWeiYangSmall} treats the small-parameter regime.

The choice of the entire profile is central to these constructions.
The classification results of Caffarelli, Gidas, and
Spruck~\cite{CaffarelliGidasSpruck1989} and Chen and
Li~\cite{ChenLi1991} identify the positive entire solutions of the
critical Yamabe equation as translations and dilations of the standard
bubble. In the nodal class, Ding~\cite{Ding1986} used variational
methods to obtain infinitely many finite-energy solutions.
Del Pino, Musso, Pacard, and
Pistoia~\cite{DelPinoMussoPacardPistoia2011} subsequently gave a
construction with a prescribed arrangement of positive and negative
peaks. Their profiles have a positive central bubble surrounded by
negative bubbles placed along a regular polygon.
Musso and Wei~\cite{MussoWei2015} proved nondegeneracy results for
these nodal profiles, identifying the kernel of the linearized
operator with the directions generated by the symmetries of the
equation. This makes the profiles suitable for a finite-dimensional
reduction after imposing orthogonality conditions.

Our construction uses these entire nodal profiles and follows the
gluing strategy developed in \cite{SunWeiYangFull}.
Conformal transformations near a regular point of the nodal set
control the leading far-field term of the profile, which is
particularly important in dimension three. We project the resulting
profiles onto the Neumann problem in a fundamental sector and solve
the correction equation by an inner-outer decomposition.
The remaining parameters are then selected through the reduced energy.

Our main result is the following.

\begin{theorem}
\label{thm:main}
For every $\mu>0$, there exists an even integer $K_0=K_0(\mu)\geq2$
such that, for every even integer $K\geq K_0$, problem~\eqref{eq:problem}
admits a nonradial sign-changing solution
$u_K\in C^2(\overline{B_1})$. Writing
$x=(z,x_3)\in\mathbb{C}\times\mathbb{R}$, the solution can be chosen
to satisfy
\begin{equation*}
\begin{aligned}
u_K(e^{4\pi i/K}z,x_3)&=u_K(z,x_3),\\
u_K(e^{2\pi i/K}\overline z,x_3)&=u_K(z,x_3),\\
u_K(z,-x_3)&=u_K(z,x_3).
\end{aligned}
\end{equation*}
In particular, for every fixed $\mu>0$, problem~\eqref{eq:problem}
has infinitely many distinct nonradial sign-changing solutions.
\end{theorem}

For small $\mu$, this family coexists with the unique positive solution
$u\equiv\mu^{1/4}$. Theorem~\ref{thm:main} thus gives a nodal
counterpart in the three-dimensional ball to the positive multiplicity
results of Wang, Wei, and Yan~\cite{WangWeiYan2011}. The coefficient
$\mu$ stays fixed, while $K$ tends to infinity and the concentration
centers approach the equatorial circle from inside the ball.

In Section~\ref{sec:outline}, we describe the construction and the main
steps of the proof.
Section~\ref{sec:correction} treats the linear
problem and the nonlinear correction, and the proof of Theorem~\ref{thm:main} is
completed in Section~\ref{sec:existence}. The appendices contain the
auxiliary interaction estimates, reduced energy and trigonometric sums.
\section{Outline of the proof}
\label{sec:outline}

Our construction follows closely the approach of Sun, Wei, and
Yang~\cite{SunWeiYangFull} to the Brezis--Nirenberg problem in
dimension three. We use their framework based on entire nodal
profiles, Kelvin inversion near the nodal set, and an inner--outer
gluing argument combined with finite-dimensional reduction. The
present work adapts this scheme to the Neumann problem. Before we start, we declare the constant C is independent of $\delta$ and $K$.

We fix a nondegenerate solution $q$, even in its third coordinate, of
\begin{equation*}
\Delta q+|q|^4q=0\quad\text{in }\mathbb{R}^3,
\qquad q\in D^{1,2}(\mathbb{R}^3).
\end{equation*}
Ding~\cite{Ding1986} proved the existence of infinitely many entire
nodal solutions. We use a profile constructed by del Pino, Musso,
Pacard, and Pistoia~\cite{DelPinoMussoPacardPistoia2011}, with the
nondegeneracy property of Musso and Wei~\cite{MussoWei2015}. 

The integer involved in the construction of $q$ is chosen sufficiently
large and then kept fixed throughout the paper. We record the properties
of this profile that will be used later.

\begin{lemma}[Properties of the profile]
\label{lem:nodal-profile}
The profile $q$ can be chosen to have the following properties.
All constants below depend only on the fixed profile and on the order of differentiation.

\begin{enumerate}
\item \textup{(Symmetry and decay.)}
The function $q$ belongs to $C^\infty(\mathbb R^3)$, is even in $y_3$,
and satisfies
\begin{equation*}
q(y)=\frac1{|y|}q\!\left(\frac{y}{|y|^2}\right),
\qquad y\ne0.
\end{equation*}
For every nonnegative integer $\ell$,
\begin{equation*}
|\nabla^\ell q(y)|\leq C_\ell\langle y\rangle^{-1-\ell},
\qquad y\in\mathbb R^3,
\qquad \langle y\rangle=(1+|y|^2)^{1/2}.
\end{equation*}
\item \textup{(A regular planar nodal curve.)}
There exists a nonempty compact smooth embedded curve
\[
\Gamma\subset\{y\in\mathbb R^3:q(y)=0,\ y_3=0\},
\]
without boundary, such that
\begin{equation}
\label{eq:q-nodal-transversality}
0<c_\Gamma\leq |\nabla q(\xi)|\leq C_\Gamma,
\qquad \xi\in\Gamma.
\end{equation}

\item \textup{(Nondegeneracy in the even subspace.)}
Let $L_q=\Delta+5q^4$ and
\[
D^{1,2}_{\mathrm{even}}(\mathbb R^3)
=\{v\in D^{1,2}(\mathbb R^3):
v(y_1,y_2,y_3)=v(y_1,y_2,-y_3)\}.
\]
Define
\begin{equation*}
\begin{aligned}
z_0&=\tfrac12q+y\cdot\nabla q,
&z_3&=y_1\partial_2q-y_2\partial_1q,\\
z_1&=\partial_1q,
&z_4&=|y|^2\partial_1q-2y_1z_0,\\
z_2&=\partial_2q,
&z_5&=|y|^2\partial_2q-2y_2z_0.
\end{aligned}
\end{equation*}
These six functions are linearly independent and
\begin{equation*}
\{v\in D^{1,2}_{\mathrm{even}}(\mathbb R^3):L_qv=0\}
=\operatorname{span}\{z_0,\ldots,z_5\}.
\end{equation*}
They satisfy the estimate
\[
|z_j(y)|+\langle y\rangle|\nabla z_j(y)|
\leq C\langle y\rangle^{-1},
\qquad j=0,\ldots,5.
\]
\end{enumerate}
\end{lemma}
These properties follow from the results
in~\cite{DelPinoMussoPacardPistoia2011,MussoWei2015,SunWeiYangFull}.
Our objective is to construct
solutions with a large number of such bubbles. For this purpose, we need
to determine their locations, their concentration scale, and the small
changes in their profiles. We shall use the number of copies, $K$, as a large parameter to
determine them. Now the problem is where to put the bubbles. Following the approach of Sun, Wei, and Yang~\cite{SunWeiYangFull},
we first construct a solution in a fundamental sector and then
extend it to the whole ball by even reflections across the lateral
faces of the sector. We place the bubble inside the sector but near its spherical boundary.
As the bubble approaches the boundary, the distance to its reflected
image decreases and the boundary interaction becomes stronger.
In our Neumann construction, even reflections allow the leading
boundary contributions of the different copies to reinforce one
another, whereas the odd reflections in~\cite{SunWeiYangFull}
produce cancellations. To control this stronger boundary effect,
we keep the bubble farther from the spherical boundary and allow
its center to approach the boundary more slowly as $K\to\infty$ than theirs. The other parameters are only slightly adjusted to accommodate the change in distance. 

To be more precise,
let $K$ be a large even integer and set $\theta_0=\pi/K$. Let $G_K$ be the dihedral group of order $K$ generated by reflections
across the two flat faces of $\Sigma_K$. For
$x=(z,x_3)\in\mathbb{C}\times\mathbb{R}$, define
\[
R_\theta(z,x_3)=(e^{i\theta}z,x_3),
\qquad S_K(z,x_3)=(e^{2i\theta_0}\overline z,x_3).
\]
We work in the symmetry class
\begin{equation*}
\mathcal H_K=
\left\{u\in H^1(B_1):
\begin{aligned}
&u\circ R_{4\theta_0}=u,\quad u\circ S_K=u,\\
&u(z,-x_3)=u(z,x_3)
\end{aligned}
\right\}.
\end{equation*}
A fundamental sector for the planar symmetries is
\begin{equation*}
\Sigma_K=
\left\{(z,x_3)\in B_1:z\ne0,\ |\arg z|<\theta_0\right\}.
\end{equation*}
Reflection across its flat faces generates $K$ congruent sectors.
For a function $v$ on $\Sigma_K$, even in $x_3$, let $\mathcal E_Kv$
denote its extension by these reflections.

The center of the profile in $\Sigma_K$ is written as
\begin{equation*}
b=(\rho e^{i\alpha_b},0),\qquad
d=\frac{1-\rho^2}{2\rho},\qquad
\rho=\sqrt{1+d^2}-d.
\end{equation*}

For $\xi\in\Gamma$, define
\[
\nu_q(\xi):=\frac{\nabla q(\xi)}{|\nabla q(\xi)|}.
\]
This is well defined by \eqref{eq:q-nodal-transversality}. Choose $\theta_*(\xi)$ modulo $2\pi$ so that
$R_{\theta_*(\xi)}^T\nabla q(\xi)=|\nabla q(\xi)|e_1$, and put
$\beta=\theta_*(\xi)+\widehat\beta$. We use the six parameters
\[
A=(\eps,\xi,a,d,\alpha_b,\widehat\beta),
\qquad \eps>0,\quad \xi\in\Gamma,
\]
where $a$ is a normal displacement from the nodal curve and
$\widehat\beta$ is the orientation correction. Define
\begin{equation*}
Q_A(x)=\frac{\eps^{1/2}}{|x-b|}
q\!\left(\xi+a\nu_q(\xi)
+\frac{\eps R_\beta(x-b)}{|x-b|^2}\right).
\end{equation*}
The value at $x=b$ is understood by smooth continuation.
This is an entire solution of the critical equation. 
The sum
\begin{equation*}
Q_A^e(x)=
\sum_{j=0}^{K/2-1}
\left[Q_A(R_{4j\theta_0}x)
      +Q_A(R_{4j\theta_0}S_Kx)\right]
\end{equation*}
records the $K$ copies and their orientations. It has the symmetries
of $\mathcal H_K$ and satisfies the Neumann condition on the flat
faces of the sectors.

We now specify the parameter region. For a small fixed
$\delta\in(0,1)$, let $\mathcal A_K$ consist of the parameters satisfying
\begin{equation}
\label{eq:outline-parameters}
\begin{aligned}
&\eps K^3\in[\delta,\delta^{-1}],\quad
dK^{2/3}\in[\delta,\delta^{-1}],\quad \xi\in\Gamma, \quad|a|\leq\delta^{-2}K^{-7/3}(\log K)^{-1},\\
&|\alpha_b|\leq\delta^{-1/2}K^{-4/3}(\log K)^{-1/2},\qquad
|\widehat\beta|\leq\delta^{-1/2}K^{-1/3}(\log K)^{-1/2}.
\end{aligned}
\end{equation}
For fixed $\mu$, we choose $\delta$ sufficiently small,
and then take $K$ sufficiently large and even. All limits below
are taken in this order.

To construct the solution,
we first define the approximation solution $P Q_A$ as the unique weak solution of
\begin{equation}
\label{eq:outline-projection}
\begin{cases}
\Delta P Q_A-\mu P Q_A+Q_A^5=0 & \text{in }\Sigma_K,\\
\partial_\nu P Q_A=0 & \text{on }\partial\Sigma_K.
\end{cases}
\end{equation}
The boundary condition is imposed on both the spherical and flat
parts of $\partial\Sigma_K$. The projection estimates proved in
Appendix~\ref{app:interactions} give
\begin{equation}
\label{eq:outline-projection-error}
\|P Q_A-Q_A\|_{L^\infty(\Sigma_K)}
\leq C\delta^{-1/2}\eps^{3/2}K^2.
\end{equation}
We first construct a solution in the fundamental sector.

\begin{theorem}
\label{thm:outline-construction}
For every fixed $\mu>0$, there exist $\delta\in(0,1)$ and
$K_0\in\mathbb N$ such that, for every even $K\geq K_0$, one can
find $A_K\in\operatorname{int}\mathcal A_K$ and a sign-changing
solution $v_K\in H^1(\Sigma_K)\cap L^\infty(\Sigma_K)$ of
\begin{equation*}
\begin{cases}
\Delta v_K-\mu v_K+v_K^5=0 & \text{in }\Sigma_K,\\
\partial_\nu v_K=0 & \text{on }\partial\Sigma_K.
\end{cases}
\end{equation*}
The solution is even in $x_3$ and has the form
\begin{equation*}
v_K=P Q_{A_K}+\phi_{A_K},
\qquad
\eps_K^{1/2}
\|\phi_{A_K}\|_{L^\infty(\Sigma_K)}
\longrightarrow0
\quad\text{as }K\to\infty.
\end{equation*}
\end{theorem}

The homogeneous Neumann condition on the flat faces allows us
to extend a sector solution to the whole ball by even reflection.

\begin{theorem}[Even reflection]
\label{thm:outline-reflection}
Let $K\geq2$ be even. If $v\in H^1(\Sigma_K)$ is a weak solution of
\begin{equation}
\label{eq:outline-sector-equation}
\begin{cases}
\Delta v-\mu v+v^5=0 & \text{in }\Sigma_K,\\
\partial_\nu v=0 & \text{on }\partial\Sigma_K,
\end{cases}
\end{equation}
then its even extension $u=\mathcal E_Kv$ is a weak solution of
\begin{equation}
\label{eq:outline-reflected-equation}
\begin{cases}
\Delta u-\mu u+u^5=0 & \text{in }B_1,\\
\partial_\nu u=0 & \text{on }\partial B_1.
\end{cases}
\end{equation}
If $v$ is even in $x_3$, then $u\in\mathcal H_K$.
\end{theorem}

\begin{proof}
Clearly,
$u=\mathcal E_Kv\in H^1(B_1)$.

For any $\zeta\in H^1(B_1)$, define
\[
\zeta_K(x)=\sum_{g\in G_K}\zeta(gx),
\qquad x\in\Sigma_K.
\]
Then $\zeta_K\in H^1(\Sigma_K)$. Changing variables in each
reflected sector and using the weak equation for $v$, we obtain
\[
\begin{aligned}
\int_{B_1}
\bigl(\nabla u\cdot\nabla\zeta+\mu u\zeta-u^5\zeta\bigr)\,dx
&=
\int_{\Sigma_K}
\bigl(\nabla v\cdot\nabla\zeta_K
      +\mu v\zeta_K-v^5\zeta_K\bigr)\,dx=0.
\end{aligned}
\]
By standard elliptic estimation, $u$ admits a representation as the classical solution of \eqref{eq:outline-reflected-equation}. The symmetry assertions
follow from the definition of $\mathcal E_K$.
\end{proof}

It therefore remains to prove Theorem~\ref{thm:outline-construction}.
The proof is based on the method of localized energy, as
in~\cite{WangWeiYan2011}. Associated with the Neumann problem
in $\Sigma_K$ is the energy functional
\begin{equation*}
J_{\Sigma_K}(v)=
\frac12\int_{\Sigma_K}(|\nabla v|^2+\mu v^2)\,\mathrm{d}x
-\frac16\int_{\Sigma_K}|v|^6\,\mathrm{d}x,
\qquad v\in H^1(\Sigma_K).
\end{equation*}
We look for solutions of the form
\[
v=P Q_A+\phi_A,
\]
where $\phi_A$ is small in suitable norms after rescaling.
For each fixed $A\in\mathcal A_K$, we first solve a nonlinear
projected problem for $\phi_A$. We then seek a critical point
of the reduced energy functional
\begin{equation*}
\Phi_K(A)=J_{\Sigma_K}(P Q_A+\phi_A)
\end{equation*}
in the interior of the configuration space $\mathcal A_K$.
At such a point, the coefficients of the kernel terms vanish,
and $P Q_A+\phi_A$ solves the Neumann problem in $\Sigma_K$.
Its even extension is then a solution on the whole ball,
with energy $K\Phi_K(A)$.

\section{Finite-dimensional reduction}
\label{sec:correction}

In this section, we carry out the finite-dimensional reduction. We first introduce the rescaled variables
\begin{equation*}
y=\frac{x}{\eps},\qquad b_\eps=\frac b\eps,\qquad
\Omega_\eps=\eps^{-1}\Sigma_K,\qquad
\bar\eps=\eps K,
\end{equation*}
and write
\begin{equation}
\label{eq:fdr-profiles}
\begin{aligned}
U(y)&=\eps^{1/2}Q_A(\eps y)
 =\frac1{|y-b_\eps|}
 q\!\left(\xi+a\nu_q(\xi)
       +\frac{R_\beta(y-b_\eps)}{|y-b_\eps|^2}\right),\\
V(y)&=\eps^{1/2}P Q_A(\eps y).
\end{aligned}
\end{equation}
For functions $v$ and $h$ on a domain $D\subset\mathbb R^3$
and a center $c\in\mathbb R^3$, define
\begin{align*}
\|v\|_{*,c}
&=\sup_{y\in D}\langle y-c\rangle |v(y)|
 +\sup_{y\in D}\langle y-c\rangle^2|\nabla v(y)|,
\\
\|h\|_{**,c}
&=\sup_{y\in D}\langle y-c\rangle^{3+2\sigma}|h(y)|,
\end{align*}
where $\langle z\rangle=(1+|z|^2)^{1/2}$, $\sigma$ is a fix constant in $(0,\frac12)$.
\label{subsec:gluing}
By the rescaling notations introduced before, \eqref{eq:outline-projection} becomes 
\begin{equation*}
\begin{cases}
\Delta V-\mu\eps^2 V+U^5=0 &\text{in }\Omega_\eps,\\
\partial_\nu V=0 &\text{on }\partial\Omega_\eps.
\end{cases}
\end{equation*}
The estimate~\eqref{eq:outline-projection-error} gives
\begin{equation}
\label{eq:fdr-projection-bound}
\|V-U\|_{L^\infty(\Omega_\eps)}
\leq C\delta^{-5/2}K^{-4}.
\end{equation}
We seek a correction $\phi$ such that $V+\phi$ solves the rescaled equation \eqref{eq:outline-sector-equation}. Set
\begin{equation}
\label{eq:fdr-operator-error}
\begin{aligned}
L\phi&=\Delta\phi-\mu\eps^2\phi+5V^4\phi,\quad E=V^5-U^5,\quad N(\phi)=(V+\phi)^5-V^5-5V^4\phi.
\end{aligned}
\end{equation}
The equation for $\phi$ is therefore
\begin{equation}
\label{eq:fdr-correction-equation}
L\phi=-E-N(\phi)\quad\text{in }\Omega_\eps,
\qquad \partial_\nu\phi=0\quad\text{on }\partial\Omega_\eps.
\end{equation}

Let $\eta_0\in C^\infty([0,\infty))$ satisfy $0\leq\eta_0\leq1$,
$\eta_0=1$ on $[0,1]$ and $\eta_0=0$ on $[2,\infty)$. Fix
$r\in(0,1/32)$ sufficiently small, independently of $\delta$, and define
\[
\eta(y)=\eta_0\!\left(\frac{\bar\eps|y-b_\eps|}{r}\right),
\qquad
\chi(y)=\eta_0\!\left(\frac{\bar\eps|y-b_\eps|}{2r}\right),
\qquad D_K=B_{4r/\bar\eps}(b_\eps).
\]
Moreover,
\begin{equation*}
\chi\eta=\eta,\qquad
|\nabla\eta|+|\nabla\chi|\leq C\bar\eps,\qquad
|\Delta\eta|+|\Delta\chi|\leq C\bar\eps^2.
\end{equation*}
We decompose the correction as
\begin{equation}
\label{eq:fdr-splitting}
\phi=\chi\phi_{\mathrm{in}}+\phi_{\mathrm{out}},
\end{equation}
where $\chi\phi_{\mathrm{in}}$ is extended by zero outside $D_K$.
It is sufficient to solve the coupled system
\begin{equation}
\label{eq:fdr-inner-unprojected}
\begin{cases}
L\phi_{\mathrm{in}}
 =-\eta\bigl[E+N(\chi\phi_{\mathrm{in}}+\phi_{\mathrm{out}})
                   +5V^4\phi_{\mathrm{out}}\bigr]
 &\text{in }D_K,\\
\phi_{\mathrm{in}}=0 &\text{on }\partial D_K,
\end{cases}
\end{equation}
and
\begin{equation}
\label{eq:fdr-outer}
\begin{cases}
\bigl(\Delta-\mu\eps^2+5(1-\eta)V^4\bigr)\phi_{\mathrm{out}}=-(1-\eta)\bigl[E+N(\chi\phi_{\mathrm{in}}+\phi_{\mathrm{out}})\bigr]
       -\mathcal C(\phi_{\mathrm{in}})
 &\text{in }\Omega_\eps,\\
\partial_\nu\phi_{\mathrm{out}}=0 &\text{on }\partial\Omega_\eps,
\end{cases}
\end{equation}
where
\[
\mathcal C(v)=2\nabla\chi\cdot\nabla v+(\Delta\chi)v.
\]

\subsection{The Outer Problem}
\label{subsec:outer}
We first deal with the outer equation \eqref{eq:fdr-outer}. The idea is basically to rewrite the equation into a fixed-point problem and apply the contraction mapping theorem. We first recall a lemma from existing literature.
\begin{lemma}
\label{lem:neumann-pointwise}
For $h\in L^\infty(\Omega_\eps)$, let
$\psi\in H^1(\Omega_\eps)$ be the unique weak solution of
\[
\begin{cases}
-\Delta\psi+\mu\eps^2\psi=h
    &\text{in }\Omega_\eps,\\
\partial_\nu\psi=0
    &\text{on }\partial\Omega_\eps.
\end{cases}
\]
Then $\psi$ has a continuous representative satisfying
\begin{equation}
\label{eq:neumann-pointwise}
\begin{aligned}
|\psi(y)|
&\leq C\sum_{g\in G_K}
       \int_{\Omega_\eps}
       \frac{|h(z)|}{|y-gz|}\,\mathrm{d}z
\end{aligned}
\end{equation}
\end{lemma}

\begin{proof}
Let $\widetilde h$ and $\widetilde\psi$ be the even
reflection extensions of $h$ and $\psi$ to
$B_{\eps^{-1}}$. The argument of Theorem~\ref{thm:outline-reflection} shows that
$\widetilde\psi$ is the weak
solution of
\[
\begin{cases}
-\Delta\widetilde\psi+\mu\eps^2\widetilde\psi
    =\widetilde h
    &\text{in }B_{\eps^{-1}},\\
\partial_\nu\widetilde\psi=0
    &\text{on }\partial B_{\eps^{-1}}.
\end{cases}
\]
Elliptic regularity gives a continuous representative.
By the Green function estimate in~\cite{WangWeiYan2011}, we obtain
\[
\begin{aligned}
|\psi(y)|
\leq C\int_{B_{\eps^{-1}}}
       \frac{|\widetilde h(z)|}{|y-z|}\,\mathrm{d}z=C\sum_{g\in G_K}
       \int_{\Omega_\eps}
       \frac{|h(z)|}{|y-gz|}\,\mathrm{d}z.
\end{aligned}
\]
\end{proof}

To describe the parameter dependence on a fixed outer domain,
we regard $\eps$ as a fixed reference scale and introduce
\[
\mathcal P=(P_0,\ldots,P_5)
=(\Lambda,\vartheta,\tau_1,\tau_2,c_1,c_2),
\qquad
\tau=(\tau_1,\tau_2,0),\quad c=(c_1,c_2,0),
\]
and the auxiliary family
\[
\mathcal U_{\mathcal P}(y)
=
\frac{\Lambda^{1/2}}{|y-c|}
q\!\left(
\tau+\frac{\Lambda R_\vartheta(y-c)}{|y-c|^2}
\right).
\]
For the current configuration $A$, we take
\[
\mathcal P_A
=(1,\beta,\tau_1,\tau_2,b_{\eps,1},b_{\eps,2}),
\qquad
\tau=\xi+a\nu_q(\xi),
\]
so that $\mathcal U_{\mathcal P_A}=U$.
The value at $y=c$ is understood by smooth continuation.
For $\mathcal P$ in a sufficiently small open neighborhood of
$\mathcal P_A$, let $V_{\mathcal P}$ be the unique weak solution of
\begin{equation}
\label{eq:fdr-parameter-projection}
\begin{cases}
\Delta V_{\mathcal P}-\mu\eps^2V_{\mathcal P}
+\mathcal U_{\mathcal P}^5=0
&\text{in }\Omega_\eps,\\
\partial_\nu V_{\mathcal P}=0
&\text{on }\partial\Omega_\eps.
\end{cases}
\end{equation}
Define
\begin{equation*}
\begin{aligned}
E_{\mathcal P}
=V_{\mathcal P}^5-\mathcal U_{\mathcal P}^5,\quad
N_{\mathcal P}(s)
=(V_{\mathcal P}+s)^5-V_{\mathcal P}^5
-5V_{\mathcal P}^4s,
\end{aligned}
\end{equation*}
and
\begin{equation*}
\begin{aligned}
\eta_{\mathcal P}(y)
&=\eta_0\!\left(\frac{\bar\eps|y-c|}{r}\right),
&\qquad
\chi_{\mathcal P}(y)
&=\eta_0\!\left(\frac{\bar\eps|y-c|}{2r}\right),\\
D_{K,\mathcal P}
&=B_{4r/\bar\eps}(c),
&
\mathcal C_{\mathcal P}(v)
&=2\nabla\chi_{\mathcal P}\cdot\nabla v
+(\Delta\chi_{\mathcal P})v.
\end{aligned}
\end{equation*}
For an input $v$ defined on
the whole domain $\Omega_\eps$, the corresponding outer
problem is
\begin{equation*}
\left\{
\begin{aligned}
(\Delta-\mu\eps^2)\psi
&+(1-\eta_{\mathcal P})
\bigl[E_{\mathcal P}
+N_{\mathcal P}(\chi_{\mathcal P}v+\psi)
+5V_{\mathcal P}^4\psi\bigr]+\mathcal C_{\mathcal P}(v)=0
&&\text{in }\Omega_\eps,\\
\partial_\nu\psi&=0
&&\text{on }\partial\Omega_\eps.
\end{aligned}
\right.
\end{equation*}
At $\mathcal P=\mathcal P_A$, we recover
\[
V_{\mathcal P_A}=V,\quad E_{\mathcal P_A}=E,\quad
N_{\mathcal P_A}=N,\quad
\eta_{\mathcal P_A}=\eta,\quad
\chi_{\mathcal P_A}=\chi,\quad
\mathcal C_{\mathcal P_A}=\mathcal C.
\]
All parameter derivatives are taken with $\eps$, $K$, and $r$
fixed, with the six components of $\mathcal P$ regarded as
independent variables, and are then evaluated at $\mathcal P_A$.
In the following estimates, we work at the fixed configuration
$\mathcal P_A$. We now introduce the following Banach spaces.
\[
\begin{aligned}
\mathcal X_A
&=\left\{
v\in C^1(\overline{D_K}):
v|_{\partial D_K}=0,\quad
v(y_1,y_2,y_3)=v(y_1,y_2,-y_3)
\right\},\\
\mathcal Y_A
&=\left\{
\psi\in L^{\infty}(\Omega_\eps):
\psi(y_1,y_2,y_3)=\psi(y_1,y_2,-y_3)
\right\},
\end{aligned}
\]
equipped with the norms $\|\cdot\|_{*,b_\eps}$ and
$\|\cdot\|_\infty$, respectively. Set
\[
\begin{aligned}
\mathcal V_A
&=\left\{v\in\mathcal X_A:
\|v\|_{*,b_\eps}\leq M_0\delta^{-4}K^{-11/3}\right\},\\
\mathcal B_A
&=\left\{\psi\in\mathcal Y_A:
\|\psi\|_\infty\leq C_1\delta^{-5}K^{-14/3}\right\}.
\end{aligned}
\]
where $M_0$ and $C_1$ are constants that will be determined later.
Let $T_\eps$ denote the inverse of $\Delta-\mu\eps^2$ with
homogeneous Neumann boundary conditions on $\Omega_\eps$.
For fixed $v\in\mathcal V_A$, define
\begin{equation}
\label{eq:fdr-outer-map}
T(\psi)=T_\eps\!\left[
-(1-\eta)\bigl(E+N(\chi v+\psi)+5V^4\psi\bigr)
-\mathcal C(v)
\right],
\qquad \psi\in\mathcal B_A.
\end{equation}
The outer problem~\eqref{eq:fdr-outer} is equivalent to
the fixed-point equation $\psi=T(\psi)$.

\begin{lemma}
\label{lem:outer-map-invariant}
For every fixed $M_0>0$, there exist
$C_1>0$, independent of $\delta$ and $K$, and $K_0(\delta)$, such
that for every even $K\geq K_0(\delta)$, every
$A\in\mathcal A_K$, and every $v\in\mathcal V_A$,
\[
T(\mathcal B_A)\subseteq\mathcal B_A.
\]
\end{lemma}

\begin{proof}
Fix $v\in\mathcal V_A$ and $\psi\in\mathcal B_A$.
The source defining $T(\psi)$ is bounded and even in $y_3$.
The Neumann inverse preserves this symmetry by uniqueness.

The weak maximum principle gives
\[
\|T_\eps f\|_\infty
\leq \mu^{-1}\eps^{-2}\|f\|_\infty.
\]
The profile decay and the projection estimate imply
\[
\|V-U\|_\infty\leq C\delta^{-5/2}K^{-4},
\qquad
|U|+|V|\leq C\bar\eps
\quad\text{on }\supp(1-\eta).
\]
Combining the parameter factors before estimating them,
we obtain, for $j=3,4$,
\[
\begin{aligned}
\bigl\|T_\eps\bigl[(1-\eta)(|U|+|V|)^j f\bigr]\bigr\|_\infty
&\leq C\eps^{j-2}K^j\|f\|_\infty\leq C\delta^{2-j}K^{6-2j}\|f\|_\infty.
\end{aligned}
\]
Since $|E|\leq C(|U|+|V|)^4|V-U|$, these estimates yield
\[
\begin{aligned}
\|T_\eps[(1-\eta)E]\|_\infty
\leq C\delta^{-9/2}K^{-6},\quad
\|T_\eps[5(1-\eta)V^4\psi]\|_\infty
\leq C\delta^{-2}K^{-2}\|\psi\|_\infty.
\end{aligned}
\]
To estimate $T_\eps(\mathcal C(v))$, we apply
Lemma~\ref{lem:neumann-pointwise} with $h=-\mathcal C(v)$.
Recall that
\[
\supp \mathcal C(v)\subset
\left\{z:\frac{2r}{\bar\eps}\leq|z-b_\eps|\leq\frac{4r}{\bar\eps}\right\}.
\]
On this annulus, the cutoff estimates and the definition
of $\|\cdot\|_{*,b_\eps}$ give
\[
\begin{aligned}
|\mathcal C(v)(z)|
&\leq 2|\nabla\chi(z)|\,|\nabla v(z)|+|\Delta\chi(z)|\,|v(z)|\leq C\bar\eps^3\|v\|_{*,b_\eps}.
\end{aligned}
\]
Therefore, we obtain
\begin{equation*}
\begin{aligned}
\|T_\eps(\mathcal C(v))\|_\infty
&\leq C\sup_{y\in\Omega_\eps}
       \sum_{g\in G_K}\int_{\Omega_\eps}
       \frac{|\mathcal C(v)(z)|}{|y-gz|}\,\mathrm dz\\
&\leq C\bar\eps^3\|v\|_{*,b_\eps}
       \sum_{g\in G_K}\sup_{y\in\Omega_\eps}
       \int_{B_{4r/\bar\eps}(g b_\eps)}
       \frac{\mathrm dz}{|y-z|}\\
&\leq C\delta^{-1}K^{-1}\|v\|_{*,b_\eps}.
\end{aligned}
\end{equation*}

On $\supp(1-\eta)$,
\[
\begin{aligned}
|\chi v+\psi|
&\leq C\delta^{-1}K^{-2}\|v\|_{*,b_\eps}+\|\psi\|_\infty
\end{aligned}
\]
Using
\[
|N(t)|\leq C\bigl(|V|^3|t|^2+|t|^5\bigr),
\]
the estimate with $j=3$, and the maximum principle, we obtain
\[
\begin{aligned}
\|T_\eps[(1-\eta)N(\chi v+\psi)]\|_\infty
&\leq C\delta^{-1}
 \bigl(\delta^{-1}K^{-2}\|v\|_{*,b_\eps}+\|\psi\|_\infty\bigr)^2+C\delta^{-2}K^6
\bigl(\delta^{-1}K^{-2}\|v\|_{*,b_\eps}+\|\psi\|_\infty\bigr)^5.
\end{aligned}
\]
We now choose $C_1$. Since $v\in\mathcal V_A$,
the estimates give
\[
\begin{aligned}
\|T_\eps[(1-\eta)E+\mathcal C(v)]\|_\infty
\leq C(M_0+1)\delta^{-5}K^{-14/3}.
\end{aligned}
\]
Fix
\[
C_1>\max\{1,M_0,4C(M_0+1)\},
\]
where $C$ is the constant in the preceding estimate.
Then
\[
\|T_\eps[(1-\eta)E+\mathcal C(v)]\|_\infty
\leq\frac{C_1}{4}\delta^{-5}K^{-14/3}.
\]

For $\psi\in\mathcal B_A$, we have
\[
\begin{aligned}
\delta^{-1}K^{-2}\|v\|_{*,b_\eps}+\|\psi\|_\infty
&\leq (M_0K^{-1}+C_1)\delta^{-5}K^{-14/3}\leq 2C_1\delta^{-5}K^{-14/3}.
\end{aligned}
\]
Consequently,
\[
\begin{aligned}
\|T_\eps[(1-\eta)N(\chi v+\psi)]\|_\infty
&\leq CC_1^2\delta^{-11}K^{-28/3}
       +CC_1^5\delta^{-27}K^{-52/3}.
\end{aligned}
\]
Combining the preceding estimates yields
\[
\begin{aligned}
\|T(\psi)\|_\infty
\leq\Bigl[
&\frac14+C\delta^{-2}K^{-2}
+CC_1\delta^{-6}K^{-14/3}+CC_1^4\delta^{-22}K^{-38/3}
\Bigr]C_1\delta^{-5}K^{-14/3}.
\end{aligned}
\]
With $C_1$ fixed, increase $K_0(\delta)$ so that
\[
C\bigl(
\delta^{-2}K^{-2}
+C_1\delta^{-6}K^{-14/3}
+C_1^4\delta^{-22}K^{-38/3}
\bigr)\leq\frac14.
\]
It follows that
\[
\|T(\psi)\|_\infty
\leq\frac{C_1}{2}\delta^{-5}K^{-14/3}.
\]
We conclude that
$T(\psi)\in\mathcal B_A$.
\end{proof}
\begin{lemma}
\label{lem:outer-map-contraction}
Let $C_1$ be chosen as in
Lemma~\ref{lem:outer-map-invariant}.
After increasing $K_0(\delta)$ if necessary, for every
even $K\geq K_0(\delta)$, every $A\in\mathcal A_K$,
and every fixed $v\in\mathcal V_A$, the map
$T:\mathcal B_A\to\mathcal B_A$ satisfies
\begin{equation}
\label{eq:fdr-outer-contraction}
\|T(\psi_1)-T(\psi_2)\|_\infty
\leq\frac12\|\psi_1-\psi_2\|_\infty,
\qquad \psi_1,\psi_2\in\mathcal B_A.
\end{equation}
\end{lemma}

\begin{proof}
The proof is almost the same as before, we display some important intermediate results. Let $\psi_1,\psi_2\in\mathcal B_A$ and put
$t_i=\chi v+\psi_i$, $i=1,2$.
Since $v$ is fixed, we have
\[
T(\psi_1)-T(\psi_2)
=-T_\eps\!\left[
(1-\eta)\bigl(
5V^4(\psi_1-\psi_2)+N(t_1)-N(t_2)
\bigr)\right].
\]
Easy to derive that
\[
\begin{aligned}
|N(t_1)-N(t_2)|
\leq C\bigl(
C_1\delta^{-5}K^{-14/3}|V|^3+C_1^4\delta^{-20}K^{-56/3}
\bigr)|\psi_1-\psi_2|.
\end{aligned}
\]
Using the maximum principle from
the proof of Lemma~\ref{lem:outer-map-invariant}, we obtain
\[
\begin{aligned}
\|T_\eps[(1-\eta)(N(t_1)-N(t_2))]\|_\infty\leq C\bigl(
C_1\delta^{-6}K^{-14/3}
+C_1^4\delta^{-22}K^{-38/3}
\bigr)\|\psi_1-\psi_2\|_\infty.
\end{aligned}
\]
Moreover,
\[
\|T_\eps[5(1-\eta)V^4(\psi_1-\psi_2)]\|_\infty
\leq C\delta^{-2}K^{-2}\|\psi_1-\psi_2\|_\infty.
\]
Consequently,
\[
\begin{aligned}
\|T(\psi_1)-T(\psi_2)\|_\infty
\leq C\bigl(
&\delta^{-2}K^{-2}
+C_1\delta^{-6}K^{-14/3}+C_1^4\delta^{-22}K^{-38/3}
\bigr)\|\psi_1-\psi_2\|_\infty.
\end{aligned}
\]
With $C_1$ fixed, increase $K_0(\delta)$ so that
\[
C\bigl(
\delta^{-2}K^{-2}
+C_1\delta^{-6}K^{-14/3}
+C_1^4\delta^{-22}K^{-38/3}
\bigr)\leq\frac12.
\]
This proves~\eqref{eq:fdr-outer-contraction}.
\end{proof}

\begin{theorem}[The outer problem]
\label{thm:fdr-outer}
Fix $M_0>0$ independently of $\delta$, and choose $C_1$
as in Lemma~\ref{lem:outer-map-invariant}.
For every even $K\geq K_0(\delta)$, every
$A\in\mathcal A_K$, and every $v\in\mathcal V_A$,
problem~\eqref{eq:fdr-outer} with $\phi_{\mathrm{in}}=v$
has a unique weak solution in $\mathcal B_A$,
denoted by $\phi_{\mathrm{out}}=\Psi_A(v)$.
Moreover,
\begin{equation}
\label{eq:fdr-outer-bound}
\begin{aligned}
\|\Psi_A(v)\|_\infty
&\leq C\bigl(
\delta^{-1}K^{-1}\|v\|_{*,b_\eps}
+\delta^{-9/2}K^{-6}
\bigr).
\end{aligned}
\end{equation}
For any $v_1,v_2\in\mathcal V_A$,
\begin{equation}
\label{eq:fdr-outer-lipschitz}
\|\Psi_A(v_1)-\Psi_A(v_2)\|_\infty
\leq C\delta^{-1}K^{-1}\|v_1-v_2\|_{*,b_\eps}.
\end{equation}
\end{theorem}

\begin{proof}
Lemma~\ref{lem:outer-map-invariant}, Lemma~\ref{lem:outer-map-contraction}, and the Banach fixed-point theorem yield the existence and the uniqueness of $\Psi_A(v)$. The estimates \eqref{eq:fdr-outer-bound} and \eqref{eq:fdr-outer-lipschitz} are basically identical to the previous calculations.
\end{proof}

We next record the differentiability of the outer solution with
respect to the inner function.

\begin{lemma}
\label{lem:fdr-outer-C1}
Under the assumptions of Theorem~\ref{thm:fdr-outer}, fix an
admissible configuration $A$. The map $v\mapsto\Psi_A(v)$
has a $C^1$ extension to an open neighborhood of each
$v_0\in\mathcal V_A$ in $\mathcal X_A$, with values in
$\mathcal Y_A$. Moreover, for every $v\in\mathcal V_A$
and $h\in\mathcal X_A$,
\begin{equation}
\label{eq:fdr-outer-v-derivative-bound}
\|D_v\Psi_A(v)[h]\|_\infty
\leq C\delta^{-1}K^{-1}\|h\|_{*,b_\eps},
\end{equation}
where $C$ is independent of $\delta$ and $K$.
\end{lemma}

\begin{proof}
We write $T(v,\psi)$ for the map in
\eqref{eq:fdr-outer-map}, displaying its dependence on both
$v$ and $\psi$. Define
\[
\mathcal F(v,\psi):=\psi-T(v,\psi).
\]
For fixed $A$, the maps $v\mapsto\chi v$ and
$v\mapsto\mathcal C(v)$, extended by zero, are bounded linear
maps from $\mathcal X_A$ to $L^\infty(\Omega_\eps)$.
Since the nonlinearity is a polynomial and $T_\eps$ is bounded
and linear, $T$ is $C^1$ on
$\mathcal X_A\times\mathcal Y_A$.
Fix $v_0\in\mathcal V_A$ and put
$\psi_0=\Psi_A(v_0)$. By the proof of
Lemma~\ref{lem:outer-map-invariant},
\[
\|\psi_0\|_\infty
\leq\frac{C_1}{2}\delta^{-5}K^{-14/3},
\]
so $\psi_0$ lies in the interior of $\mathcal B_A$.
Lemma~\ref{lem:outer-map-contraction} gives
\[
\|D_\psi T(v_0,\psi_0)\|_{\mathcal L(\mathcal Y_A)}
\leq\frac12.
\]
Hence
\[
D_\psi\mathcal F(v_0,\psi_0)
=I-D_\psi T(v_0,\psi_0)
\]
is invertible, with inverse norm at most $2$.
The implicit function theorem therefore gives a $C^1$ map
$v\mapsto\psi(v)$ on an open neighborhood of $v_0$ such that
\[
\mathcal F(v,\psi(v))=0,
\qquad \psi(v_0)=\psi_0.
\]
After shrinking this neighborhood, $\psi(v)$ remains in
$\mathcal B_A$. By uniqueness in
Theorem~\ref{thm:fdr-outer}, it agrees with $\Psi_A(v)$
whenever $v\in\mathcal V_A$ in this neighborhood.
Finally, the derivative bound follows from
\eqref{eq:fdr-outer-lipschitz}.
If $v$ is in the interior of $\mathcal V_A$, then for
every $h\in\mathcal X_A$ and sufficiently small nonzero $s$,
\[
\left\|
\frac{\Psi_A(v+sh)-\Psi_A(v)}{s}
\right\|_\infty
\leq C\delta^{-1}K^{-1}\|h\|_{*,b_\eps}.
\]
Letting $s\to0$ gives
\eqref{eq:fdr-outer-v-derivative-bound}.
For $v$ on the boundary of $\mathcal V_A$, the same estimate
follows by approximation from its interior and the continuity
of the derivative of the local $C^1$ extension.
\end{proof}

We now study the parameter dependence of the outer solution
on fixed spaces. Fix the reference configuration
$\mathcal P_A$, put $c^0=b_\eps$, and keep $\eps$, $K$,
and $r$ fixed. We use the Banach space
\begin{equation*}
\mathcal X_0
:=\left\{
v_0\in C^1(\overline{\Omega_\eps}):
v_0(y_1,y_2,y_3)=v_0(y_1,y_2,-y_3)
\right\},
\end{equation*}
equipped with the fixed norm
\begin{equation}
\label{eq:fdr-ambient-star-norm}
\|v_0\|_{*,c^0}
:=\sup_{y\in\Omega_\eps}\langle y-c^0\rangle|v_0(y)|
 +\sup_{y\in\Omega_\eps}\langle y-c^0\rangle^2
 |\nabla v_0(y)|.
\end{equation}
For $v_0\in\mathcal X_0$ and $\psi\in\mathcal Y_A$, define
\begin{equation}
\label{eq:fdr-joint-outer-map}
\begin{aligned}
T(\mathcal P,v_0,\psi)
=-T_\eps\Bigl[&(1-\eta_{\mathcal P})
\bigl(E_{\mathcal P}
+N_{\mathcal P}(\chi_{\mathcal P}v_0+\psi)
+5V_{\mathcal P}^4\psi\bigr)+\mathcal C_{\mathcal P}(v_0)\Bigr].
\end{aligned}
\end{equation}
In a parameter partial derivative of this map, $v_0(y)$
and $\psi(y)$ are held fixed on $\Omega_\eps$.
The outer solution $\Psi_{\mathcal P}(v_0)$ is the
fixed point of $T(\mathcal P,v_0,\cdot)$.
\begin{proposition}
\label{prop:fdr-outer-parameter-dependence}
Fix an admissible configuration $A$ and
$v_0\in\mathcal X_0$ with
$\|v_0\|_{*,c^0}\leq M_0\delta^{-4}K^{-11/3}$.
Near $(\mathcal P_A,v_0)$, the outer solution
$\Psi_{\mathcal P}(v_0)$ is a $C^1$ function of
$(\mathcal P,v_0)$ with values in $\mathcal Y_A$.
If $\phi_{\mathrm{in}}(\mathcal P)$ is a $C^1$ family in
$\mathcal X_0$ and
\[
\phi_{\mathrm{out}}(\mathcal P)
=\Psi_{\mathcal P}(\phi_{\mathrm{in}}(\mathcal P)),
\]
then, provided the inner input at $\mathcal P_A$ satisfies
the above star-norm bound, we have, at this point,
\begin{equation}
\label{eq:fdr-outer-total-parameter-bound}
\|\partial_{P_j}\phi_{\mathrm{out}}\|_\infty
\leq C\delta^{-1}K^{-1}
\|\partial_{P_j}\phi_{\mathrm{in}}\|_{*,c^0}
+C\delta^{-3}K^{-2}.
\end{equation}
\end{proposition}

\begin{proof}
We first check differentiability of the map in
\eqref{eq:fdr-joint-outer-map} on the fixed spaces
$\mathcal X_0$ and $\mathcal Y_A$.
By Lemma~\ref{lem:appA-parameter-projection}, the map
$\mathcal P\mapsto V_{\mathcal P}$ is $C^1$ into
$L^\infty(\Omega_\eps)$, as is
$\mathcal P\mapsto\mathcal U_{\mathcal P}$.
Clearly, $T$ is jointly $C^1$ in
$(\mathcal P,v_0,\psi)$.

Let $\psi_0=\Psi_{\mathcal P_A}(v_0)$.
As noted above, the proofs of
Lemmas~\ref{lem:outer-map-invariant} and
\ref{lem:outer-map-contraction} apply to the present ambient
inputs and give
\[
\|\psi_0\|_\infty
\leq\frac{C_1}{2}\delta^{-5}K^{-14/3},
\qquad
\|D_\psi T(\mathcal P_A,v_0,\psi_0)\|
_{\mathcal L(\mathcal Y_A)}\leq\frac12.
\]
Define
\[
\mathcal F(\mathcal P,v,\psi)
:=\psi-T(\mathcal P,v,\psi).
\]
Then $D_\psi\mathcal F=I-D_\psi T$ is invertible at
$(\mathcal P_A,v_0,\psi_0)$, with inverse norm at most $2$.
The implicit function theorem gives a $C^1$ solution map on
an open neighborhood of $(\mathcal P_A,v_0)$.
After shrinking the neighborhood, its values remain in
$\mathcal B_A$, and uniqueness identifies it with the
outer solution. Moreover,
\begin{equation}
\label{eq:fdr-outer-parameter-identity}
\partial_{P_j}\Psi_{\mathcal P}(v_0)
=(I-D_\psi T)^{-1}\partial_{P_j}T,
\end{equation}
where the operators are evaluated at
$(\mathcal P,v_0,\Psi_{\mathcal P}(v_0))$.

We next estimate $\partial_{P_j}T$ at $\mathcal P_A$,
keeping $v_0$ and $\psi\in\mathcal B_A$ fixed.
Set $t=\chi v_0+\psi$, we have
\[
\|v_0\|_{*,c^0}\leq C\delta^{-4}K^{-11/3},
\qquad
\|\psi\|_\infty
\leq C_1\delta^{-5}K^{-14/3},
\].
On $\operatorname{supp}(1-\eta)$, the preceding estimates
and Lemma~\ref{lem:appA-parameter-projection} give
\[
\begin{aligned}
|U|+|V|+|t|
\leq C\delta^{-1}K^{-2},\quad|\partial_{P_j}\mathcal U_{\mathcal P}|
\leq C\delta^{-1}K^{-2},\quad
|\partial_{P_j}V_{\mathcal P}|
\leq C\delta^{-1}K^{-1}.
\end{aligned}
\]
Only the center derivatives act on the cutoffs. Their
supports are contained in the corresponding transition
annuli, where
\[
|\partial_{P_j}\eta_{\mathcal P}|\leq C\bar\eps,
\qquad
|\partial_{P_j}(\chi_{\mathcal P}v_0)|
\leq C\bar\eps^2\|v_0\|_{*,c^0}.
\]
The identity
\[
\begin{aligned}
&E_{\mathcal P}
+N_{\mathcal P}(\chi_{\mathcal P}v_0+\psi)
+5V_{\mathcal P}^4\psi=(V_{\mathcal P}+\chi_{\mathcal P}v_0+\psi)^5
-\mathcal U_{\mathcal P}^5
-5V_{\mathcal P}^4\chi_{\mathcal P}v_0
\end{aligned}
\]
therefore yields
\[
\begin{aligned}
&\Bigl\|\partial_{P_j}\Bigl\{
(1-\eta_{\mathcal P})
\bigl[E_{\mathcal P}
+N_{\mathcal P}(\chi_{\mathcal P}v_0+\psi)
+5V_{\mathcal P}^4\psi\bigr]
\Bigr\}\Bigr\|_\infty\leq CK\bar\eps^5.
\end{aligned}
\]
The maximum principle bounds the contribution
of this source to $\partial_{P_j}T$ by
\(
C\delta^{-3}K^{-3}.
\)
For the $\mathcal C$ term, the derivatives in
$\Lambda$, $\vartheta$, and $\tau$ vanish when $v_0$ is
fixed. For $i=1,2$, the cutoff bounds give
\[
|\nabla\partial_{c_i}\chi_{\mathcal P}|
\leq C\bar\eps^2,
\qquad
|\Delta\partial_{c_i}\chi_{\mathcal P}|
\leq C\bar\eps^3.
\]
On their supports,
$|v_0|\leq C\bar\eps\|v_0\|_{*,c^0}$ and
$|\nabla v_0|\leq C\bar\eps^2\|v_0\|_{*,c^0}$.
Consequently, the maximum principle implies
\begin{equation*}
\begin{aligned}
\|T_\eps[\partial_{c_i}\mathcal C_{\mathcal P}(v_0)]\|_\infty
\leq C\eps^{-2}\bar\eps^4\|v_0\|_{*,c^0}
\leq C\delta^{-6}K^{-17/3}.
\end{aligned}
\end{equation*}
Combining the two source estimates, we find
\[
\|\partial_{P_j}T\|_\infty
\leq C\bigl(\delta^{-3}K^{-3}
+\delta^{-6}K^{-17/3}\bigr)
\leq C\delta^{-3}K^{-2},
\]
uniformly for $\psi\in\mathcal B_A$.
In particular, evaluating at
$\psi=\Psi_{\mathcal P_A}(v_0)$ and using
\eqref{eq:fdr-outer-parameter-identity}, we obtain
\begin{equation}
\label{eq:fdr-outer-parameter-bound}
\left\|
\left.\partial_{P_j}\Psi_{\mathcal P}(v_0)
\right|_{\mathcal P=\mathcal P_A}
\right\|_\infty
\leq C\delta^{-3}K^{-2}.
\end{equation}

The input derivative estimate follows from
Lemma~\ref{lem:fdr-outer-C1}.
Since its proof does not use a boundary condition on the
input, the same estimate applies to $\mathcal X_0$:
\begin{equation}
\label{eq:fdr-ambient-input-derivative-bound}
\|D_v\Psi_{\mathcal P_A}(v_0)[h]\|_\infty
\leq C\delta^{-1}K^{-1}\|h\|_{*,c^0},
\qquad h\in\mathcal X_0.
\end{equation}
Finally, for a $C^1$ family of inputs, the chain rule
on the fixed spaces reads
\begin{equation*}
\begin{aligned}
\partial_{P_j}\phi_{\mathrm{out}}
={}&(\partial_{P_j}\Psi_{\mathcal P})
(\phi_{\mathrm{in}})
+D_v\Psi_{\mathcal P}(\phi_{\mathrm{in}})
       [\partial_{P_j}\phi_{\mathrm{in}}].
\end{aligned}
\end{equation*}
Evaluating at $\mathcal P_A$ and applying
\eqref{eq:fdr-outer-parameter-bound} and
\eqref{eq:fdr-ambient-input-derivative-bound} proves
\eqref{eq:fdr-outer-total-parameter-bound}.
\end{proof}
\begin{remark}
It's worthwhile to mention that when the inner problem is written in centered coordinates,
we put
\[
\widetilde\phi_{\mathrm{in}}(\mathcal P,s)
=\phi_{\mathrm{in}}(\mathcal P,c+s).
\]
For a differentiable family with the required
regularity, its derivative at fixed $y$ is related to the
derivative at fixed $s$ by
\begin{equation*}
\begin{aligned}
\partial_{P_j}\phi_{\mathrm{in}}(\mathcal P,y)
={}(\partial_{P_j}\widetilde\phi_{\mathrm{in}})(\mathcal P,y-c)-\nabla_s\widetilde\phi_{\mathrm{in}}(\mathcal P,y-c)
       \cdot\partial_{P_j}c.
\end{aligned}
\end{equation*}
\end{remark}

\subsection{The Linear Problem}
\label{subsec:linear}
By Theorem~\ref{thm:fdr-outer}, for each
$\phi_{\mathrm{in}}\in\mathcal V_A$, the outer equation has
a unique solution
$\phi_{\mathrm{out}}=\Psi_A(\phi_{\mathrm{in}})$ in
$\mathcal B_A$.
Substituting this expression into
\eqref{eq:fdr-inner-unprojected}, we reduce the gluing problem
to the following problem on the inner ball:
\begin{equation}
\label{eq:fdr-reduced-inner}
\begin{cases}
L\phi_{\mathrm{in}}
=-\eta\bigl[E
+N(\chi\phi_{\mathrm{in}}+\Psi_A(\phi_{\mathrm{in}}))
+5V^4\Psi_A(\phi_{\mathrm{in}})\bigr]
&\text{in }D_K,\\
\phi_{\mathrm{in}}=0
&\text{on }\partial D_K.
\end{cases}
\end{equation}
Once a solution $\phi_{\mathrm{in}}\in\mathcal V_A$ is found,
the correction
\[
\phi=\chi\phi_{\mathrm{in}}+\Psi_A(\phi_{\mathrm{in}})
\]
solves \eqref{eq:fdr-correction-equation}.

To solve \eqref{eq:fdr-reduced-inner}, we first consider a
linear problem with a prescribed right-hand side $h$.
The linear operator is the one introduced in
\eqref{eq:fdr-operator-error}. To handle the kernel directions associated with the parameters,
define
\begin{equation}
\label{eq:fdr-kernel-functions}
Z_j(y)
:=\left.\partial_{P_j}\mathcal U_{\mathcal P}(y)
\right|_{\mathcal P=\mathcal P_A},
\qquad
\widehat Z_j(y):=\chi(y)Z_j(y),
\qquad j=0,\ldots,5,
\end{equation}
where $\mathcal U_{\mathcal P}$ and the coordinates $P_j$
are those introduced in the preceding subsection, and $\chi$
is the cutoff already used in \eqref{eq:fdr-splitting}.
Since
$\Delta\mathcal U_{\mathcal P}+\mathcal U_{\mathcal P}^5=0$
in $\mathbb R^3$, differentiation gives
\begin{equation*}
\Delta Z_j+5U^4Z_j=0
\quad\text{in }\mathbb R^3,
\qquad j=0,\ldots,5.
\end{equation*}
These functions are even in $y_3$. We impose orthogonality
to their cutoffs to obtain uniform estimates for the
projected linear problem.

More precisely, given $h$ even in $y_3$ with
$\|h\|_{**,b_\eps}<\infty$, we seek
$\varphi\in\mathcal X_A$ and coefficients
$\lambda_0,\ldots,\lambda_5\in\mathbb R$ satisfying
\begin{equation}
\label{eq:fdr-linear-problem}
\begin{cases}
\Delta\varphi-\mu\eps^2\varphi+5V^4\varphi
=h+\displaystyle\sum_{j=0}^5\lambda_jV^4\widehat Z_j
&\text{in }D_K,\\[3pt]
\varphi=0
&\text{on }\partial D_K,\\[3pt]
\displaystyle\int_{D_K}V^4\varphi\widehat Z_j\,dy=0
&j=0,\ldots,5.
\end{cases}
\end{equation}
\begin{theorem}
\label{thm:fdr-linear}
There exist a constant $C>0$, independent of $\delta$ and $K$,
and $K_0(\delta)$ such that, for every even
$K\geq K_0(\delta)$, every configuration $A$ satisfying
\eqref{eq:outline-parameters}, and every $h$ even in $y_3$
with $\|h\|_{**,b_\eps}<\infty$, problem
\eqref{eq:fdr-linear-problem} admits a unique weak solution
\[
(\varphi,\lambda_0,\ldots,\lambda_5)
\in\mathcal X_A\times\mathbb R^6.
\]
Moreover,
\begin{equation}
\label{eq:fdr-linear-estimate}
\|\varphi\|_{*,b_\eps}+\sum_{j=0}^5|\lambda_j|
\leq C\|h\|_{**,b_\eps}.
\end{equation}
\end{theorem}
\begin{proof}
The proof follows Steps~1 and~2 of the proof of
\cite[Proposition~6.1]{SunWeiYangFull}, with the projected
profile replaced by $V$ and the zeroth-order term by
$-\mu\eps^2$.
The perturbation estimates required in the a priori
argument follow from \eqref{eq:outline-parameters},
\eqref{eq:fdr-projection-bound}, and $0<\sigma<1/2$.
The same Fredholm argument then gives existence and
uniqueness, together with \eqref{eq:fdr-linear-estimate}.
The result for bounded right-hand sides follows by
approximation and elliptic regularity.
\end{proof}

We denote the resulting linear solution operator by
$\mathfrak{G}_A$, so that $\varphi=\mathfrak{G}_A(h)$ and
\[
\|\mathfrak{G}_A(h)\|_{*,b_\eps}\leq C\|h\|_{**,b_\eps}.
\]
\subsection{The Nonlinear Problem}
\label{subsec:nonlinear}
Before we start, we first record an estimate that
will be used later in this section.

\begin{lemma}
\label{lem:fdr-neumann-divergence}
Under \eqref{eq:outline-parameters}, let
$F$ is a map from $\Omega_\eps$ to $\mathbb{R}^3$ and belongs to $L^\infty(\Omega_\eps) $ satisfy
$\operatorname{supp}F\Subset\Omega_\eps$.
Then the problem
\[
\begin{cases}
(\Delta-\mu\eps^2)w=\operatorname{div}F
&\text{in }\Omega_\eps,\\
\partial_\nu w=0
&\text{on }\partial\Omega_\eps
\end{cases}
\]
has a unique weak solution
$w\in H^1(\Omega_\eps)\cap L^\infty(\Omega_\eps)$.
Writing $w=T_\eps[\operatorname{div}F]$, we have
\begin{equation}
\label{eq:fdr-neumann-divergence-inverse}
\|T_\eps[\operatorname{div}F]\|_\infty
\leq C\eps^{-1}\|F\|_\infty,
\end{equation}
\end{lemma}

\begin{proof}
Existence and uniqueness in $H^1(\Omega_\eps)$ follow
from the Lax--Milgram theorem.
Extend $w$ and $F$ to $B_{\eps^{-1}}$ by
$\widetilde w(gx)=w(x)$ and $\widetilde F(gx)=gF(x)$
for $x\in\Omega_\eps$ and $g\in G_K$. The reflection argument in
Theorem~\ref{thm:outline-reflection} shows that
\[
\begin{cases}
(\Delta-\mu\eps^2)\widetilde w
=\operatorname{div}\widetilde F
&\text{in }B_{\eps^{-1}},\\
\partial_\nu\widetilde w=0
&\text{on }\partial B_{\eps^{-1}}
\end{cases}
\]
in the weak sense. Moreover,
$\|\widetilde F\|_\infty=\|F\|_\infty$.

Set
\[
\widehat w(x)=\widetilde w(x/\eps),
\qquad
\widehat F(x)=\widetilde F(x/\eps),
\qquad x\in B_1.
\]
Then
\[
\begin{cases}
(\Delta-\mu)\widehat w
=\eps^{-1}\operatorname{div}\widehat F
&\text{in }B_1,\\
\partial_\nu\widehat w=0
&\text{on }\partial B_1.
\end{cases}
\]
A Diorgi iteration argument shows
give
\[
\begin{aligned}
\|w\|_{L^\infty(\Omega_\eps)}
=\|\widehat w\|_{L^\infty(B_1)}
\leq C\eps^{-1}\|\widehat F\|_{L^4(B_1)}
\leq C\eps^{-1}\|F\|_{L^\infty(\Omega_\eps)}.
\end{aligned}
\]
This proves \eqref{eq:fdr-neumann-divergence-inverse}.
\end{proof}
We now use Theorem~\ref{thm:fdr-linear} to solve the
projected nonlinear inner problem.
For $v\in\mathcal V_A$, set
\begin{equation*}
\overline N_A(v)
:=N\bigl(\chi v+\Psi_A(v)\bigr)+5V^4\Psi_A(v).
\end{equation*}
We seek $\phi_{\mathrm{in}}\in\mathcal V_A$ and coefficients
$\lambda_0,\ldots,\lambda_5\in\mathbb R$ satisfying
\begin{equation}
\label{eq:fdr-nonlinear-problem}
\begin{cases}
L\phi_{\mathrm{in}}
=-\eta\bigl[E+\overline N_A(\phi_{\mathrm{in}})\bigr]
+\displaystyle\sum_{j=0}^5\lambda_jV^4\widehat Z_j
&\text{in }D_K,\\[3pt]
\phi_{\mathrm{in}}=0
&\text{on }\partial D_K,\\[3pt]
\displaystyle\int_{D_K}V^4\phi_{\mathrm{in}}\widehat Z_j\,dy=0
&j=0,\ldots,5.
\end{cases}
\end{equation}
By the definition of $\mathfrak{G}_A$, this problem is
equivalent to the fixed-point equation
\begin{equation*}
\phi_{\mathrm{in}}
=-\mathfrak{G}_A\!\left(
\eta\bigl[E+\overline N_A(\phi_{\mathrm{in}})\bigr]
\right).
\end{equation*}

\begin{theorem}
\label{thm:fdr-nonlinear}
Choose $M_0$ sufficiently large,
independently of $\delta$ and $K$.
There exist a constant $C>0$, independent of $\delta$ and
$K$, and $K_0(\delta)$ such that, for every even
$K\geq K_0(\delta)$ and every configuration $A$ satisfying
\eqref{eq:outline-parameters}, problem
\eqref{eq:fdr-nonlinear-problem} admits a unique solution
\[
(\phi_{\mathrm{in}},\lambda_0,\ldots,\lambda_5)
\in\mathcal V_A\times\mathbb R^6.
\]
Moreover,
\begin{equation}
\label{eq:fdr-nonlinear-estimate}
\|\phi_{\mathrm{in}}\|_{*,b_\eps}
+\sum_{j=0}^5|\lambda_j|
\leq C\delta^{-4}K^{-11/3}.
\end{equation}
\end{theorem}

\begin{proof}
We follow the contraction mapping argument for
\cite[Proposition~6.2]{SunWeiYangFull}.
In the present setting, the required estimates follow from
\eqref{eq:fdr-projection-bound},
\eqref{eq:fdr-outer-bound}, and
\eqref{eq:fdr-outer-lipschitz}.
The Banach fixed-point theorem therefore yields the unique
solution $\phi_{\mathrm{in}}\in\mathcal V_A$.
The corresponding coefficients $\lambda_j$ are uniquely
determined by the linear problem, and
\eqref{eq:fdr-linear-estimate} gives
\eqref{eq:fdr-nonlinear-estimate}.
\end{proof}

We next study the parameter dependence of the inner solution.
As in the outer problem, set $s = y-c$, we develop the following notations,
\[
\widetilde\phi_{\mathrm{in}}(\mathcal P,s)
=\phi_{\mathrm{in}}(\mathcal P,c+s),\qquad
\widetilde U_{\mathcal P}(s)=\mathcal U_{\mathcal P}(c+s),
\qquad
\widetilde V_{\mathcal P}(s)=V_{\mathcal P}(c+s).
\]
The cutoffs become the fixed functions
\[
\eta_*(s)=\eta_0\!\left(\frac{\bar\eps|s|}{r}\right),
\qquad
\chi_*(s)=\eta_0\!\left(\frac{\bar\eps|s|}{2r}\right).
\]
For $\mathcal P$ near $\mathcal P_A$, define
\[
Z_{\ell,\mathcal P}(y)
:=\partial_{P_\ell}\mathcal U_{\mathcal P}(y),
\qquad
\widetilde Z_{\ell,\mathcal P}(s)
:=\chi_*(s)Z_{\ell,\mathcal P}(c+s),
\qquad \ell=0,\ldots,5.
\]
At the reference configuration, we write
\[
\widetilde Z_\ell(s)
:=\widetilde Z_{\ell,\mathcal P_A}(s)
=\widehat Z_\ell(b_\eps+s),
\qquad \ell=0,\ldots,5,
\]
where $\widehat Z_\ell$ is defined in
\eqref{eq:fdr-kernel-functions}.

Parameter derivatives of the centered quantities below
are taken at fixed $s$.

\begin{proposition}
\label{prop:fdr-inner-parameter-dependence}
Under the assumptions of Theorem~\ref{thm:fdr-nonlinear},
the inner solution and its coefficients extend locally
near $\mathcal P_A$ to a $C^1$ family
\[
\mathcal P\longmapsto
\bigl(\widetilde\phi_{\mathrm{in}}(\mathcal P),
\lambda_0(\mathcal P),\ldots,\lambda_5(\mathcal P)\bigr)
\]
with values in the fixed inner space on
$B_{4r/\bar\eps}(0)$ and $\mathbb R^6$.
At $\mathcal P=\mathcal P_A$,
\begin{equation}
\label{eq:fdr-inner-parameter-bound}
\|\partial_{P_j}\widetilde\phi_{\mathrm{in}}\|_{*,0}
+\sum_{\ell=0}^5|\partial_{P_j}\lambda_\ell|
\leq C\delta^{-1}K^{-1},
\qquad j=0,\ldots,5.
\end{equation}
\end{proposition}
\begin{proof}
We first establish differentiability by applying the
implicit function theorem on fixed Banach spaces.
We write the inner equation and its six orthogonality
conditions as a single equation
$\mathcal F(v,\boldsymbol\lambda,\mathcal P)=0$
through the linear inverse. We then prove that
$\mathcal F$ is jointly $C^1$ and that its derivative
in $(v,\boldsymbol\lambda)$ is invertible at the
constructed solution.

Let $\widetilde{\mathcal X}_A$ be the space of
$C^1$ functions on $\overline{B_{4r/\bar\eps}(0)}$,
even in $s_3$ and with zero boundary values, equipped
with $\|\cdot\|_{*,0}$.
Write $\boldsymbol\lambda=(\lambda_0,\ldots,\lambda_5)$
and use the product norm
\[
\|(v,\boldsymbol\lambda)\|
_{\widetilde{\mathcal X}_A\times\mathbb R^6}
:=\|v\|_{*,0}+\sum_{\ell=0}^5|\lambda_\ell|.
\]
We relate the centered inner inputs to the outer space
$\mathcal X_0$. For $v\in\widetilde{\mathcal X}_A$,
choose a fixed bounded linear extension
$v^{\mathrm{ext}}\in C^1(\mathbb R^3)$, preserving the
symmetry in $s_3$, such that
\begin{equation*}
\begin{aligned}
v^{\mathrm{ext}}&=v
\quad\text{on }B_{4r/\bar\eps}(0),\\
\sup_{s\in\mathbb R^3}\langle s\rangle
|v^{\mathrm{ext}}(s)|
&+\sup_{s\in\mathbb R^3}\langle s\rangle^2
|\nabla v^{\mathrm{ext}}(s)|
\leq C\|v\|_{*,0}.
\end{aligned}
\end{equation*}
For $|c-c^0|\leq1$, the inequality
$\langle y-c^0\rangle\leq2\langle y-c\rangle$
shows that the translated extension, restricted to
$\Omega_\eps$, belongs to $\mathcal X_0$ and satisfies
\[
\|v^{\mathrm{ext}}(\,\cdot-c)\|_{*,c^0}
\leq C\|v\|_{*,0}.
\]
The norm on the left is the fixed outer input norm
in \eqref{eq:fdr-ambient-star-norm}.
Clearly, the outer solution is
independent of the choice of extension.
Define
\[
\widetilde\Psi_{\mathcal P}(v)(s)
:=\Psi_{\mathcal P}
\bigl(v^{\mathrm{ext}}(\,\cdot-c)\bigr)(c+s).
\]
Thus the actual inner and outer solutions satisfy
\[
\phi_{\mathrm{out}}(\mathcal P,c+s)
=\widetilde\Psi_{\mathcal P}
\bigl(\widetilde\phi_{\mathrm{in}}(\mathcal P)\bigr)(s).
\]
In a parameter partial derivative of
$\widetilde\Psi_{\mathcal P}(v)$, the centered input
$v(s)$ is held fixed; both its translated extension
and the evaluation point $c+s$ vary with the center.

Keep the coefficients $\widetilde V$ and
$\widetilde Z_\ell$ fixed at $\mathcal P_A$.
For $\|f\|_{**,0}<\infty$ and
$\boldsymbol m=(m_0,\ldots,m_5)\in\mathbb R^6$,
define $\mathfrak G_A^\sharp(f,\boldsymbol m)
=(w,\boldsymbol\nu)$ by
\[
\begin{cases}
(\Delta_s-\mu\eps^2+5\widetilde V^4)w
=f+\displaystyle\sum_{\ell=0}^5\nu_\ell
\widetilde V^4\widetilde Z_\ell
&\text{in }B_{4r/\bar\eps}(0),\\
w=0
&\text{on }\partial B_{4r/\bar\eps}(0),\\
\displaystyle\int_{B_{4r/\bar\eps}(0)}
\widetilde V^4\widetilde Z_\ell w\,ds=m_\ell
&\ell=0,\ldots,5,
\end{cases}
\]
where $f$ and $w$ are even in $s_3$.
Theorem~\ref{thm:fdr-linear} gives this inverse for
$\boldsymbol m=0$.
For general $\boldsymbol m$, subtracting a suitable linear combination reduces
the problem to zero moments.
Theorem~\ref{thm:fdr-linear} then gives
\begin{equation}
\label{eq:fdr-inner-augmented-inverse}
\|\mathfrak G_A^\sharp(f,\boldsymbol m)\|
_{\widetilde{\mathcal X}_A\times\mathbb R^6}
\leq C\left(\|f\|_{**,0}
+\sum_{\ell=0}^5|m_\ell|\right).
\end{equation}

For parameters and inputs near the reference solution,
this outer solution is defined by the same contraction
argument as in Theorem~\ref{thm:fdr-outer}.
We define
\[
\begin{aligned}
\mathcal F(v,\boldsymbol\lambda,\mathcal P)
:={}&
\begin{pmatrix}v\\ \boldsymbol\lambda\end{pmatrix}
+\mathfrak G_A^\sharp
\begin{pmatrix}
\begin{aligned}
&5(\widetilde V_{\mathcal P}^4-\widetilde V^4)v+\eta_*\Bigl[
(\widetilde V_{\mathcal P}+\chi_*v
+\widetilde\Psi_{\mathcal P}(v))^5
-\widetilde U_{\mathcal P}^5
-5\widetilde V_{\mathcal P}^4\chi_*v\Bigr]\\
&-\displaystyle\sum_{\ell=0}^5\lambda_\ell
\bigl(\widetilde V_{\mathcal P}^4
\widetilde Z_{\ell,\mathcal P}
-\widetilde V^4\widetilde Z_\ell\bigr)
\end{aligned}\\[6pt]
\left(
\displaystyle\int_{B_{4r/\bar\eps}(0)}
\bigl(\widetilde V_{\mathcal P}^4
\widetilde Z_{\ell,\mathcal P}
-\widetilde V^4\widetilde Z_\ell\bigr)v\,ds
\right)_{\ell=0}^5
\end{pmatrix}.
\end{aligned}
\]
In particular, Theorem~\ref{thm:fdr-nonlinear} gives
\[
\mathcal F\bigl(
\widetilde\phi_{\mathrm{in}}(\mathcal P_A),
\boldsymbol\lambda(\mathcal P_A),\mathcal P_A\bigr)=0.
\]

We next prove that $\mathcal F$ is jointly $C^1$.
We first justify the parameter dependence of the outer
solution for $C^1$ inputs. Define the centered term, extended by zero
outside the inner ball, by
\[
\mathcal C_*(v)
:=2\nabla\chi_*\cdot\nabla v+(\Delta\chi_*)v.
\]
It is a continuous function with compact support for
$v\in\widetilde{\mathcal X}_A$, and
\[
\mathcal C_{\mathcal P}
\bigl(v^{\mathrm{ext}}(\,\cdot-c)\bigr)
=\mathcal C_*(v)(\,\cdot-c).
\]
For $i=1,2$ and $t\ne0$, the fundamental theorem
of calculus applied to test functions gives, in
the sense of distributions,
\[
\begin{aligned}
&\frac1t\Bigl[
\mathcal C_*(v)(\,\cdot-c-te_i)
-\mathcal C_*(v)(\,\cdot-c)\Bigr]=-\operatorname{div}\left(
e_i\int_0^1\mathcal C_*(v)
(\,\cdot-c-\theta te_i)\,d\theta\right).
\end{aligned}
\]
Here $e_i$ is the $i$th coordinate vector.
Since $\mathcal C_*(v)$ is uniformly continuous,
the average on the right converges uniformly to
$\mathcal C_*(v)(\,\cdot-c)$ as $t\to0$.
Applying $T_\eps$ and using
\eqref{eq:fdr-neumann-divergence-inverse}, we obtain
the derivative in $L^\infty$,
\[
\partial_{c_i}
T_\eps[\mathcal C_*(v)(\,\cdot-c)]
=-T_\eps\!\left[
\operatorname{div}\bigl(e_i\mathcal C_*(v)(\,\cdot-c)\bigr)
\right].
\]
The same estimate shows that this derivative is
continuous in $(c,v)$.
Moreover, for $h\in\widetilde{\mathcal X}_A$,
\[
\begin{aligned}
&\bigl\|T_\eps[
\mathcal C_*(h)(\,\cdot-c')
-\mathcal C_*(h)(\,\cdot-c)]\bigr\|_\infty\leq C\eps^{-1}|c'-c|
\|\mathcal C_*(h)\|_\infty.
\end{aligned}
\]
Since $\mathcal C_*$ is bounded and linear from
$\widetilde{\mathcal X}_A$ into $C^0$, this also
gives continuity in operator norm of the derivative
with respect to $v$.
Thus this term is jointly $C^1$
after applying $T_\eps$. The cutoff products and the polynomial terms of the
outer fixed-point map are jointly $C^1$ into
$L^\infty$, by the $C^1$ regularity of the input
and Lemma~\ref{lem:appA-parameter-projection}.
The estimates in the outer unknown
therefore give a local $C^1$ outer solution by
the implicit function theorem.
Consequently, $(\mathcal P,v)\mapsto
\widetilde\Psi_{\mathcal P}(v)$ is locally $C^1$
into $L^\infty(B_{4r/\bar\eps}(0))$.

The profile formula and Lemma~\ref{lem:appA-parameter-projection}
give $C^1$ dependence
of $\widetilde U_{\mathcal P}$,
$\widetilde V_{\mathcal P}$, and
$\widetilde Z_{\ell,\mathcal P}$ on the fixed inner ball. Multiplication and the six moment integrals are
bounded multilinear maps in these spaces.
Since $\mathfrak G_A^\sharp$ is bounded and linear,
the formula for $\mathcal F$ now proves that it is
jointly $C^1$ on an open neighborhood of the
reference solution.

We next verify invertibility of the derivative in
the unknowns. At $\mathcal P=\mathcal P_A$, the
coefficient differences in the definition of
$\mathcal F$ vanish. Differentiation in
$(v,\boldsymbol\lambda)$ therefore gives
\[
\begin{aligned}
D_{(v,\boldsymbol\lambda)}\mathcal F
[h,\boldsymbol\nu]=
\begin{pmatrix}h\\ \boldsymbol\nu\end{pmatrix}
+\mathfrak G_A^\sharp
\begin{pmatrix}
\begin{aligned}
&5\eta_*\Bigl[
(\widetilde V+\chi_*v
+\widetilde\Psi_{\mathcal P_A}(v))^4
-\widetilde V^4\Bigr]\chi_*h\\
&+5\eta_*(\widetilde V+\chi_*v
+\widetilde\Psi_{\mathcal P_A}(v))^4
D_v\widetilde\Psi_{\mathcal P_A}(v)[h]
\end{aligned}\\[4pt]
0
\end{pmatrix}.
\end{aligned}
\]
At $v=\widetilde\phi_{\mathrm{in}}(\mathcal P_A)$,
Lemma~\ref{lem:fdr-outer-C1} gives
\[
\|D_v\widetilde\Psi_{\mathcal P_A}(v)[h]\|_\infty
\leq C\delta^{-1}K^{-1}\|h\|_{*,0}.
\]
The preceding pointwise bounds show that the first
source term in the derivative is bounded by
\[
C\bigl(\delta^{-4}K^{-11/3}
+\delta^{-6}K^{-8/3}\bigr)
\langle s\rangle^{-5}\|h\|_{*,0},
\]
and the second by
\[
C\delta^{-1}K^{-1}
\langle s\rangle^{-4}\|h\|_{*,0}.
\]
Since $0<\sigma<1/2$, multiplication by the
double-star weight preserves both bounds.
Together with \eqref{eq:fdr-inner-augmented-inverse},
this yields
\[
\begin{aligned}
&\|D_{(v,\boldsymbol\lambda)}\mathcal F-I\|
_{\mathcal L(\widetilde{\mathcal X}_A\times\mathbb R^6)}
\leq
C\bigl(\delta^{-4}K^{-11/3}
+\delta^{-6}K^{-8/3}
+\delta^{-1}K^{-1}\bigr)
\leq\frac12
\end{aligned}
\]
Hence,
$D_{(v,\boldsymbol\lambda)}\mathcal F$ is invertible,
with inverse norm at most $2$.

We have verified both hypotheses of the implicit
function theorem at
$(\widetilde\phi_{\mathrm{in}}(\mathcal P_A),
\boldsymbol\lambda(\mathcal P_A),\mathcal P_A)$.
It follows that the inner solution and its six
coefficients form a local $C^1$ family in
$\widetilde{\mathcal X}_A\times\mathbb R^6$.
By uniqueness, this family agrees with the previously
constructed solution for nearby admissible parameters.
Differentiating $\mathcal F=0$, we obtain
\begin{equation}
\label{eq:fdr-inner-parameter-identity}
\begin{pmatrix}
\partial_{P_j}\widetilde\phi_{\mathrm{in}}\\
(\partial_{P_j}\lambda_\ell)_{\ell=0}^5
\end{pmatrix}
=-\bigl[D_{(v,\boldsymbol\lambda)}\mathcal F\bigr]^{-1}
\partial_{P_j}\mathcal F.
\end{equation}
We now estimate $\partial_{P_j}\mathcal F$ at
$\mathcal P_A$, keeping
$v=\widetilde\phi_{\mathrm{in}}(\mathcal P_A)$ and
$\boldsymbol\lambda=\boldsymbol\lambda(\mathcal P_A)$
fixed. All parameter derivatives below are evaluated
at $\mathcal P_A$.
The chain rule gives
\[
\begin{aligned}
\partial_{P_j}
(\widetilde V_{\mathcal P}-\widetilde U_{\mathcal P})(s)
={}&\bigl[\partial_{P_j}(V_{\mathcal P}
-\mathcal U_{\mathcal P})\bigr](c+s)+\nabla(V_{\mathcal P}-\mathcal U_{\mathcal P})(c+s)
\cdot\partial_{P_j}c.
\end{aligned}
\]
Lemmas~\ref{lem:appA-projection-error-derivative} and
\ref{lem:appA-projection-error-gradient} therefore imply
\[
\|\partial_{P_j}
(\widetilde V_{\mathcal P}-\widetilde U_{\mathcal P})
\|_\infty\leq C\eps K^2.
\]
Together with the profile formula, this gives
\[
|\partial_{P_j}\widetilde U_{\mathcal P}(s)|
\leq C\langle s\rangle^{-1},
\qquad
|\partial_{P_j}\widetilde V_{\mathcal P}(s)|
\leq C\bigl(\langle s\rangle^{-1}+\eps K^2\bigr).
\]
Using
\[
\begin{aligned}
\partial_{P_j}
(\widetilde V_{\mathcal P}^5-\widetilde U_{\mathcal P}^5)
={}&5\widetilde V^4
\partial_{P_j}(\widetilde V_{\mathcal P}
-\widetilde U_{\mathcal P})+5(\widetilde V^4-\widetilde U^4)
\partial_{P_j}\widetilde U_{\mathcal P},
\end{aligned}
\]
and \eqref{eq:fdr-projection-bound}, we obtain
\[
\left|\partial_{P_j}
(\widetilde V_{\mathcal P}^5-\widetilde U_{\mathcal P}^5)(s)
\right|
\leq C\eps K^2\langle s\rangle^{-4}.
\]
Since $0<\sigma<1/2$, it follows that
\begin{equation}
\label{eq:fdr-inner-error-parameter-bound}
\left\|\eta_*\partial_{P_j}
(\widetilde V_{\mathcal P}^5-\widetilde U_{\mathcal P}^5)
\right\|_{**,0}
\leq C\eps K^2.
\end{equation}

We next estimate $\partial_{P_j}\widetilde\Psi_{\mathcal P}(v)$
with the centered input $v$ fixed.
On the support of the derivatives of $\chi_*$,
the starred norm and the cutoff bounds give
\[
\|\mathcal C_*(v)\|_\infty
\leq C\bar\eps^3\|v\|_{*,0}.
\]
By Lemma~\ref{lem:fdr-neumann-divergence} and the
distributional differentiation formula established above,
\[
\left\|\partial_{c_i}
T_\eps[\mathcal C_*(v)(\,\cdot-c)]\right\|_\infty
\leq C\eps^{-1}\bar\eps^3\|v\|_{*,0},
\qquad i=1,2.
\]
The other parameter derivatives of this term vanish.

For the polynomial part of the outer fixed-point map,
keep $\psi\in\mathcal B_A$ fixed when differentiating.
On $\operatorname{supp}(1-\eta)$, the profile estimates
and Lemma~\ref{lem:appA-projection-error-derivative} give
\[
\begin{aligned}
|\mathcal U_{\mathcal P}|+|V_{\mathcal P}|
+|\chi_{\mathcal P}v^{\mathrm{ext}}(\,\cdot-c)+\psi|
&\leq C\bar\eps,\\
|\partial_{P_j}\mathcal U_{\mathcal P}|
+|\partial_{P_j}V_{\mathcal P}|
&\leq CK\bar\eps,\\
|\partial_{P_j}\eta_{\mathcal P}|
&\leq C\bar\eps,\\
\left|\partial_{P_j}
\bigl[\chi_{\mathcal P}v^{\mathrm{ext}}(\,\cdot-c)\bigr]\right|
&\leq C\bar\eps^2\|v\|_{*,0}.
\end{aligned}
\]
The cutoff derivatives vanish for the noncenter parameters.
Using the identity for the polynomial source in the
outer equation, we obtain
\[
\begin{aligned}
&\Bigl\|\partial_{P_j}\Bigl\{
(1-\eta_{\mathcal P})\Bigl[
E_{\mathcal P}
+N_{\mathcal P}\bigl(
\chi_{\mathcal P}v^{\mathrm{ext}}(\,\cdot-c)+\psi\bigr)
+5V_{\mathcal P}^4\psi\Bigr]\Bigr\}\Bigr\|_\infty\leq CK\bar\eps^5.
\end{aligned}
\]
The maximum principle and the inverse bound
$\|(I-D_\psi T)^{-1}\|\leq2$ from the outer
contraction argument therefore yield
\[
\begin{aligned}
\left\|\partial_{P_j}
\bigl[\Psi_{\mathcal P}
(v^{\mathrm{ext}}(\,\cdot-c))\bigr]\right\|_\infty
&\leq C\bigl(
\eps^{-2}K\bar\eps^5
+\eps^{-1}\bar\eps^3\|v\|_{*,0}\bigr)\leq C\bigl(
\delta^{-3}K^{-3}
+\delta^{-6}K^{-20/3}\bigr).
\end{aligned}
\]
To account for the evaluation at $c+s$, we also
estimate the spatial gradient of the outer solution.
The outer equation and the preceding pointwise bounds give
\[
\begin{aligned}
\|\Delta\phi_{\mathrm{out}}\|_\infty
\leq{}&C\Bigl(
\eps^2\delta^{-5}K^{-14/3}
+\bar\eps^4\delta^{-5/2}K^{-4}
+\bar\eps^3\delta^{-4}K^{-11/3}
+\bar\eps^4\delta^{-5}K^{-14/3}\Bigr)
\leq{}C\delta^{-5}K^{-14/3}.
\end{aligned}
\]
Since $B_2(y)\subset\Omega_\eps$ for every $y\in D_K$,
the interior gradient estimate implies
\[
\begin{aligned}
\|\nabla\phi_{\mathrm{out}}\|_{L^\infty(D_K)}
&\leq C\bigl(
\|\phi_{\mathrm{out}}\|_\infty
+\|\Delta\phi_{\mathrm{out}}\|_\infty\bigr)\leq C\delta^{-5}K^{-14/3}.
\end{aligned}
\]
By the definition of $\widetilde\Psi_{\mathcal P}$,
\[
\begin{aligned}
\partial_{P_j}\widetilde\Psi_{\mathcal P}(v)(s)
={}&\Bigl[\partial_{P_j}\bigl\{
\Psi_{\mathcal P}(v^{\mathrm{ext}}(\,\cdot-c))
\bigr\}\Bigr](b_\eps+s)+\nabla\phi_{\mathrm{out}}(\mathcal P_A,b_\eps+s)
\cdot\partial_{P_j}c.
\end{aligned}
\]
Consequently,
\begin{equation}
\label{eq:fdr-centered-outer-fixed-input-bound}
\begin{aligned}
\|\partial_{P_j}\widetilde\Psi_{\mathcal P}(v)\|_\infty
&\leq C\bigl(
\delta^{-3}K^{-3}
+\delta^{-6}K^{-20/3}
+\delta^{-5}K^{-14/3}\bigr)\leq C\delta^{-3}K^{-2}
\end{aligned}
\end{equation}

We now differentiate the nonlinear source in
$\mathcal F$, keeping $v$ fixed, and separate
$\partial_{P_j}(\widetilde V_{\mathcal P}^5
-\widetilde U_{\mathcal P}^5)$.
In the remaining terms, the contributions linear in
$\chi_*v$ multiplying
$\partial_{P_j}\widetilde V_{\mathcal P}$ cancel.
Since $\langle s\rangle\leq C\bar\eps^{-1}$ on the
inner ball, the remaining contributions involving this
derivative have double-star norm bounded by
\[
\begin{aligned}
&C\delta^{-5}K^{-14/3}
\bigl(1+\eps K^2\bar\eps^{-2\sigma}\bigr)
+C\bigl(\delta^{-4}K^{-11/3}
+\delta^{-6}K^{-8/3}\bigr)^2\leq C\eps K^2
\end{aligned}
\]
Together with \eqref{eq:fdr-inner-error-parameter-bound}
and the decay of $\widetilde V+\chi_*v
+\widetilde\Psi_{\mathcal P_A}(v)$, this yields
\begin{equation}
\label{eq:fdr-inner-nonlinearity-parameter-bound}
\begin{aligned}
\left\|\eta_*\partial_{P_j}\Bigl[
(\widetilde V_{\mathcal P}+\chi_*v
+\widetilde\Psi_{\mathcal P}(v))^5
-\widetilde U_{\mathcal P}^5
-5\widetilde V_{\mathcal P}^4\chi_*v
\Bigr]\right\|_{**,0}\leq C\eps K^2
+C\|\partial_{P_j}\widetilde\Psi_{\mathcal P}(v)\|_\infty.
\end{aligned}
\end{equation}

Finally, the profile formula and the projection
derivative estimate give
\[
\left|\partial_{P_j}
(\widetilde V_{\mathcal P}^4
\widetilde Z_{\ell,\mathcal P})(s)\right|
\leq C\bigl(\langle s\rangle^{-5}
+\eps K^2\langle s\rangle^{-4}\bigr).
\]
It follows that
\[
\begin{aligned}
\left\|20\widetilde V^3
(\partial_{P_j}\widetilde V_{\mathcal P})v
\right\|_{**,0}+\sum_{\ell=0}^5|\lambda_\ell|
\left\|\partial_{P_j}
(\widetilde V_{\mathcal P}^4
\widetilde Z_{\ell,\mathcal P})\right\|_{**,0}+\sum_{\ell=0}^5
\left|\int_{B_{4r/\bar\eps}(0)}
\partial_{P_j}
(\widetilde V_{\mathcal P}^4
\widetilde Z_{\ell,\mathcal P})v\,ds\right|
\leq C\delta^{-4}K^{-11/3}.
\end{aligned}
\]
where we use the integrability of
$\langle s\rangle^{-6}$ and $\langle s\rangle^{-5}$
in $\mathbb R^3$.
Since $\mathfrak G_A^\sharp$ is fixed,
\eqref{eq:fdr-inner-augmented-inverse},
\eqref{eq:fdr-centered-outer-fixed-input-bound}, and
\eqref{eq:fdr-inner-nonlinearity-parameter-bound} imply
\[
\begin{aligned}
\|\partial_{P_j}\mathcal F\|
_{\widetilde{\mathcal X}_A\times\mathbb R^6}
&\leq C\delta^{-1}K^{-1}.
\end{aligned}
\]
Combining this estimate with
\eqref{eq:fdr-inner-parameter-identity} and the inverse
bound gives
\[
\|\partial_{P_j}\widetilde\phi_{\mathrm{in}}\|_{*,0}
+\sum_{\ell=0}^5|\partial_{P_j}\lambda_\ell|
\leq C\delta^{-1}K^{-1},
\qquad j=0,\ldots,5.
\]
This proves \eqref{eq:fdr-inner-parameter-bound}.
\end{proof}

\section{Proof of the Main Result}
\label{sec:existence}

The construction in Section~3 reduces the problem to
choosing the parameters so that the six coefficients in
the projected inner equation vanish. More precisely,
Theorem~\ref{thm:fdr-nonlinear} gives, for each admissible
configuration $A$, an inner correction $\phi_{\mathrm{in}}$
and coefficients $\lambda_0(A),\ldots,\lambda_5(A)$ satisfying
\begin{equation*}
\begin{cases}
L\phi_{\mathrm{in}}
=-\eta\bigl[E+\overline N_A(\phi_{\mathrm{in}})\bigr]
+\displaystyle\sum_{j=0}^5\lambda_j(A)V^4\widehat Z_j
&\text{in }D_K,\\
\phi_{\mathrm{in}}=0
&\text{on }\partial D_K,\\
\displaystyle\int_{D_K}
V^4\phi_{\mathrm{in}}\widehat Z_j\,dy=0,
&j=0,\ldots,5.
\end{cases}
\end{equation*}
Set $\phi_{\mathrm{out}}=\Psi_A(\phi_{\mathrm{in}})$ and,
returning to the original variables, define
\begin{equation*}
\phi_A(x)
:=\eps^{-1/2}
\bigl(\chi\phi_{\mathrm{in}}+\phi_{\mathrm{out}}\bigr)
\left(\frac{x}{\eps}\right),
\qquad x\in\Sigma_K.
\end{equation*}
If
\begin{equation}
\label{eq:main-coefficient-equations}
\lambda_j(A)=0,\qquad j=0,\ldots,5,
\end{equation}
then the equations we derived before imply that
$PQ_A+\phi_A$ solves the Neumann problem on $\Sigma_K$. It therefore remains to find an admissible configuration
satisfying \eqref{eq:main-coefficient-equations}. To choose these parameters, we consider the reduced energy
introduced in Section~2,
\begin{equation*}
\Phi_K(A):=J_{\Sigma_K}(PQ_A+\phi_A).
\end{equation*}
The following lemma shows that an interior critical point
of $\Phi_K$ gives the required configuration.

\begin{lemma}
\label{lem:main-critical-point}
For sufficiently large even $K\geq K_0(\delta)$, let
$A_*\in\operatorname{int}\mathcal A_K$ be a critical point
of $\Phi_K$. Then
\(
\lambda_j(A_*)=0  \text{  for  } j=0,\ldots,5.
\)
\end{lemma}

\begin{proof}
The proof follows that of \cite[Lemma~7.1]{SunWeiYangFull},
using Lemma~\ref{lem:appA-projection-error-derivative}
and Propositions~\ref{prop:fdr-inner-parameter-dependence}
and~\ref{prop:fdr-outer-parameter-dependence}.
\end{proof}
We next show that the correction $\phi_A$ contributes only
a higher-order term to the reduced energy.
Recall that
\[
\phi(y)=\eps^{1/2}\phi_A(\eps y)
=\chi\phi_{\mathrm{in}}+\phi_{\mathrm{out}}.
\]
Expanding the energy and integrating by parts, we obtain
\begin{equation}
\label{eq:main-energy-correction-identity}
\begin{aligned}
J_{\Sigma_K}(PQ_A+\phi_A)-J_{\Sigma_K}(PQ_A)
={}&-\int_{\Omega_\eps}E\phi\,dy
-\frac12\int_{\Omega_\eps}\phi L\phi\,dy+O\!\left(\int_{\Omega_\eps}
\bigl(|V|^3|\phi|^3+|\phi|^6\bigr)\,dy\right).
\end{aligned}
\end{equation}
The inner and outer equations give
\[
L\phi=-E-N(\phi)
+\sum_{j=0}^5\lambda_j(A)\chi V^4\widehat Z_j.
\]
By Theorem~\ref{thm:fdr-outer},
\eqref{eq:fdr-nonlinear-estimate}, and
\eqref{eq:fdr-projection-bound}, we have
\[
\|\phi\|_\infty+\sum_{j=0}^5|\lambda_j(A)|
\leq C\delta^{-4}K^{-11/3}.
\]
The profile bounds also imply
$\int_{\Omega_\eps}(|U|^4+|V|^4)\,dy\leq C$.
Consequently,
\[
\begin{aligned}
\int_{\Omega_\eps}|E\phi|\,dy
\leq C\delta^{-13/2}K^{-23/3}, \quad
\sum_{j=0}^5|\lambda_j(A)|
\left|\int_{D_K}\chi V^4\widehat Z_j\phi\,dy\right|
\leq C\delta^{-8}K^{-22/3}.
\end{aligned}
\]
Moreover, Young's inequality and
$|\Omega_\eps|=4\pi/(3K\eps^3)$ yield
\[
\begin{aligned}
\int_{\Omega_\eps}
\bigl(|V|^3|\phi|^3+|\phi|^6\bigr)\,dy
\leq C\int_{\Omega_\eps}
\bigl(V^4\phi^2+|\phi|^6\bigr)\,dy
\leq C\delta^{-8}K^{-22/3}
\end{aligned}
\]
Since $|N(\phi)\phi|
\leq C(|V|^3|\phi|^3+|\phi|^6)$, substitution into
\eqref{eq:main-energy-correction-identity} gives
\begin{equation*}
\Phi_K(A)
=J_{\Sigma_K}(PQ_A)
+O\!\left(\delta^{-8}K^{-22/3}\right),
\end{equation*}
uniformly for $A\in\mathcal A_K$.
Combining this with
Proposition~\ref{prop:appB-projected-energy-expansion}, we conclude that
\begin{equation*}
\Phi_K(A)
=\frac13\int_{\mathbb R^3}q^6\,dx
+2\pi\Psi(A)
+\delta^{-8}o\!\left(\frac{K^{-20/3}}{\log K}\right),
\end{equation*}
where $\Psi(A)$ is given by \eqref{eq:appB-leading-energy}.

We now locate an interior maximum of the reduced energy
by comparing each boundary point with a suitable
admissible configuration. Before we start, first denote
\[
J_1:=\Phi_K(\eps,\xi,a,d,\alpha_b,\widehat\beta).
\] to simplify our notation.

\begin{theorem}
\label{thm:main-interior-maximum}
There exists a sufficiently small $\delta\in(0,1)$ such
that, for every sufficiently large even $K\geq K_0(\delta)$,
the maximum of $\Phi_K$ over the set $\mathcal A_K$
defined by \eqref{eq:outline-parameters} is attained at
an interior point. In particular, there exists
$A_K\in\operatorname{int}\mathcal A_K$ such that
\[
\Phi_K(A_K)
=\max_{A\in\mathcal A_K}J_{\Sigma_K}(PQ_A+\phi_A).
\]
\end{theorem}

\begin{proof}
We use the boundary comparison argument.
The set $\mathcal A_K$ is compact and $\Phi_K$ is
continuous, so its maximum is attained.
We shall show that no boundary point can be a maximizer.
Throughout the proof, $\delta$ is first chosen sufficiently
small and then $K\geq K_0(\delta)$ is taken sufficiently
large. All remainder estimates are uniform for
$A\in\mathcal A_K$. Put
\[
C_q:=\frac13\int_{\mathbb R^3}q^6\,dx,
\qquad
b_0(\xi):=\frac{\zeta(3)}{4\pi^3}|\nabla q(\xi)|^2,
\qquad
b_1(\xi):=\frac1{4\pi}|\nabla q(\xi)|^2.
\]
Since $q(\xi)=0$ on $\Gamma$, the integral defining
$C_*(\xi)$ is convergent and positive.
The profile properties and compactness of $\Gamma$ give
\[
0<c\leq C_*(\xi),\ b_0(\xi),\ b_1(\xi)\leq C,
\qquad \xi\in\Gamma.
\]
Recall that $\widehat\alpha_b=K\alpha_b$ and
$\alpha_w=-\widehat\beta+O(a)$.
The estimates in Appendix~C give
\[
S_1^+(K)+S_1^+(K,d)
=\frac{10}{3\pi}K\log K+O(\delta^{-1}K),
\]
and
\[
A_K^\gamma
=K^3
\begin{pmatrix}
\dfrac{\zeta(3)}{\pi^3}&0\\[4pt]
0&\dfrac{93\zeta(5)}{8\pi^5}
\end{pmatrix}
+O(K^2).
\]
Using these estimates, we obtain
\begin{equation}
\label{eq:main-boundary-energy-expansion}
\begin{aligned}
\Phi_K(A)
={}&C_q+2\pi\bigg[
\mu\eps^2C_*(\xi)-b_0(\xi)\eps^3K^3\\
&\quad-3b_0(\xi)\eps^3dK^3-b_1(\xi)\eps^3Kd^{-2}\\
&\quad-\frac5{3\pi}\eps K\log K
|\nabla q(\xi)|^2a^2
-\frac1\pi\eps^2Kd^{-1}|\nabla q(\xi)|^2a\\
&\quad-\frac{\eps^3K^3|\nabla q(\xi)|^2}{8}
\left(\frac{\zeta(3)}{\pi^3}\widehat\beta^2
+\frac{93\zeta(5)}{8\pi^5}\widehat\alpha_b^2\right)
\bigg]+\delta^{-8}o\!\left(\frac{K^{-20/3}}{\log K}\right).
\end{aligned}
\end{equation}

Since $\Gamma$ has no boundary, it remains to exclude the
boundary values of $\eps$, $a$, $d$, $\alpha_b$, and
$\widehat\beta$.

\medskip
\noindent
(1) Suppose that $\eps=\delta K^{-3}$ or
$\eps=\delta^{-1}K^{-3}$.
Consider
\[
\eps_*(\xi):=\frac{2\mu C_*(\xi)}{3b_0(\xi)}K^{-3},
\qquad
A^{(\eps)}:=(\eps_*,\xi,a,d,\alpha_b,\widehat\beta).
\]
By the positive uniform bounds above, we may choose
$\delta$ so small that $\eps_*(\xi)$ lies strictly
between the two endpoints for every $\xi\in\Gamma$.
By \eqref{eq:main-boundary-energy-expansion},
\[
\begin{aligned}
J_1(\delta^{-1}K^{-3})
&=C_q+2\pi K^{-6}
\bigl[\mu C_*(\xi)\delta^{-2}-b_0(\xi)\delta^{-3}\bigr]+O(\delta^{-8}K^{-20/3}),\\
J_1(\delta K^{-3})
&=C_q+2\pi K^{-6}
\bigl[\mu C_*(\xi)\delta^2-b_0(\xi)\delta^3\bigr]+O(\delta^{-8}K^{-20/3}).
\end{aligned}
\]
At the comparison point, we have
\[
J_1(\eps_*)
=C_q+2\pi K^{-6}
\frac{4\mu^3C_*(\xi)^3}{27b_0(\xi)^2}
+O(\delta^{-8}K^{-20/3}).
\]
Shrinking $\delta$ if necessary, the uniform positive
bounds on $C_*(\xi)$ and $b_0(\xi)$ ensure that
\[
\begin{aligned}
\mu C_*(\xi)\delta^{-2}-b_0(\xi)\delta^{-3}
\leq0, \quad
\mu C_*(\xi)\delta^2-b_0(\xi)\delta^3
\leq\frac{2\mu^3C_*(\xi)^3}{27b_0(\xi)^2},
\end{aligned}
\]
Consequently,
\[
\begin{aligned}
&J_1(\eps_*)-
\max\bigl\{J_1(\delta^{-1}K^{-3}),J_1(\delta K^{-3})\bigr\}\geq c_\eps K^{-6}-C\delta^{-8}K^{-20/3}>0
\end{aligned}
\]
where
$c_\eps>0$ is independent of $\delta$, $K$, and $\xi$.

\medskip
\noindent
(2) Suppose that
$a=\pm\delta^{-2}K^{-7/3}(\log K)^{-1}$.
Consider
\[
a_*:=-\frac{3\eps}{10d\log K},
\qquad
A^{(a)}:=(\eps,\xi,a_*,d,\alpha_b,\widehat\beta).
\]
The parameter bounds give
\[
|a_*|\leq\frac3{10}\delta^{-2}K^{-7/3}(\log K)^{-1},
\]
so $a_*$ lies strictly between the two endpoints.
By \eqref{eq:main-boundary-energy-expansion},
\[
\begin{aligned}
J_1\!\left(\pm\delta^{-2}K^{-7/3}(\log K)^{-1}\right)
={}&J_1(0)
-\frac{10\delta^{-4}\eps K^{-11/3}|\nabla q(\xi)|^2}
{3\log K}\\
&\mp\frac{2\delta^{-2}\eps^2d^{-1}K^{-4/3}|\nabla q(\xi)|^2}
{\log K}
+\delta^{-8}o\!\left(\frac{K^{-20/3}}{\log K}\right).
\end{aligned}
\]
At the comparison point, we have
\[
J_1(a_*)
=J_1(0)
+\frac{3\eps^3K|\nabla q(\xi)|^2}{10d^2\log K}
+\delta^{-8}o\!\left(\frac{K^{-20/3}}{\log K}\right).
\]
We obtain
\[
\begin{aligned}
&J_1(a_*)-
\max_{\pm}J_1\!\left(\pm\delta^{-2}K^{-7/3}(\log K)^{-1}\right)\\
&\qquad\geq
\frac{4\delta^{-4}\eps K^{-11/3}|\nabla q(\xi)|^2}{3\log K}
+\delta^{-8}o\!\left(\frac{K^{-20/3}}{\log K}\right)\\
&\qquad\geq c_a\delta^{-3}\frac{K^{-20/3}}{\log K}
+\delta^{-8}o\!\left(\frac{K^{-20/3}}{\log K}\right)>0
\end{aligned}
\]
where
$c_a>0$ is independent of $\delta$, $K$, and $\xi$.

\medskip
\noindent
(3) Suppose that $d=\delta K^{-2/3}$ or
$d=\delta^{-1}K^{-2/3}$.
Consider
\[
d_*:=\left(\frac{2b_1(\xi)}{3b_0(\xi)}\right)^{1/3}K^{-2/3},
\qquad
A^{(d)}:=(\eps,\xi,a,d_*,\alpha_b,\widehat\beta).
\]
By the positive uniform bounds above, we may choose
$\delta$ so small that $d_*$ lies strictly
between the two endpoints for every $\xi\in\Gamma$.
The parameter bounds give
\[
2\eps^2Kd^{-1}|\nabla q(\xi)|^2|a|
\leq C\delta^{-5}\frac{K^{-20/3}}{\log K}.
\]
Using the admissible value $d=K^{-2/3}$ as a reference,
\eqref{eq:main-boundary-energy-expansion} therefore yields
\[
\begin{aligned}
J_1(\delta^{-1}K^{-2/3})
={}&J_1(K^{-2/3})
-2\pi\eps^3K^{7/3}
\bigl[3b_0(\xi)(\delta^{-1}-1)+b_1(\xi)(\delta^2-1)\bigr]+O\!\left(\delta^{-8}\frac{K^{-20/3}}{\log K}\right),\\
J_1(\delta K^{-2/3})
={}&J_1(K^{-2/3})
-2\pi\eps^3K^{7/3}
\bigl[3b_0(\xi)(\delta-1)+b_1(\xi)(\delta^{-2}-1)\bigr]+O\!\left(\delta^{-8}\frac{K^{-20/3}}{\log K}\right).
\end{aligned}
\]
At the comparison point, 
\[
\begin{aligned}
J_1(d_*)
={}&J_1(K^{-2/3})
-2\pi\eps^3K^{7/3}
\left[\frac92b_0(\xi)d_*K^{2/3}-3b_0(\xi)-b_1(\xi)\right]+O\!\left(\delta^{-8}\frac{K^{-20/3}}{\log K}\right).
\end{aligned}
\]
Shrinking $\delta$ if necessary, the uniform positive
bounds on $b_0(\xi)$ and $b_1(\xi)$ ensure that
\[
\begin{aligned}
&\min\bigl\{
3b_0(\xi)\delta^{-1}+b_1(\xi)\delta^2,
3b_0(\xi)\delta+b_1(\xi)\delta^{-2}
\bigr\}-\frac92b_0(\xi)d_*K^{2/3}\geq c_d>0,
\end{aligned}
\]
where $c_d$ is independent of $\delta$, $K$, and $\xi$.
Consequently,
\[
\begin{aligned}
&J_1(d_*)-
\max\bigl\{J_1(\delta^{-1}K^{-2/3}),J_1(\delta K^{-2/3})\bigr\}\\
&\qquad\geq 2\pi c_d\eps^3K^{7/3}
-C\delta^{-8}\frac{K^{-20/3}}{\log K}\\
&\qquad\geq 2\pi c_d\delta^3K^{-20/3}
-C\delta^{-8}\frac{K^{-20/3}}{\log K}>0
\end{aligned}
\]

\medskip
\noindent
(4) Suppose that
\[
|\widehat\beta|=\delta^{-1/2}K^{-1/3}(\log K)^{-1/2}
\quad\text{or}\quad
|\widehat\alpha_b|=\delta^{-1/2}K^{-1/3}(\log K)^{-1/2}.
\]
Take
\[
A^{(\mathrm{ang})}:=(\eps,\xi,a,d,0,0).
\]
Both angle coordinates of this point lie strictly inside
their allowed intervals.
By \eqref{eq:main-boundary-energy-expansion},
\[
\begin{aligned}
\Phi_K(A^{(\mathrm{ang})})-\Phi_K(A)
={}&\frac{\pi}{4}\eps^3K^3|\nabla q(\xi)|^2
\left(\frac{\zeta(3)}{\pi^3}\widehat\beta^2
+\frac{93\zeta(5)}{8\pi^5}\widehat\alpha_b^2\right)+\delta^{-8}o\!\left(\frac{K^{-20/3}}{\log K}\right)\\
\geq{}&c\delta^2\frac{K^{-20/3}}{\log K}
+\delta^{-8}o\!\left(\frac{K^{-20/3}}{\log K}\right)>0
\end{aligned}
\]

Through these discussions, the maximum
cannot be attained on $\partial\mathcal A_K$ and must
lie in $\operatorname{int}\mathcal A_K$.
\end{proof}

\begin{proof}[Proof of Theorem~1.1]
Fix $\mu>0$ and choose $\delta$ as in
Theorem~\ref{thm:main-interior-maximum}.
For every sufficiently large even $K$, this theorem
and Lemma~\ref{lem:main-critical-point} give
$A_K\in\operatorname{int}\mathcal A_K$ such that
$\lambda_j(A_K)=0$ for $j=0,\ldots,5$.
Consequently,
\[
v_K:=PQ_{A_K}+\phi_{A_K}
\]
solves the Neumann problem in $\Sigma_K$ and is even
in $x_3$. By the even reflection result in Section~2,
\[
u_K:=E_Kv_K\in\mathcal H_K
\]
is a bounded weak solution of the Neumann problem in
$B_1$. Standard elliptic regularity gives
$u_K\in C^2(\overline{B_1})$, and the required symmetries
follow from $u_K\in\mathcal H_K$.

We first verify that $u_K$ changes sign.
Write $\eps_K$, $b_K$, $\widehat\xi_K$, and $\beta_K$
for the quantities associated with $A_K$.
The projection estimate, Theorem~\ref{thm:fdr-outer},
and \eqref{eq:fdr-nonlinear-estimate} imply
\begin{equation}
\label{eq:main-sector-profile-approximation}
\eps_K^{1/2}\|v_K-Q_{A_K}\|_{L^\infty(\Sigma_K)}
\leq C\delta^{-5/2}K^{-4}
+C\delta^{-4}K^{-11/3}\longrightarrow0.
\end{equation}
Choose fixed points $p^+,p^-\in\mathbb R^3$ such that
$q(p^+)>0>q(p^-)$. The quantities $|p^\pm-\widehat\xi_K|$ are bounded
above and away from zero for large $K$.
Define
\[
x_K^\pm
:=b_K+\eps_K R_{\beta_K}^{T}
\frac{p^\pm-\widehat\xi_K}{|p^\pm-\widehat\xi_K|^2}.
\]
These points belong to $\Sigma_K$ for large $K$, since
$|x_K^\pm-b_K|\leq C\eps_K$ and
\eqref{eq:appA-profile-smallness} gives
$\operatorname{dist}(b_K,\partial\Sigma_K)>K^{-1}$.
The formula for $Q_{A_K}$ yields
\[
\eps_K^{1/2}Q_{A_K}(x_K^\pm)
=|p^\pm-\widehat\xi_K|\,q(p^\pm).
\]
Therefore, by \eqref{eq:main-sector-profile-approximation},
there exists $c>0$, independent of $K$, such that
\begin{equation}
\label{eq:main-opposite-sign-values}
\eps_K^{1/2}u_K(x_K^+)\geq c,
\qquad
\eps_K^{1/2}u_K(x_K^-)\leq-c
\end{equation}
for all sufficiently large $K$.
Thus $u_K$ is sign-changing.

To prove that $u_K$ is nonradial, set
\[
y_K:=\frac{|x_K^+|}{\sqrt2}(1,0,1).
\]
Then $y_K\in\Sigma_K$ and $|y_K|=|x_K^+|$.
Moreover, $|x_K^+|\to1$, while the third coordinate
of $b_K$ is zero, so
\[
|y_K-b_K|\geq\frac{|x_K^+|}{\sqrt2}\geq c>0.
\]
Since $q$ is bounded, the formula for $Q_{A_K}$ gives
\[
\eps_K^{1/2}|Q_{A_K}(y_K)|
\leq\frac{\eps_K\|q\|_\infty}{|y_K-b_K|}
\leq C\eps_K\longrightarrow0.
\]
Together with \eqref{eq:main-sector-profile-approximation},
this implies $\eps_K^{1/2}|u_K(y_K)|\to0$.
In view of \eqref{eq:main-opposite-sign-values},
$u_K(y_K)\neq u_K(x_K^+)$ for large $K$.
Hence it rules out the radial symmetry property.

Finally, we distinguish the solutions by their energies.
Since $B_1$ consists of $K$ congruent reflected sectors,
\eqref{eq:main-boundary-energy-expansion} and
\eqref{eq:outline-parameters} give
\[
\begin{aligned}
J_{B_1}(u_K)
&:=\frac12\int_{B_1}
\bigl(|\nabla u_K|^2+\mu u_K^2\bigr)\,dx
-\frac16\int_{B_1}|u_K|^6\,dx\\
&=KJ_{\Sigma_K}(v_K)=KC_q+O(\delta^{-8}K^{-5}).
\end{aligned}
\]
In particular,
\[
J_{B_1}(u_{K+2})-J_{B_1}(u_K)
=2C_q+O(\delta^{-8}K^{-5})>0
\]
for all sufficiently large even $K$.
This proves the theorem.
\end{proof}

\appendix
\section{Estimate of Approximate Solution}
\label{app:interactions}

In this appendix, we derive an expansion of the approximate solution
$P Q_A$ introduced in Section~2 and estimate its difference from $Q_A$.
In particular, we shall prove that
\[
\|P Q_A-Q_A\|_{L^\infty(\Sigma_K)}
\leq C\delta^{-1/2}\varepsilon^{3/2}K^2.
\]
The expansion will also be used in the energy estimates in Appendix~B.
Throughout this appendix, the parameters satisfy
\eqref{eq:outline-parameters}, and $K\geq K_0(\delta)$ is sufficiently
large and even.

Put
\begin{equation*}
\widehat\xi:=\xi+a\nu_q(\xi).
\end{equation*}
Since $q(\xi)=0$ on $\Gamma$, under
\eqref{eq:outline-parameters}, for sufficiently large $K$ we have
\begin{equation}
|q(\widehat\xi)|\leq C|a|,
\qquad
\operatorname{dist}(b,\partial\Sigma_K)>\frac1K.
\label{eq:appA-profile-smallness}
\end{equation}
We shall use the notation
\begin{equation*}
w:=R_\beta^T\nabla q(\widehat\xi),
\qquad
W:=R_\beta^T\nabla^2q(\widehat\xi)R_\beta.
\end{equation*}
Then
\begin{equation*}
|w|=|\nabla q(\widehat\xi)|\leq C,
\qquad
|W|=|\nabla^2q(\widehat\xi)|\leq C,
\end{equation*}
where $C$ depends only on the fixed profile $q$.
Since $q$ is even in its third coordinate, $w_3=0$, and we may write
$w=(|w|e^{i\alpha_w},0)^T$. By the definition of $\theta_*(\xi)$
and $\beta=\theta_*(\xi)+\widehat\beta$, we have
\begin{equation*}
\alpha_w=-\widehat\beta+O(a).
\end{equation*}
In particular, under \eqref{eq:outline-parameters},
\begin{equation*}
|\alpha_w|+|\widehat\alpha_b|
\leq C\delta^{-1/2}K^{-1/3}(\log K)^{-1/2},
\qquad \widehat\alpha_b:=K\alpha_b.
\end{equation*}

For the function $Q_A^e$ introduced in Section~2, we write
\begin{equation*}
Q_A^e=Q_A+T_A,
\end{equation*}
where $T_A$ is the sum of the other copies of $Q_A$.
To estimate their contribution in $\Sigma_K$, we use the kernel
$\gamma(x,p)$ defined in \eqref{eq:appA-gamma-definition}
in Appendix~C.
The following lemma gives a lower bound for each denominator.

\begin{lemma}
\label{lem:appA-image-separation}
Under \eqref{eq:outline-parameters}, for every $x\in\Sigma_K$,
\begin{equation*}
\begin{aligned}
|R_{4j\theta_0}x-b|
&\geq C\bigl(|x-b|+\sin(2j\theta_0)\bigr),
&&1\leq j\leq\frac K2-1,\\
|R_{4j\theta_0}S_Kx-b|
&\geq C\bigl(|x-b|+\sin((2j+1)\theta_0)\bigr),
&&0\leq j\leq\frac K2-1.
\end{aligned}
\end{equation*}
\end{lemma}

\begin{proof}
The proof follows directly from the geometric arguments in
\cite[Lemma~3.1]{SunWeiYangFull},
using $\rho\geq1/2$ and $|\alpha_b|\leq\theta_0/2$ from
\eqref{eq:outline-parameters}.
\end{proof}

\begin{lemma}
\label{lem:appA-bubble-expansion}
Under \eqref{eq:outline-parameters}, for every $x\in\mathbb R^3$
satisfying $|x-b|\geq\varepsilon$, one has
\begin{equation}
\begin{aligned}
Q_A(x)={}&\frac{\varepsilon^{1/2}q(\widehat\xi)}{|x-b|}
+\varepsilon^{3/2}w_i
\left.\partial_{p_i}\frac1{|x-p|}\right|_{p=b}+\frac{\varepsilon^{5/2}}6W_{ij}
\left.\partial_{p_ip_j}^2\frac1{|x-p|}\right|_{p=b}
+R_A(x).
\end{aligned}
\label{eq:appA-bubble-expansion}
\end{equation}
Moreover, for $m=0,1,2,3$,
\begin{equation}
|D_x^mR_A(x)|
\leq C\Bigl(
\varepsilon^{5/2}|q(\widehat\xi)|^5|x-b|^{-3-m}
+\varepsilon^{7/2}|x-b|^{-4-m}
\Bigr).
\label{eq:appA-bubble-remainder}
\end{equation}
\end{lemma}

\begin{proof}
Set
\[
s=|x-b|,
\qquad
h=\frac{\varepsilon R_\beta(x-b)}{s^2}.
\]
Then $Q_A(x)=\varepsilon^{1/2}s^{-1}q(\widehat\xi+h)$.
Since $s\geq\varepsilon$, we have $|h|=\varepsilon/s\leq1$.
Taylor expansion of $q$ at $\widehat\xi$ gives
\begin{equation*}
q(\widehat\xi+h)
=q(\widehat\xi)+\nabla q(\widehat\xi)\cdot h
+\frac12h^T\nabla^2q(\widehat\xi)h+R_3(h),
\end{equation*}
where
\begin{equation}
|D_h^\ell R_3(h)|\leq C|h|^{3-\ell},
\qquad \ell=0,1,2,3.
\label{eq:appA-taylor-remainder-bounds}
\end{equation}
Therefore, by the definitions of $w$ and $W$,
\begin{equation}
Q_A(x)
=\frac{\varepsilon^{1/2}}s q(\widehat\xi)
+\varepsilon^{3/2}w_i\frac{(x-b)_i}{s^3}
+\frac{\varepsilon^{5/2}}2W_{ij}
  \frac{(x-b)_i(x-b)_j}{s^5}
+\frac{\varepsilon^{1/2}}s R_3(h).
\label{eq:appA-bubble-taylor}
\end{equation}
We now rewrite the first- and second-order terms by derivatives of
the Newtonian kernel. Direct differentiation gives
\begin{equation*}
\partial_{p_i}\frac1{|x-p|}
=\frac{(x-p)_i}{|x-p|^3},
\end{equation*}
and
\[
\frac{(x-b)_i(x-b)_j}{s^5}
=\frac13\left[
\left.\partial_{p_ip_j}^2\frac1{|x-p|}\right|_{p=b}
+\frac{\delta_{ij}}{s^3}
\right].
\]
Substituting these identities into \eqref{eq:appA-bubble-taylor}
gives \eqref{eq:appA-bubble-expansion} with
\[
R_A=R_{A,1}+R_{A,2},
\qquad
R_{A,1}(x):=\frac{\varepsilon^{1/2}}sR_3(h),
\qquad
R_{A,2}(x):=\frac{\varepsilon^{5/2}}6
\frac{W_{ij}\delta_{ij}}{s^3}.
\]
Since $R_\beta$ is orthogonal and $\Delta q+q^5=0$,
\[
W_{ij}\delta_{ij}
=\operatorname{tr}W
=\operatorname{tr}\bigl(R_\beta^T\nabla^2q(\widehat\xi)R_\beta\bigr)
=\Delta q(\widehat\xi)
=-q(\widehat\xi)^5.
\]
Consequently,
\[
|R_{A,2}(x)|
\leq C\varepsilon^{5/2}|q(\widehat\xi)|^5s^{-3}.
\]
Thus \eqref{eq:appA-bubble-remainder} holds for $m=0$.

It remains to estimate the derivatives of the remainder.
First observe that, for $\ell=0,1,2,3$,
\begin{equation}
|D_x^\ell h|\leq C_\ell\varepsilon s^{-1-\ell},
\qquad
\left|D_x^\ell\frac1s\right|
\leq C_\ell s^{-1-\ell}.
\label{eq:appA-kelvin-factor-derivatives}
\end{equation}
For one derivative, the product and chain rules give
\[
|D_xR_{A,1}|
\leq C\varepsilon^{1/2}s^{-2}|h|^3
+C\varepsilon^{1/2}s^{-1}|h|^2|D_xh|
\leq C\varepsilon^{7/2}s^{-5}.
\]
For $m=2,3$, applying the same rules together with
\eqref{eq:appA-taylor-remainder-bounds} and
\eqref{eq:appA-kelvin-factor-derivatives} yields
\[
|D_x^mR_{A,1}(x)|
\leq C\varepsilon^{7/2}s^{-4-m}.
\]
Similarly,
\[
|D_x^mR_{A,2}(x)|
\leq C\varepsilon^{5/2}|q(\widehat\xi)|^5s^{-3-m},
\qquad m=0,1,2,3.
\]
Combining the two estimates proves
\eqref{eq:appA-bubble-remainder}.
\end{proof}

\begin{proposition}
\label{lem:appA-interaction-expansion}
Under \eqref{eq:outline-parameters}, for every $x\in\Sigma_K$,
one has
\begin{equation}
\begin{aligned}
T_A(x)={}&\varepsilon^{1/2}q(\widehat\xi)\gamma(x,b)
+\varepsilon^{3/2}
\left.w\cdot\nabla_p\gamma(x,p)\right|_{p=b}+\frac{\varepsilon^{5/2}}6W_{ij}
\left.\partial_{p_ip_j}^2\gamma(x,p)\right|_{p=b}
+E_T(x),
\end{aligned}
\label{eq:appA-interaction-expansion}
\end{equation}
where, for $m=0,1,2,3$,
\begin{equation}
\|D_x^mE_T\|_{L^\infty(\Sigma_K)}
\leq C\varepsilon^{7/2}K^{4+m}.
\label{eq:appA-interaction-remainder}
\end{equation}
In particular,
\begin{equation}
\|T_A\|_{L^\infty(\Sigma_K)}
\leq C\varepsilon^{3/2}K^2.
\label{eq:appA-interaction-sup-bound}
\end{equation}
\end{proposition}

\begin{proof}
By Lemma~\ref{lem:appA-image-separation}, the arguments of all
the copies entering $T_A$ have distance at least $cK^{-1}$ from
$b$, which is larger than $\varepsilon$ for sufficiently large $K$.
Applying Lemma~\ref{lem:appA-bubble-expansion} to these copies
and summing gives \eqref{eq:appA-interaction-expansion}, with
\[
E_T(x)
=\sum_{j=1}^{K/2-1}R_A(R_{4j\theta_0}x)
+\sum_{j=0}^{K/2-1}R_A(R_{4j\theta_0}S_Kx).
\]

Since the transformations $R_{4j\theta_0}$ and
$R_{4j\theta_0}S_K$ are orthogonal, the remainder estimate in
Lemma~\ref{lem:appA-bubble-expansion}, together with
Lemma~\ref{lem:appA-image-separation} and
\[
\sum_{\ell=1}^{K-1}\frac1{\sin^s(\ell\theta_0)}
\leq C_sK^s,
\qquad s>1,
\]
yields, for $m=0,1,2,3$,
\[
\begin{aligned}
\|D_x^mE_T\|_{L^\infty(\Sigma_K)}
&\leq C\Bigl(
\varepsilon^{5/2}|q(\widehat\xi)|^5K^{3+m}
+\varepsilon^{7/2}K^{4+m}
\Bigr)\leq C\varepsilon^{7/2}K^{4+m}.
\end{aligned}
\]
This proves \eqref{eq:appA-interaction-remainder}.

Finally, using $|w|+|W|\leq C$ and the estimates for $\gamma$
and its derivatives in Lemma~\ref{lem:appA-gamma-estimates},
we obtain
\[
\begin{aligned}
\|T_A\|_{L^\infty(\Sigma_K)}
&\leq C\Bigl(
\varepsilon^{1/2}|q(\widehat\xi)|K\log K
+\varepsilon^{3/2}K^2
+\varepsilon^{5/2}K^3
+\varepsilon^{7/2}K^4
\Bigr)\leq C\varepsilon^{3/2}K^2,
\end{aligned}
\]
This completes the proof.
\end{proof}

Recall the kernels $H_N$ and $H_N^e$ defined in
\eqref{eq:appC-HN-explicit} and \eqref{eq:appC-HN-even}.
For each $p\in B_1$, a direct computation gives
\begin{equation}
\Delta_xH_N(x,p)=-3\quad\text{in }B_1,
\qquad
\partial_{\nu_x}H_N(x,p)
=\partial_{\nu_x}\frac1{|x-p|}
\quad\text{on }\partial B_1.
\label{eq:appA-HN-boundary-identities}
\end{equation}

\begin{proposition}
\label{prop:appA-boundary-correction}
Under \eqref{eq:outline-parameters}, define
\begin{equation}
\begin{aligned}
\eta_A(x):={}&\varepsilon^{1/2}q(\widehat\xi)H_N^e(x,b)
+\varepsilon^{3/2}
\left.w\cdot\nabla_pH_N^e(x,p)\right|_{p=b}+\frac{\varepsilon^{5/2}}6W_{ij}
\left.\partial_{p_ip_j}^2H_N^e(x,p)\right|_{p=b}.
\end{aligned}
\label{eq:appA-boundary-correction-definition}
\end{equation}
Then
\begin{equation}
\Delta\eta_A=-3K\varepsilon^{1/2}q(\widehat\xi)
\qquad\text{in }B_1.
\label{eq:appA-boundary-correction-equation}
\end{equation}
Moreover,
\begin{equation}
\|\eta_A\|_{L^\infty(B_1)}
\leq C\delta^{-4}\varepsilon^{3/2}Kd^{-1}.
\label{eq:appA-boundary-correction-bound}
\end{equation}
\end{proposition}

\begin{proof}
Since the transformations in the definition of $H_N^e$ are
orthogonal, \eqref{eq:appA-HN-boundary-identities} gives
\[
\Delta_xH_N^e(x,p)=-3K.
\]
The right-hand side is independent of $p$. Differentiating in $p$
and using \eqref{eq:appA-boundary-correction-definition}, we obtain
\eqref{eq:appA-boundary-correction-equation}.

By Lemma~\ref{lem:appC-HN-estimates} and the symmetry of
$H_N^e$ in $x$,
\[
\|H_N^e(\cdot,b)\|_{L^\infty(B_1)}
\leq CK\log\frac Cd,
\qquad
\left\|
\left.\nabla_p^mH_N^e(\cdot,p)\right|_{p=b}
\right\|_{L^\infty(B_1)}
\leq CKd^{-m},
\quad m=1,2.
\]
Together with $|w|+|W|\leq C$, these estimates yield
\[
\|\eta_A\|_{L^\infty(B_1)}
\leq C\left(
\varepsilon^{1/2}|q(\widehat\xi)|K\log\frac Cd
+\varepsilon^{3/2}Kd^{-1}
+\varepsilon^{5/2}Kd^{-2}
\right).
\]
Consequently,
\[
\|\eta_A\|_{L^\infty(B_1)}
\leq C\delta^{-4}\varepsilon^{3/2}Kd^{-1}.
\]
This completes the proof.
\end{proof}

On $\partial B_1$, define
\begin{equation*}
F_A:=\partial_\nu(Q_A^e-\eta_A).
\end{equation*}
By Lemma~\ref{lem:appA-bubble-expansion} and
\eqref{eq:appA-HN-boundary-identities}, the normal derivatives
of the three leading terms cancel, and hence
\[
F_A(x)=\partial_{\nu_x}\sum_{g\in G_K}R_A(gx),
\qquad x\in\partial B_1.
\]
Here $R_A$ is the Taylor remainder in
\eqref{eq:appA-bubble-expansion}.
By Lemma~\ref{lem:appA-image-separation}, symmetry, and
$Kd\geq1$, we have
\[
\sup_{x\in\partial B_1}\sum_{g\in G_K}|gx-b|^{-s}
\leq C_sKd^{1-s},
\qquad s>1.
\]
Using \eqref{eq:appA-bubble-remainder} with $m=1$, we obtain
\begin{equation*}
\|F_A\|_{L^\infty(\partial B_1)}
\leq C\varepsilon^{7/2}Kd^{-4}.
\end{equation*}

Let $\mathcal R_A$ be the weak solution of
\begin{equation}
\begin{cases}
(\Delta-\mu)\mathcal R_A=0 &\text{in }B_1,\\
\partial_\nu\mathcal R_A=F_A &\text{on }\partial B_1.
\end{cases}
\label{eq:appA-neumann-correction-problem}
\end{equation}

\begin{proposition}
\label{prop:appA-neumann-correction}
Under \eqref{eq:outline-parameters},
\eqref{eq:appA-neumann-correction-problem} has a unique weak
solution $\mathcal R_A\in H^1(B_1)$, and
\begin{equation}
\|\mathcal R_A\|_{L^\infty(B_1)}
\leq C_\mu\|F_A\|_{L^\infty(\partial B_1)}
\leq C\varepsilon^{7/2}Kd^{-4}.
\label{eq:appA-neumann-correction-bound}
\end{equation}
Moreover, $\mathcal R_A\in C^\infty(\overline{B_1})$ is
$G_K$-invariant and even in $x_3$.
\end{proposition}

\begin{proof}
Since $\mu>0$, existence and uniqueness follow from the
Lax--Milgram theorem. Smoothness follows from elliptic
regularity, since $F_A$ is smooth on $\partial B_1$.
Since $F_A$ is $G_K$-invariant and even in $x_3$, uniqueness
implies that $\mathcal R_A$ has the same symmetries.

Set
\[
M:=\|F_A\|_{L^\infty(\partial B_1)},
\qquad
\Psi(x):=M\left(\frac3\mu+\frac{|x|^2}{2}\right).
\]
Then
\[
(\Delta-\mu)\Psi
=-\frac{\mu M}{2}|x|^2\leq0
\quad\text{in }B_1,
\qquad
\partial_\nu\Psi=M\geq|F_A|
\quad\text{on }\partial B_1.
\]
The weak comparison principle applied to
$\Psi+\mathcal R_A$ and $\Psi-\mathcal R_A$ gives
\[
\|\mathcal R_A\|_{L^\infty(B_1)}
\leq\left(\frac3\mu+\frac12\right)
\|F_A\|_{L^\infty(\partial B_1)}
\leq C\varepsilon^{7/2}Kd^{-4}.
\]
This completes the proof.
\end{proof}

We now seek a function $w_A$ such that
\begin{equation}
P Q_A=Q_A^e-\eta_A-\mathcal R_A-w_A
=Q_A+T_A-\eta_A-\mathcal R_A-w_A
\label{eq:appA-projection-decomposition}
\end{equation}
satisfies \eqref{eq:outline-projection}.
Consequently, it suffices to choose $w_A$ as the solution of
\begin{equation}
\left\{
\begin{aligned}
(\Delta-\mu)w_A
&=3K\varepsilon^{1/2}q(\widehat\xi)
-\mu(Q_A+T_A-\eta_A)-\sum_{g\in G_K\setminus\{I\}}Q_A(gx)^5
&&\text{in }\Sigma_K,\\
\partial_\nu w_A&=0
&&\text{on }\partial\Sigma_K.
\end{aligned}
\right.
\label{eq:appA-sector-correction}
\end{equation}
Since $\mu>0$, the Lax--Milgram theorem gives a unique weak
solution $w_A\in H^1(\Sigma_K)$. By uniqueness, $w_A$ is even
in $x_3$. The function $E_Kw_A$ solves the corresponding equation in
$B_1$, with the right-hand side extended by $E_K$ and with
homogeneous Neumann boundary condition on $\partial B_1$.
The right-hand side of \eqref{eq:appA-sector-correction} is smooth
on $\overline{\Sigma_K}$, so its even extension belongs to
$C^{0,\alpha}(\overline{B_1})$ for every $0<\alpha<1$.
Standard elliptic regularity therefore gives
\[
E_Kw_A\in C^{2,\alpha}(\overline{B_1}),
\qquad
w_A\in C^{2,\alpha}(\overline{\Sigma_K}).
\]
The expression in \eqref{eq:appA-projection-decomposition}
therefore satisfies \eqref{eq:outline-projection} and coincides
with $P Q_A$ by uniqueness.

To estimate $w_A$, we shall use the identity for the Neumann problem
in the unit ball.

\begin{lemma}
\label{lem:appA-neumann-green-estimate}
Let $\mu>0$, and let $G_\mu^N(x,y)$ be the Neumann Green
function of $\Delta-\mu$ in $B_1$. Then
\begin{equation}
|G_\mu^N(x,y)|
\leq\frac{C_\mu}{|x-y|},
\qquad x,y\in B_1,\quad x\neq y.
\label{eq:appA-neumann-green-pointwise}
\end{equation}
Consequently, for every $f\in L^2(B_1)$, the weak solution of
\begin{equation*}
\begin{cases}
(\Delta-\mu)u=f &\text{in }B_1,\\
\partial_\nu u=0 &\text{on }\partial B_1
\end{cases}
\end{equation*}
satisfies
\begin{equation}
\|u\|_{L^\infty(B_1)}
\leq C_\mu\|f\|_{L^2(B_1)}.
\label{eq:appA-neumann-L2-Linfinity}
\end{equation}
\end{lemma}

\begin{proof}
The bound \eqref{eq:appA-neumann-green-pointwise} is the
three-dimensional form of the standard Neumann Green-function
estimate recalled in
\cite[Appendix~A, p.~1115]{WangWeiYan2011}.
The opposite sign convention for the operator does not affect
this absolute-value estimate.
By the Green representation formula and the Cauchy--Schwarz
inequality,
\[
\begin{aligned}
|u(x)|=\left|\int_{B_1}G_\mu^N(x,y)f(y)\,dy\right|\leq C_\mu
\left(\int_{B_1}\frac{dy}{|x-y|^2}\right)^{1/2}
\|f\|_{L^2(B_1)}.
\end{aligned}
\]
Since $B_1-x\subset B_2$, in dimension three we have
\[
\sup_{x\in B_1}
\int_{B_1}\frac{dy}{|x-y|^2}
\leq\int_{B_2}\frac{dz}{|z|^2}
=4\pi\int_0^2dr=8\pi.
\]
Taking the supremum over $x\in B_1$ gives
\eqref{eq:appA-neumann-L2-Linfinity}.
\end{proof}

We can now complete the estimate of the approximate solution.

\begin{proposition}
\label{prop:appA-projection-estimate}
Under \eqref{eq:outline-parameters}, the solution of
\eqref{eq:outline-projection} can be written as
\begin{equation}
P Q_A=Q_A+\widetilde T_A,
\qquad
\widetilde T_A:=T_A-\eta_A-\mathcal R_A-w_A,
\label{eq:appA-projection-remainder-definition}
\end{equation}
where
\begin{equation}
\|\widetilde T_A\|_{L^\infty(\Sigma_K)}
\leq C\delta^{-1/2}\varepsilon^{3/2}K^2.
\label{eq:appA-projection-error-bound}
\end{equation}
Moreover, for each fixed $K$, $\widetilde T_A$ depends $C^1$
on the parameters in $A$.
\end{proposition}

\begin{proof}
It remains to estimate $w_A$ in
\eqref{eq:appA-projection-decomposition}.
The decay of $q$ and Taylor's theorem give, for
$x\in\mathbb R^3$ and $s=|x-b|$,
\begin{equation}
|Q_A(x)|\leq
\begin{cases}
C\varepsilon^{-1/2}, &s\leq\varepsilon,\\[3pt]
C\left(\displaystyle\frac{\varepsilon^{1/2}|q(\widehat\xi)|}{s}
+\displaystyle\frac{\varepsilon^{3/2}}{s^2}\right),
&s\geq\varepsilon.
\end{cases}
\label{eq:appA-bubble-pointwise}
\end{equation}

Consequently,
\begin{equation*}
\|E_KQ_A\|_{L^2(B_1)}
=K^{1/2}\|Q_A\|_{L^2(\Sigma_K)}
\leq C\varepsilon K^{1/2}.
\end{equation*}

For the other copies, Lemma~\ref{lem:appA-image-separation}
gives $|gx-b|\geq cK^{-1}>\varepsilon$.
Thus \eqref{eq:appA-bubble-pointwise} implies
\[
\begin{aligned}
\sup_{x\in\Sigma_K}
\sum_{g\in G_K\setminus\{I\}}|Q_A(gx)|^5
&\leq C\Bigl(
\varepsilon^{5/2}|q(\widehat\xi)|^5K^5
+\varepsilon^{15/2}K^{10}\Bigr)
\leq C\varepsilon^{3/2}K^2.
\end{aligned}
\]
Moreover, \eqref{eq:appA-interaction-sup-bound} and
\eqref{eq:appA-boundary-correction-bound} give
\[
\|T_A\|_{L^\infty(\Sigma_K)}
+\|\eta_A\|_{L^\infty(\Sigma_K)}
\leq C\varepsilon^{3/2}K^2.
\]
The weak maximum principle gives an $L^\infty$ operator norm
at most $\mu^{-1}$ for the Neumann inverse of $\Delta-\mu$.
By linearity and Lemma~\ref{lem:appA-neumann-green-estimate}
applied after even reflection, equation
\eqref{eq:appA-sector-correction} yields
\[
\begin{aligned}
\|w_A\|_{L^\infty(\Sigma_K)}
\leq{}&C\Bigl(
K\varepsilon^{1/2}|q(\widehat\xi)|
+\|E_KQ_A\|_{L^2(B_1)}
+\|T_A\|_{L^\infty(\Sigma_K)}\\
&\qquad+\|\eta_A\|_{L^\infty(\Sigma_K)}
+\sup_{x\in\Sigma_K}
\sum_{g\in G_K\setminus\{I\}}|Q_A(gx)|^5
\Bigr).
\end{aligned}
\]

Thus
\begin{equation*}
\|w_A\|_{L^\infty(\Sigma_K)}
\leq C\delta^{-1/2}\varepsilon^{3/2}K^2.
\end{equation*}

Combining these estimates with
\eqref{eq:appA-projection-decomposition} proves
\eqref{eq:appA-projection-error-bound}.
The $C^1$ dependence follows from the smooth parameter
dependence of $Q_A$ and the bounded invertibility of the
Neumann operator $-\Delta+\mu$.
\end{proof}

We also care about the magnitude of the derivatives with respect to parameters. Next, we derive some estimates on them. We use the family $\mathcal U_{\mathcal P}$ and its projection
$V_{\mathcal P}$ introduced in Section~\ref{subsec:outer}.
Throughout next two lemmas, $\eps$, $K$, and $r$ are fixed
when differentiating with respect to $\mathcal P$.

\begin{lemma}
\label{lem:appA-parameter-projection}
Under \eqref{eq:outline-parameters}, the map
$\mathcal P\mapsto V_{\mathcal P}$ is $C^1$ near
$\mathcal P_A$, with values in
$H^1(\Omega_\eps)\cap L^\infty(\Omega_\eps)$.
At $\mathcal P=\mathcal P_A$, uniformly for
$y\in\Omega_\eps$, we have
\begin{equation}
\label{eq:appA-parameter-projection-pointwise}
\begin{aligned}
|\partial_{P_j}\mathcal U_{\mathcal P}(y)|
\leq C\langle y-b_\eps\rangle^{-1}, \quad
|\partial_{P_j}V_{\mathcal P}(y)|
\leq C\left(
\langle y-b_\eps\rangle^{-1}+\eps K^2
\right).
\end{aligned}
\end{equation}
\end{lemma}

\begin{proof}
The Kelvin invariance of $q$ gives a smooth continuation of
$\mathcal U_{\mathcal P}$ across $y=c$, jointly in $y$ and
$\mathcal P$. Its first parameter derivatives are therefore
uniformly bounded near $c$. Away from $c$, direct
differentiation of its defining formula gives
\[
|\mathcal U_{\mathcal P}(y)|
+|\partial_{P_j}\mathcal U_{\mathcal P}(y)|
\leq C\langle y-c\rangle^{-1}.
\]
Since the Neumann inverse on $\Omega_\eps$ is
fixed, differentiation of
\eqref{eq:fdr-parameter-projection} proves the $C^1$ dependence. At $\mathcal P=\mathcal P_A$, the differentiated source satisfies
\[
|\mathcal U_{\mathcal P}^4
\partial_{P_j}\mathcal U_{\mathcal P}|
\leq C\langle y-b_\eps\rangle^{-5}.
\]
Using the rescaled Green function estimate
\eqref{eq:neumann-pointwise}, we obtain
\[
\begin{aligned}
|\partial_{P_j}V_{\mathcal P}(y)|
&\leq C\sum_{g\in G_K}\int_{\Omega_\eps}
\frac{\langle z-b_\eps\rangle^{-5}}{|y-gz|}\,dz\leq C\sum_{g\in G_K}
\langle y-gb_\eps\rangle^{-1}\leq C(\langle y-b_\eps\rangle^{-1}+\eps K^2).
\end{aligned}
\]
In the last step, we used the elementary convolution estimate
\[
\int_{\mathbb R^3}\frac{\langle z\rangle^{-5}}{|Y-z|}\,dz
\leq C\langle Y\rangle^{-1}.
\]
This proves \eqref{eq:appA-parameter-projection-pointwise}.
\end{proof}

We next estimate the derivatives of the projection error.

\begin{lemma}
\label{lem:appA-projection-error-derivative}
Under \eqref{eq:outline-parameters}, at
$\mathcal P=\mathcal P_A$ we have
\begin{equation}
\label{eq:appA-projection-error-derivative}
\left\|\left.\partial_{P_j}
\bigl(V_{\mathcal P}-\mathcal U_{\mathcal P}\bigr)
\right|_{\mathcal P_A}\right\|_{L^\infty(\Omega_\eps)}
\leq C\eps K^2,
\qquad j=0,\ldots,5.
\end{equation}
\end{lemma}

\begin{proof}
Let $Z_j$ be defined by \eqref{eq:fdr-kernel-functions}.
The profile formula gives
\[
|Z_j(y)|+\langle y-b_\eps\rangle|\nabla Z_j(y)|
\leq C\langle y-b_\eps\rangle^{-1},
\qquad
\Delta Z_j+5U^4Z_j=0.
\]
Set
\[
H_j=\left.\partial_{P_j}V_{\mathcal P}\right|_{\mathcal P_A}
-\chi Z_j.
\]
Since $\chi Z_j$ vanishes near $\partial\Omega_\eps$,
$H_j$ satisfies homogeneous Neumann boundary conditions
and
\[
(\Delta-\mu\eps^2)H_j
=-5(1-\chi)U^4Z_j+\mu\eps^2\chi Z_j-\mathcal C(Z_j).
\]
Writing $\rho=|y-b_\eps|$, the source is bounded by
\[
C\left(
\mathbf 1_{\{\rho\geq 2r/\bar\eps\}}\langle\rho\rangle^{-5}
+\eps^2\mathbf 1_{\{\rho\leq 4r/\bar\eps\}}
 \langle\rho\rangle^{-1}
+\bar\eps^3
 \mathbf 1_{\{2r/\bar\eps\leq\rho\leq 4r/\bar\eps\}}
\right).
\]
The Newton potentials of these three terms have uniform
norms bounded by $C\bar\eps^3$, $C\eps^2\bar\eps^{-1}$,
and $C\bar\eps$, respectively.
Thus \eqref{eq:neumann-pointwise}, summed over the $K$
reflected copies, gives
\[
\|H_j\|_\infty
\leq CK\bigl(\bar\eps^3+\eps^2\bar\eps^{-1}+\bar\eps\bigr)
\leq C\eps K^2.
\]
Since $\|(1-\chi)Z_j\|_\infty\leq C\bar\eps$, we obtain
\eqref{eq:appA-projection-error-derivative}.
\end{proof}

\begin{lemma}
\label{lem:appA-projection-error-gradient}
Under \eqref{eq:outline-parameters}, for
$K\geq K_0(\delta)$ sufficiently large, the profiles
$U$ and $V$ defined in \eqref{eq:fdr-profiles} satisfy
\begin{equation*}
\|\nabla(V-U)\|_{L^\infty(D_K)}
\leq C\delta^{-5/2}K^{-4},
\end{equation*}
\end{lemma}

\begin{proof}
The equations for $U$ and $V$ give
\[
\Delta(V-U)=\mu\eps^2V
\qquad\text{in }\Omega_\eps.
\]
Since $B_2(y)\subset\Omega_\eps$ for every $y\in D_K$.
Applying the interior gradient estimate on $B_2(y)$,
we obtain
\[
\begin{aligned}
|\nabla(V-U)(y)|\leq C\Bigl(
\|V-U\|_{L^\infty(\Omega_\eps)}
+\eps^2\|V\|_{L^\infty(\Omega_\eps)}
\Bigr)
\leq C\delta^{-5/2}K^{-4}.
\end{aligned}
\]
\end{proof}
\section{Energy Expansion}
\label{app:energy}

In this appendix, we derive an asymptotic expansion of the energy
$J(P Q_A)$ under \eqref{eq:outline-parameters} and identify its
leading-order dependence on the parameters.
We first collect the following estimates.

\begin{lemma}
\label{lem:appB-bubble-integrals}
Under \eqref{eq:outline-parameters}, we have
\begin{equation*}
\int_{\Sigma_K}|Q_A|^\ell\,dx
\leq C
\begin{cases}
\delta^{-3}\varepsilon^{3/2}K^{-1/3}(\log K)^{-1},
&\ell=1,\\[3pt]
\varepsilon^{3-\ell/2},
&\ell=2,3,4,5,
\end{cases}
\end{equation*}
\begin{equation}
\int_{\Sigma_K}Q_A^6\,dx
=\int_{\mathbb R^3}q(x)^6\,dx
+O\!\left(\varepsilon^9K^9\right),
\label{eq:appB-bubble-sixth-moment}
\end{equation}
\begin{equation}
\int_{\Sigma_K}Q_A^5\,dx
=4\pi\varepsilon^{1/2}q(\widehat\xi)
+O\!\left(\varepsilon^{15/2}K^7\right),
\label{eq:appB-bubble-fifth-moment}
\end{equation}
\begin{equation*}
\int_{\Sigma_K}Q_A(x)^5(x-b)\,dx
=4\pi\varepsilon^{3/2}w
+O\!\left(\varepsilon^{15/2}K^6\right),
\end{equation*}
and
\begin{equation*}
\int_{\Sigma_K}|Q_A|^5|x-b|^2\,dx
=O\!\left(\varepsilon^{5/2}\right).
\end{equation*}
\end{lemma}

\begin{lemma}
\label{lem:appB-boundary-interaction}
Under \eqref{eq:outline-parameters}, we have
\begin{equation*}
\begin{aligned}
\frac1{4\pi}\int_{\Sigma_K}Q_A^5\eta_A\,dx
={}&\varepsilon[q(\widehat\xi)]^2H_N^e(b,b)\\
&+\varepsilon^2q(\widehat\xi)w\cdot
\bigl[\nabla_xH_N^e(b,b)+\nabla_pH_N^e(b,b)\bigr]\\
&+\varepsilon^3w^T\nabla_{x,p}^2H_N^e(b,b)w
+O\!\left(\varepsilon^4K^4\right).
\end{aligned}
\end{equation*}
\end{lemma}

\begin{lemma}
\label{lem:appB-image-interaction}
Under \eqref{eq:outline-parameters}, we have
\begin{equation*}
\begin{aligned}
\frac1{4\pi}\int_{\Sigma_K}Q_A^5T_A\,dx
={}&\varepsilon[q(\widehat\xi)]^2\gamma(b,b)\\
&+\varepsilon^2q(\widehat\xi)w\cdot
\bigl[\nabla_x\gamma(b,b)+\nabla_p\gamma(b,b)\bigr]\\
&+\varepsilon^3w^T\nabla_{x,p}^2\gamma(b,b)w
+O\!\left(\varepsilon^4K^4\right).
\end{aligned}
\end{equation*}
\end{lemma}

\begin{lemma}
\label{lem:appB-linear-correction}
Under \eqref{eq:outline-parameters}, let $w_{A,1}\in H^1(\Sigma_K)$
be the unique weak solution of
\begin{equation}
\begin{cases}
(\Delta-\mu)w_{A,1}=-\mu Q_A &\text{in }\Sigma_K,\\
\partial_\nu w_{A,1}=0 &\text{on }\partial\Sigma_K.
\end{cases}
\label{eq:appB-linear-correction-problem}
\end{equation}
Then
\begin{equation*}
\frac1{4\pi}\int_{\Sigma_K}Q_A^5w_{A,1}\,dx
=\frac\mu{4\pi}\int_{\Sigma_K}Q_A^2\,dx
+O\!\left(
\delta^{-7/2}\varepsilon^3K^{5/3}(\log K)^{-1}
\right).
\end{equation*}
\end{lemma}
\begin{lemma}
\label{lem:appB-bubble-L2-expansion}
Under \eqref{eq:outline-parameters}, we have
\begin{equation*}
\frac1{4\pi}\int_{\Sigma_K}Q_A^2\,dx
=\varepsilon^2\operatorname{P.V.}\!\int_{\mathbb R^3}
\frac{[q(x+\xi)]^2}{4\pi|x|^4}\,dx
+\delta^{-5}o\!\left(\frac{K^{-20/3}}{\log K}\right),
\end{equation*}
\end{lemma}

\begin{proof}[Proof of Lemmas~\ref{lem:appB-bubble-integrals}--\ref{lem:appB-bubble-L2-expansion}]
The proofs follow the arguments of
\cite[Lemmas~4.1--4.5]{SunWeiYangFull}, respectively, adapted
to the present Neumann setting and the parameter bounds
\eqref{eq:outline-parameters}.
We replace $H_0^e$ by $H_N^e$ and use the estimates established
in Appendices~A and~C; for Lemma~\ref{lem:appB-linear-correction},
the integration-by-parts argument is applied to $w_1$ defined in
\eqref{eq:appB-linear-correction-problem}.
\end{proof}

We next estimate the contribution of $\mathcal R_A$ to the energy.

\begin{lemma}
\label{lem:appB-neumann-correction-energy}
Under \eqref{eq:outline-parameters}, we have
\begin{equation*}
\left|\int_{\Sigma_K}Q_A^5\mathcal R_A\,dx\right|
\leq C\varepsilon^4Kd^{-4}.
\end{equation*}
\end{lemma}

\begin{proof}
By \eqref{eq:appA-neumann-correction-bound} and
Lemma~\ref{lem:appB-bubble-integrals}, we obtain
\[
\begin{aligned}
\left|\int_{\Sigma_K}Q_A^5\mathcal R_A\,dx\right|
&\leq\|\mathcal R_A\|_{L^\infty(\Sigma_K)}
\int_{\Sigma_K}|Q_A|^5\,dx\leq C\varepsilon^4Kd^{-4}.
\end{aligned}
\]
\end{proof}

We now turn to the contribution of $w_A$ to the energy and
first consider its constant part.

\begin{lemma}
\label{lem:appB-constant-correction-energy}
Under \eqref{eq:outline-parameters}, let
\begin{equation}
w_{A,0}:=-\frac{3K\varepsilon^{1/2}q(\widehat\xi)}{\mu}.
\label{eq:appB-constant-correction-definition}
\end{equation}
Then
\begin{equation*}
\frac1{4\pi}\int_{\Sigma_K}Q_A^5w_{A,0}\,dx
=-\frac3\mu\varepsilon K[q(\widehat\xi)]^2
+O\!\left(\varepsilon^4K^4\right).
\end{equation*}
\end{lemma}
We omit the proof, which follows directly from
\eqref{eq:appB-bubble-fifth-moment} and
\eqref{eq:outline-parameters}.
Since the contribution of $w_{A,1}$ has already been estimated
in Lemma~\ref{lem:appB-linear-correction}, it remains to estimate
the contribution of $w_{A,2}$, where
\begin{equation}
w_{A,2}:=w_A-w_{A,0}-w_{A,1}.
\label{eq:appB-correction-splitting}
\end{equation}
By \eqref{eq:appA-sector-correction},
\eqref{eq:appB-constant-correction-definition}, and
\eqref{eq:appB-linear-correction-problem}, we have
\begin{equation}
\left\{
\begin{aligned}
(\Delta-\mu)w_{A,2}
&=-\mu(T_A-\eta_A)
-\sum_{g\in G_K\setminus\{I\}}Q_A(gx)^5
&&\text{in }\Sigma_K,\\
\partial_\nu w_{A,2}&=0
&&\text{on }\partial\Sigma_K.
\end{aligned}
\right.
\label{eq:appB-remainder-correction-problem}
\end{equation}
After even reflection, this becomes a Neumann problem in $B_1$.
We shall use the following estimate.

\begin{lemma}
\label{lem:appB-neumann-W24-estimate}
Let $\mu>0$ and $f\in L^4(B_1)$, and suppose that
$u\in H^1(B_1)$ is a weak solution of
\begin{equation}
\begin{cases}
(\Delta-\mu)u=f &\text{in }B_1,\\
\partial_\nu u=0 &\text{on }\partial B_1.
\end{cases}
\label{eq:appB-neumann-L4-problem}
\end{equation}
Then $u\in W^{2,4}(B_1)$ and
\begin{equation}
\|u\|_{L^\infty(B_1)}
+\|\nabla u\|_{L^\infty(B_1)}
\leq C_\mu\|f\|_{L^4(B_1)},
\label{eq:appB-neumann-W1infty-bound}
\end{equation}
where $C_\mu>0$ depends only on $\mu$.
\end{lemma}

\begin{proof}
Testing \eqref{eq:appB-neumann-L4-problem} by $u$ gives
\[
\int_{B_1}|\nabla u|^2\,dx
+\mu\int_{B_1}u^2\,dx
=-\int_{B_1}fu\,dx
\leq\|f\|_{L^2(B_1)}\|u\|_{L^2(B_1)}.
\]
Consequently,
\[
\|u\|_{L^2(B_1)}\leq\mu^{-1}\|f\|_{L^2(B_1)},
\qquad
\|\nabla u\|_{L^2(B_1)}
\leq\mu^{-1/2}\|f\|_{L^2(B_1)}.
\]
The Sobolev embedding $H^1(B_1)\hookrightarrow L^4(B_1)$
and the boundedness of $B_1$ therefore yield
\[
\|u\|_{L^4(B_1)}
\leq C\|u\|_{H^1(B_1)}
\leq C_\mu\|f\|_{L^2(B_1)}
\leq C_\mu\|f\|_{L^4(B_1)}.
\]
Global $W^{2,4}$ regularity for the homogeneous Neumann problem
on the smooth domain $B_1$ now gives
\begin{equation*}
\|u\|_{W^{2,4}(B_1)}
\leq C_\mu\bigl(\|f\|_{L^4(B_1)}+\|u\|_{L^4(B_1)}\bigr)
\leq C_\mu\|f\|_{L^4(B_1)}.
\end{equation*}
Finally, the Sobolev--Morrey embedding
$W^{2,4}(B_1)\hookrightarrow C^{1,1/4}(\overline{B_1})$
implies
\[
\|u\|_{L^\infty(B_1)}
+\|\nabla u\|_{L^\infty(B_1)}
\leq C\|u\|_{W^{2,4}(B_1)}
\leq C_\mu\|f\|_{L^4(B_1)}.
\]
This proves \eqref{eq:appB-neumann-W1infty-bound}.
\end{proof}

\begin{lemma}
\label{lem:appB-remainder-correction-energy}
Under \eqref{eq:outline-parameters}, we have
\begin{equation}
\frac1{4\pi}\int_{\Sigma_K}Q_A^5w_{A,2}\,dx
=\delta^{-5/2}o\!\left(\frac{K^{-20/3}}{\log K}\right).
\label{eq:appB-remainder-correction-energy}
\end{equation}
\end{lemma}

\begin{proof}
We first estimate the right-hand side of
\eqref{eq:appB-remainder-correction-problem} in $L^4$ after
even reflection. The pointwise estimates in
Lemmas~\ref{lem:appA-gamma-estimates} and~\ref{lem:appC-HN-estimates},
together with integration over the sector, give
\begin{equation*}
\begin{aligned}
\|E_K\gamma(\cdot,b)\|_{L^4(B_1)}
+\|E_KH_N^e(\cdot,b)\|_{L^4(B_1)}
&\leq CK,\\
\left\|E_K\!\left[
\left.D_p^m\gamma(\cdot,p)\right|_{p=b}
\right]\right\|_{L^4(B_1)}
&\leq CK^{m+1/2},\\
\left\|E_K\!\left[
\left.D_p^mH_N^e(\cdot,p)\right|_{p=b}
\right]\right\|_{L^4(B_1)}
&\leq CKd^{1/2-m},
\qquad m=1,2.
\end{aligned}
\end{equation*}

Using \eqref{eq:appA-interaction-expansion},
\eqref{eq:appA-interaction-remainder}, and
\eqref{eq:appA-boundary-correction-definition}, we obtain
\[
\begin{aligned}
\|E_K(T_A-\eta_A)\|_{L^4(B_1)}
\leq C\Bigl(&\varepsilon^{1/2}|q(\widehat\xi)|K
+\varepsilon^{3/2}K^{3/2}+\varepsilon^{5/2}K^{5/2}
+\varepsilon^{7/2}K^4\Bigr).
\end{aligned}
\]
Moreover, the separation of the other copies and
\eqref{eq:appA-bubble-pointwise} imply
\[
\left\|E_K\!\left[
\sum_{g\in G_K\setminus\{I\}}Q_A(g\cdot)^5
\right]\right\|_{L^4(B_1)}
\leq C\Bigl(
\varepsilon^{5/2}|q(\widehat\xi)|^5K^5
+\varepsilon^{15/2}K^{10}\Bigr).
\]
Consequently, \eqref{eq:outline-parameters} give
\[
\left\|E_K\!\left[
\mu(T_A-\eta_A)
+\sum_{g\in G_K\setminus\{I\}}Q_A(g\cdot)^5
\right]\right\|_{L^4(B_1)}
\leq C\delta^{-5/2}\frac{K^{-17/6}}{\log K}.
\]
Applying Lemma~\ref{lem:appB-neumann-W24-estimate} to
\eqref{eq:appB-remainder-correction-problem} after even reflection,
we conclude that
\begin{equation}
\|E_Kw_{A,2}\|_{L^\infty(B_1)}
+\|\nabla E_Kw_{A,2}\|_{L^\infty(B_1)}
\leq C\delta^{-5/2}\frac{K^{-17/6}}{\log K}.
\label{eq:appB-remainder-correction-regularity}
\end{equation}

We now estimate the energy contribution. By Cauchy--Schwarz and
Lemma~\ref{lem:appB-bubble-integrals},
\begin{equation*}
\begin{aligned}
\int_{\Sigma_K}|Q_A|^5|x-b|\,dx
&\leq
\left(\int_{\Sigma_K}|Q_A|^5\,dx\right)^{1/2}
\left(\int_{\Sigma_K}|Q_A|^5|x-b|^2\,dx\right)^{1/2}\leq C\varepsilon^{3/2}.
\end{aligned}
\end{equation*}
Since $B_1$ is convex, the gradient bound in
\eqref{eq:appB-remainder-correction-regularity} yields
\[
|w_{A,2}(x)-w_{A,2}(b)|
\leq C\delta^{-5/2}\frac{K^{-17/6}}{\log K}|x-b|,
\qquad x\in\Sigma_K.
\]
Thus, writing
\[
\int_{\Sigma_K}Q_A^5w_{A,2}\,dx
=w_{A,2}(b)\int_{\Sigma_K}Q_A^5\,dx
+\int_{\Sigma_K}Q_A^5\bigl(w_{A,2}(x)-w_{A,2}(b)\bigr)\,dx,
\]
and using \eqref{eq:appB-bubble-fifth-moment}, we obtain
\[
\begin{aligned}
\left|\frac1{4\pi}\int_{\Sigma_K}Q_A^5w_{A,2}\,dx\right|
&\leq C\delta^{-5/2}\frac{K^{-17/6}}{\log K}
\Bigl(\varepsilon^{1/2}|q(\widehat\xi)|
+\varepsilon^{3/2}+\varepsilon^{15/2}K^7\Bigr)
\end{aligned}
\]
This proves \eqref{eq:appB-remainder-correction-energy}.
\end{proof}

We now combine the preceding estimates to obtain the energy
expansion of $P Q_A$. Here we set
\begin{equation*}
C_*(\xi):=\operatorname{P.V.}\!\int_{\mathbb R^3}
\frac{[q(x+\xi)]^2}{4\pi|x|^4}\,dx.
\end{equation*}

\begin{proposition}
\label{prop:appB-projected-energy-expansion}
Under \eqref{eq:outline-parameters}, we have
\begin{equation}
J(P Q_A)
=\frac13\int_{\mathbb R^3}q(x)^6\,dx
+2\pi\Psi(A)
+\delta^{-8}o\!\left(\frac{K^{-20/3}}{\log K}\right),
\label{eq:appB-projected-energy-expansion}
\end{equation}
where
\begin{equation}
\begin{aligned}
\Psi(A)={}&\varepsilon[q(\widehat\xi)]^2
\left[H_N^e(b,b)-\gamma(b,b)-\frac{3K}{\mu}\right]\\
&+\varepsilon^2q(\widehat\xi)w\cdot
\left[\nabla_x(H_N^e-\gamma)(b,b)
+\nabla_p(H_N^e-\gamma)(b,b)\right]\\
&+\varepsilon^3w^T\nabla_{x,p}^2(H_N^e-\gamma)(b,b)w
+\mu\varepsilon^2C_*(\xi).
\end{aligned}
\label{eq:appB-leading-energy}
\end{equation}
\end{proposition}

\begin{proof}
Taking $P Q_A$ as a test function in
\eqref{eq:outline-projection}, we obtain
\[
\int_{\Sigma_K}\bigl(|\nabla P Q_A|^2+\mu(P Q_A)^2\bigr)\,dx
=\int_{\Sigma_K}Q_A^5P Q_A\,dx.
\]
Consequently,
\begin{equation*}
J(P Q_A)
=\frac12\int_{\Sigma_K}Q_A^5P Q_A\,dx
-\frac16\int_{\Sigma_K}(P Q_A)^6\,dx.
\end{equation*}

Since
\[
\left|(Q_A+\widetilde T_A)^6-Q_A^6
-6Q_A^5\widetilde T_A\right|
\leq C\left(|Q_A|^4|\widetilde T_A|^2
+|\widetilde T_A|^6\right).
\]
By \eqref{eq:appA-projection-error-bound},
Lemma~\ref{lem:appB-bubble-integrals}, it follows that
\[
\begin{aligned}
\int_{\Sigma_K}
\left(|Q_A|^4|\widetilde T_A|^2
+|\widetilde T_A|^6\right)\,dx
\leq C\delta^{-1}\varepsilon^4K^4,
\end{aligned}
\]
Therefore,
\begin{equation}
J(P Q_A)
=\frac13\int_{\Sigma_K}Q_A^6\,dx
-\frac12\int_{\Sigma_K}Q_A^5\widetilde T_A\,dx
+O\!\left(\delta^{-1}\varepsilon^4K^4\right).
\label{eq:appB-energy-reduction}
\end{equation}
By \eqref{eq:appA-projection-remainder-definition} and
\eqref{eq:appB-correction-splitting},
\[
\widetilde T_A
=T_A-\eta_A-\mathcal R_A-w_{A,0}-w_{A,1}-w_{A,2}.
\]
Hence \eqref{eq:appB-energy-reduction} becomes
\begin{equation*}
\begin{aligned}
J(P Q_A)={}&\frac13\int_{\Sigma_K}Q_A^6\,dx\\
&+\frac12\int_{\Sigma_K}Q_A^5
\left(\eta_A-T_A+\mathcal R_A
+w_{A,0}+w_{A,1}+w_{A,2}\right)\,dx+O\!\left(\delta^{-1}\varepsilon^4K^4\right).
\end{aligned}
\end{equation*}
Substituting the estimates of
Lemmas~\ref{lem:appB-boundary-interaction}--\ref{lem:appB-constant-correction-energy}
and Lemma~\ref{lem:appB-remainder-correction-energy}, and using
\eqref{eq:appB-bubble-sixth-moment}, we obtain
\[
\begin{aligned}
J(P Q_A)={}&\frac13\int_{\mathbb R^3}q(x)^6\,dx
+2\pi\Psi(A)\\
&+O\!\left(
\delta^{-1}\varepsilon^4K^4
+\varepsilon^4Kd^{-4}
+\delta^{-7/2}\varepsilon^3K^{5/3}(\log K)^{-1}
\right)+\delta^{-5}o\!\left(\frac{K^{-20/3}}{\log K}\right).
\end{aligned}
\]
This proves \eqref{eq:appB-projected-energy-expansion}.
\end{proof}
\section{Trigonometric Sums}
\label{app:sums}

In this section, we prove some results involving trigonometric sums.
Throughout this appendix, $K$ is a sufficiently large even integer and
$\theta_0=\pi/K$. We write $\gamma_E$ for the Euler constant and
$\zeta$ for the Riemann zeta function. 

\subsection{Existing Results}
For a positive integer $m$, define
\begin{align}
S_m^o(K)
&:=\sum_{j=0}^{K/2-1}\frac{1}{\sin^m((2j+1)\theta_0)},
\notag\\
S_m^e(K)
&:=\sum_{j=1}^{K/2-1}\frac{1}{\sin^m(2j\theta_0)},
\notag\\
S_m^+(K)
&:=S_m^o(K)+S_m^e(K)
=\sum_{\ell=1}^{K-1}\frac{1}{\sin^m(\ell\theta_0)}.
\label{eq:appC-full-sum}
\end{align}
For $d>0$, we also define
\begin{equation}
S_m^+(K,d)
:=\sum_{\ell=0}^{K-1}
\bigl(d^2+\sin^2(\ell\theta_0)\bigr)^{-m/2}.
\label{eq:appC-d-sum}
\end{equation}

The sums independent of $d$ satisfy
\begin{equation}
S_m^e(K)=S_m^+\!\left(\frac K2\right),
\qquad
S_m^o(K)=S_m^+(K)-S_m^+\!\left(\frac K2\right),
\label{eq:appC-sum-relations}
\end{equation}
where the full sum $S_m^+(N)$ is defined by the last expression in
\eqref{eq:appC-full-sum} for every integer $N\geq2$.
Table~\ref{tab:appC-cosecant-sums} records the cases used in the
interaction estimates.
The formulas for $m=1$ follow from
\cite[Theorem~7]{BlagouchineMoreau2024} and
\eqref{eq:appC-sum-relations}. For $m=3,5$, the odd sums and the
even sum with $m=3$ are given in \cite[Lemma~9.1]{SunWeiYangFull};
the other entries follow from the same endpoint expansion and
\eqref{eq:appC-sum-relations}. The even sum in that reference is
denoted by $\widehat S_m^e$.

\begin{table}[htbp]
\centering
\small
\setlength{\tabcolsep}{5pt}
\renewcommand{\arraystretch}{2.1}
\caption{Finite cosecant sums as $K\to\infty$.}
\label{tab:appC-cosecant-sums}
\smallskip
\begin{tabular}{@{}cccc@{}}
\toprule
$m$ & $S_m^o(K)$ & $S_m^e(K)$ & $S_m^+(K)$\\
\midrule
$1$
& $\begin{gathered}
\frac K\pi\left(\log\frac{4K}{\pi}+\gamma_E\right)+O(K^{-1})
\end{gathered}$
& $\begin{gathered}
\frac K\pi\left(\log\frac K\pi+\gamma_E\right)
+O(K^{-1})
\end{gathered}$
& $\begin{gathered}
\frac{2K}{\pi}\left(\log\frac{2K}{\pi}+\gamma_E\right)
+O(K^{-1})
\end{gathered}$\\[4pt]
$3$
& $\dfrac{7\zeta(3)}{4\pi^3}K^3+O(K\log K)$
& $\dfrac{\zeta(3)}{4\pi^3}K^3+O(K\log K)$
& $\dfrac{2\zeta(3)}{\pi^3}K^3+O(K\log K)$\\[4pt]
$5$
& $\dfrac{31\zeta(5)}{16\pi^5}K^5+O(K^3)$
& $\dfrac{\zeta(5)}{16\pi^5}K^5+O(K^3)$
& $\dfrac{2\zeta(5)}{\pi^5}K^5+O(K^3)$\\
\bottomrule
\end{tabular}
\end{table}

For the $d$-dependent sums defined in \eqref{eq:appC-d-sum},
the expansions used in the interaction estimates are collected in
Table~\ref{tab:appC-d-sums}.

\begin{table}[htbp]
\centering
\small
\setlength{\tabcolsep}{12pt}
\renewcommand{\arraystretch}{2.2}
\caption{Asymptotic expansions of the $d$-dependent full sums.}
\label{tab:appC-d-sums}
\smallskip
\begin{tabular}{@{}cc@{}}
\toprule
$m$ & Expansion of $S_m^+(K,d)$ as $K\to\infty$\\
\midrule
$1$
& $\displaystyle
\frac{2K}{\pi}\log\frac4d
+O\!\left(Kd^2\log\frac2d\right)$\\[4pt]
$3$
& $\displaystyle
\frac{2K}{\pi d^2}
+O\!\left(K\log\frac2d\right)$\\[4pt]
$5$
& $\displaystyle
\frac{4K}{3\pi d^4}
+O\!\left(\frac K{d^2}\right)$\\
\bottomrule
\end{tabular}
\end{table}

\begin{lemma}
\label{lem:appC-d-sum-expansions}
For every fixed $\delta\in(0,1)$ and every sufficiently large even
$K\geq K_0(\delta)$, the expansions in Table~\ref{tab:appC-d-sums}
hold uniformly for $d$ satisfying \eqref{eq:outline-parameters}.
The constants in the $O$-terms are independent of $K,d$ and $\delta$.
\end{lemma}

\begin{proof}
We take $K_0(\delta)$ large enough that $d\leq1/2$ and $Kd\geq1$.
For $m=1,3,5$, set
\[
I_m(d):=\int_0^\pi(d^2+\sin^2 t)^{-m/2}\,dt.
\]
The function $(d^2+\sin^2 t)^{-m/2}$ is $\pi$-periodic and analytic
in $|\operatorname{Im}t|<\tfrac12\operatorname{arsinh}d$, where it is
bounded by $Cd^{-m}$. Since the width of this strip is comparable to
$d$ for $0<d\leq1$, the periodic trapezoidal rule
\cite[Theorem~3.2]{TrefethenWeideman2014} gives, when $Kd\geq1$,
\begin{equation*}
S_m^+(K,d)=\frac K\pi I_m(d)
+O\!\left(Kd^{-m}e^{-cKd}\right).
\end{equation*}

It remains to estimate $I_m(d)$. By symmetry and the substitution
$x=\sin t$,
\[
I_m(d)=2\int_0^1
\frac{(d^2+x^2)^{-m/2}}{\sqrt{1-x^2}}\,dx.
\]
Write $h(x)=(1-x^2)^{-1/2}-1$. Then $h(x)=O(x^2)$ near zero,
$h\in L^1(0,1)$, and
\[
\int_0^1\frac{h(x)}x\,dx
=\int_0^{\pi/2}\tan\frac t2\,dt=\log2.
\]
For $m=1$, it follows that
\begin{align*}
I_1(d)
&=2\operatorname{arsinh}(d^{-1})
  +2\int_0^1\frac{h(x)}{\sqrt{d^2+x^2}}\,dx\\
&=2\log\frac2d+2\log2
  +O\!\left(d^2\log\frac2d\right)\\
&=2\log\frac4d+O\!\left(d^2\log\frac2d\right).
\end{align*}
Here the error estimate follows by splitting the integral at $x=d$
and $x=1/2$ and using
\[
0\leq\frac1x-\frac1{\sqrt{d^2+x^2}}
\leq\min\left\{\frac1x,\frac{d^2}{2x^3}\right\}.
\]

For $m=3,5$, the same decomposition gives
\[
\left|I_m(d)-2\int_0^\infty(d^2+x^2)^{-m/2}\,dx\right|
\leq C\left(1+\int_0^{1/2}
\frac{x^2}{(d^2+x^2)^{m/2}}\,dx\right).
\]
Using $x=ds$, together with
\[
\int_0^\infty(1+s^2)^{-3/2}\,ds=1,
\qquad
\int_0^\infty(1+s^2)^{-5/2}\,ds=\frac23,
\]
we obtain
\[
I_3(d)=\frac2{d^2}+O\!\left(\log\frac2d\right),
\qquad
I_5(d)=\frac4{3d^4}+O(d^{-2}).
\]
Substituting the three integral estimates proves the lemma.
\end{proof}

\subsection{Estimation on Type I}
We first establish the basic estimates for the interaction kernel
associated with the even extension. Recall that
\begin{equation}
\gamma(x,p)
:=\sum_{j=1}^{K/2-1}\frac{1}{|R_{4j\theta_0}x-p|}
 +\sum_{j=0}^{K/2-1}\frac{1}{|R_{4j\theta_0}S_Kx-p|},
\qquad x,p\in\Sigma_K.
\label{eq:appA-gamma-definition}
\end{equation}

\begin{lemma}
\label{lem:appA-gamma-estimates}
Under \eqref{eq:outline-parameters}, for any $x\in\Sigma_K$, one has
\begin{equation*}
|\gamma(x,b)|\leq CK\min\bigl\{\log K,\,1+|\log|x-b||\bigr\}.
\end{equation*}
Moreover, for $m=1,2,3,4$,
\begin{equation*}
|\nabla^m\gamma(x,b)|
\leq C\min\bigl\{K^{m+1},\,K|x-b|^{-m}\bigr\}.
\end{equation*}
Moreover,
\begin{equation*}
\int_{\Sigma_K}|\gamma(x,b)|^2\,dx\leq CK,
\qquad
\int_{\Sigma_K}|\nabla\gamma(x,b)|^2\,dx\leq CK\log K.
\end{equation*}
\end{lemma}
\begin{proof}
Write $x=(re^{i\varphi},x_3)$ and set $s=|x-b|$. Then
\[
s^2=(r-\rho)^2+x_3^2+4r\rho\sin^2\frac{\varphi-\alpha_b}{2}.
\]
If $r<1/4$, in this region all denominators are uniformly bounded
away from zero, and therefore
\[
|\gamma(x,b)|\leq CK.
\]
We therefore assume from now on that $r\geq1/4$.
We proved those properties before
\begin{align*}
|R_{4j\theta_0}x-b|^2
&\geq C\bigl(s^2+\sin^2(2j\theta_0)\bigr),
&&1\leq j\leq\frac K2-1,
\\
|R_{4j\theta_0}S_Kx-b|^2
&\geq C\bigl(s^2+\sin^2((2j+1)\theta_0)\bigr),
&&0\leq j\leq\frac K2-1.
\end{align*}
Consequently, since the even and odd indices together cover
$1,\ldots,K-1$,
\begin{align*}
\gamma(x,b)
&\leq C\sum_{j=1}^{K/2-1}
\frac{1}{\bigl(s^2+\sin^2(2j\theta_0)\bigr)^{1/2}}
+C\sum_{j=0}^{K/2-1}
\frac{1}{\bigl(s^2+\sin^2((2j+1)\theta_0)\bigr)^{1/2}}
\notag\\
&=C\sum_{\ell=1}^{K-1}
\frac{1}{\bigl(s^2+\sin^2(\ell\theta_0)\bigr)^{1/2}}.
\end{align*}

We first assume that $s\geq\frac1K$. If $s\leq1$, an elementary
comparison with the corresponding integral gives
\[
\sum_{\ell=1}^{K-1}
\frac{1}{\bigl(s^2+\sin^2(\ell\theta_0)\bigr)^{1/2}}
\leq CK(1+|\log s|).
\]
If $s\geq1$, the same conclusion follows immediately since
$s^{-1}\leq1$ and there are $O(K)$ terms. It remains to consider
the case when
\[
0<s<\frac1K.
\]
Note that
\[
\gamma(x,b)
\leq C\sum_{\ell=1}^{K-1}\frac1{\sin(\ell\theta_0)}
\leq C\sum_{\ell=1}^{K/2}\frac1{\sin(\ell\pi/K)}
\leq CK\sum_{\ell=1}^{K/2}\frac1\ell
\leq CK\log K.
\]
Combining these cases, we obtain
\[
|\gamma(x,b)|
\leq CK\min\bigl\{\log K,\,1+|\log|x-b||\bigr\}.
\]

We next estimate the derivatives. For every orthogonal transformation
$R$ and every multi-indices $\alpha,\beta$ with
\[
|\alpha|+|\beta|=m,
\]
the homogeneity of the Newtonian kernel gives
\begin{equation*}
\left|D_x^\alpha D_p^\beta\frac1{|Rx-p|}\right|
\leq\frac{C_m}{|Rx-p|^{m+1}}.
\end{equation*}
Therefore,
\begin{equation*}
\left.\left|D_x^\alpha D_p^\beta\gamma(x,p)\right|\right|_{p=b}
\leq C\sum_{\ell=1}^{K-1}
\frac{1}{\bigl(s^2+\sin^2(\ell\theta_0)\bigr)^{(m+1)/2}}.
\end{equation*}
Suppose first that $s\geq\frac1K$. By the same elementary finite-sum
estimate, now with exponent $k=m+1\geq2$,
\begin{equation*}
\sum_{\ell=1}^{K-1}
\frac{1}{\bigl(s^2+\sin^2(\ell\theta_0)\bigr)^{(m+1)/2}}
\leq CKs^{-m}.
\end{equation*}
If instead $s<\frac1K$, then
\begin{equation*}
\left.\left|D_x^\alpha D_p^\beta\gamma(x,p)\right|\right|_{p=b}
\leq C\sum_{\ell=1}^{K-1}\csc^{m+1}(\ell\theta_0)
\leq CK^{m+1}\sum_{\ell=1}^{K/2}\ell^{-(m+1)}
\leq CK^{m+1}.
\end{equation*}
Combining these cases, we obtain
\[
\left.\left|D_x^\alpha D_p^\beta\gamma(x,p)\right|\right|_{p=b}
\leq C\min\bigl\{K^{m+1},\,K|x-b|^{-m}\bigr\}.
\]

It remains to prove the square-integral estimates. We first record
a simple geometric observation. After translating by $b$, the sector
$\Sigma_K$ is contained in a slab whose tangential width is bounded
by $C/K$. Consequently, for every nonnegative nonincreasing function
$F=F(t)$, one has
\begin{equation*}
\int_{\Sigma_K}F(|x-b|)\,dx
\leq\frac CK\int_0^{C_0}F(t)\,t\,dt,
\end{equation*}
where $C_0>0$ is independent of $K$. We first consider $\gamma$.
From the pointwise estimate obtained above,
\[
|\gamma(x,b)|\leq CK\log\frac{C}{|x-b|+K^{-1}}.
\]
we obtain
\begin{align*}
\int_{\Sigma_K}|\gamma(x,b)|^2\,dx
&\leq CK\int_0^{C_0}t\log^2\frac{C}{t+K^{-1}}\,dt
\\
&\leq CK\int_0^{C_0}(t+K^{-1})
\log^2\frac{C}{t+K^{-1}}\,dt
\leq CK.
\end{align*}
Next, from the gradient estimate
\[
|\nabla_x\gamma(x,b)|\leq\frac{CK}{|x-b|+K^{-1}},
\]
Consequently,
\begin{align*}
\int_{\Sigma_K}|\nabla_x\gamma(x,b)|^2\,dx
&\leq CK\int_0^{C_0}\frac{t}{(t+K^{-1})^2}\,dt
\notag\\
&=CK\left[\log(t+K^{-1})
+\frac{K^{-1}}{t+K^{-1}}\right]_0^{C_0}
\leq CK\log K.
\end{align*}
The proof is complete.
\end{proof}

In this particular case $x=b$, we have a more precise estimate.

\begin{lemma}
\label{lem:appA-gamma-diagonal}
Under \eqref{eq:outline-parameters}, one has
\[
\gamma(b,b)=\frac{1}{2\rho}
\bigl[S_1^+(K)+O(K^3\alpha_b^2)\bigr],
\]
where $S_1^+(K)$ is defined in Table~\ref{tab:appC-cosecant-sums}.
\end{lemma}

\begin{lemma}
\label{lem:appA-gamma-gradient-diagonal}
Under \eqref{eq:outline-parameters}, write
$w=(|w|e^{i\alpha_w},0)^T$. Then
\[
w\cdot\nabla_x\gamma(b,b)
=\frac{|w|}{4\rho^2}
\Bigl[-S_1^+(K)
+O\bigl(K(\log K)(\alpha_w-\alpha_b)^2
+K^3|(\alpha_w-\alpha_b)\alpha_b|+K^3\alpha_b^2\bigr)\Bigr]
\]
and the same expansion applies to $w\cdot\nabla_p\gamma(b,b)$.
\end{lemma}

\begin{lemma}
\label{lem:appA-gamma-hessian-diagonal}
Under \eqref{eq:outline-parameters}, we have
\[
w^T\nabla_{x,p}^2\gamma(b,b)w
=\frac{|w|^2}{8\rho^3}
\Bigl[S_1^+(K)+S_3^+(K)
+\vartheta A_K^\gamma\vartheta^T
+O(K^3|\vartheta|^4)\Bigr],
\]
where we use the notation in Table~\ref{tab:appC-cosecant-sums} and
\[
\vartheta=(\alpha_w,\widehat\alpha_b),
\qquad
\widehat\alpha_b=K\alpha_b,
\]
\[
A_K^\gamma=
\begin{pmatrix}
-3S_3^e(K)+S_3^o(K)-2S_1^o(K)
&\dfrac{3}{K}\bigl[S_3^e(K)+S_1^o(K)-S_3^o(K)\bigr]
\\[6pt]
\dfrac{3}{K}\bigl[S_3^e(K)+S_1^o(K)-S_3^o(K)\bigr]
&\dfrac{1}{K^2}\bigl[6S_5^o(K)+\tfrac32S_3^o(K)
-\tfrac92S_1^o(K)-3S_3^e(K)\bigr]
\end{pmatrix}.
\]
\end{lemma}
\begin{proof}[Proof of Lemmas~\ref{lem:appA-gamma-diagonal},
\ref{lem:appA-gamma-gradient-diagonal}
and~\ref{lem:appA-gamma-hessian-diagonal}]
The proofs follow from the computations in
\cite[Lemmas~8.2--8.4]{SunWeiYangFull}, with the signs of the
rotational terms changed according to the even extension.
\end{proof}

\subsection{Estimation on Type II}
In this subsection, we prove the corresponding estimates for $H_N^e$.
Since most computations are similar to those for $\gamma(x,p)$,
we only record the main differences. Recall
\begin{equation}
H_N(x,p)
=-\frac{1}{B(x,p)}
-\log\frac{2}{B(x,p)+A(x,p)}
-\frac12\bigl(|x|^2+|p|^2\bigr)+\frac{14}{5},
\label{eq:appC-HN-explicit}
\end{equation}
where
\begin{equation*}
B(x,p)=\bigl(1-2x\cdot p+|x|^2|p|^2\bigr)^{1/2},
\qquad
A(x,p)=1-x\cdot p.
\end{equation*}
We define
\begin{equation}
H_N^e(x,p)
=\sum_{j=0}^{K/2-1}
\bigl[H_N(R_{4j\theta_0}x,p)
+H_N(R_{4j\theta_0}S_Kx,p)\bigr].
\label{eq:appC-HN-even}
\end{equation}

\begin{lemma}
\label{lem:appC-HN-estimates}
Under \eqref{eq:outline-parameters}, let $b^*=b/\rho^2$.
Then, for every $x\in\Sigma_K$,
\begin{equation*}
|H_N^e(x,b)|\leq CK\log\frac{C}{|x-b^*|}.
\end{equation*}
Moreover, for $m=1,2,3,4$,
\begin{equation}
|\nabla^mH_N^e(x,b)|
\leq CK|x-b^*|^{-m}
\leq CK(1-\rho|x|)^{-m}.
\label{eq:appC-HN-derivatives}
\end{equation}
Moreover,
\begin{equation*}
\int_{\Sigma_K}|H_N^e(x,b)|^2\,dx\leq CK,
\qquad
\int_{\Sigma_K}|\nabla H_N^e(x,b)|^2\,dx
\leq CK\log K.
\end{equation*}
\end{lemma}

\begin{proof}
Write $x=(re^{i\varphi},x_3)$ and $s_*=|x-b^*|$.
By \eqref{eq:outline-parameters}, for sufficiently large $K$,
\[
s_*\geq d_*:=|b^*|-1
=\frac{1-\rho}{\rho}
=\frac{2d}{1+\rho}
\geq\delta K^{-2/3}\geq K^{-1}.
\]
By the definition of $B$,
\[
B(x,b)^2=1-2x\cdot b+|x|^2\rho^2
=\rho^2|x-b^*|^2.
\]
For the rotated terms, direct computation gives
\begin{align*}
B(R_{4j\theta_0}x,b)^2
&=(1-r\rho)^2+\rho^2x_3^2
+4r\rho\sin^2\left(2j\theta_0
+\frac{\varphi-\alpha_b}{2}\right),
\\
B(R_{4j\theta_0}S_Kx,b)^2
&=(1-r\rho)^2+\rho^2x_3^2
+4r\rho\sin^2\left((2j+1)\theta_0
-\frac{\varphi+\alpha_b}{2}\right).
\end{align*}
By the same geometric estimates used in Type I,
\[
B(R_{4j\theta_0}x,b)^2
\geq C\bigl(s_*^2+\sin^2(2j\theta_0)\bigr),
\]
and
\[
B(R_{4j\theta_0}S_Kx,b)^2
\geq C\bigl(s_*^2+\sin^2((2j+1)\theta_0)\bigr).
\]

We first estimate the Newtonian part of $H_N$.
Since the even and odd indices together cover $0,\ldots,K-1$,
\begin{equation*}
\sum_{j=0}^{K/2-1}
\left[\frac{1}{B(R_{4j\theta_0}x,b)}
+\frac{1}{B(R_{4j\theta_0}S_Kx,b)}\right]
\leq\frac{C}{s_*}
+C\sum_{\ell=1}^{K-1}
\frac{1}{\bigl(s_*^2+\sin^2(\ell\theta_0)\bigr)^{1/2}}.
\end{equation*}
Since $s_*\geq K^{-1}$, the elementary finite-sum estimate gives
\[
\sum_{\ell=1}^{K-1}
\frac{1}{\bigl(s_*^2+\sin^2(\ell\theta_0)\bigr)^{1/2}}
\leq CK(1+|\log s_*|).
\]

For the logarithmic part, since $x,p\in B_1$,
\[
A(x,p)=1-x\cdot p>0,
\qquad
B(x,p)\leq B(x,p)+A(x,p)\leq C.
\]
Hence
\[
\left|\log\frac{2}{B(x,p)+A(x,p)}\right|
\leq C\bigl(1+|\log B(x,p)|\bigr).
\]
Using the lower bounds for $B$ above, we obtain
\[
\sum_{j=0}^{K/2-1}
\left|\log\frac{2}
{B(R_{4j\theta_0}x,b)+A(R_{4j\theta_0}x,b)}\right|
\leq CK\log\frac{C}{s_*},
\]
and the same estimate holds for the reflected terms.
The polynomial term is uniformly bounded. Combining the three parts of $H_N$, we obtain
\[
|H_N^e(x,b)|\leq CK\log\frac{C}{|x-b^*|}.
\]

We next estimate the derivatives. Direct differentiation gives
\[
|\nabla^m B(x,p)^{-1}|
\leq CB(x,p)^{-m-1},
\qquad m=1,2,3,4,
\]
while
\[
|\nabla^m\log(B(x,p)+A(x,p))|
\leq CB(x,p)^{-m},
\qquad m=1,2,3,4.
\]
Therefore,
\begin{equation}
|\nabla^mH_N^e(x,b)|
\leq C\sum_{\ell=0}^{K-1}
\frac{1}{\bigl(s_*^2+\sin^2(\ell\theta_0)\bigr)^{(m+1)/2}}
+C\sum_{\ell=0}^{K-1}
\frac{1}{\bigl(s_*^2+\sin^2(\ell\theta_0)\bigr)^{m/2}}
+CK.
\label{eq:appC-HN-derivative-sum}
\end{equation}
By the same elementary finite-sum estimate with $k=m+1$,
\[
\sum_{\ell=0}^{K-1}
\frac{1}{\bigl(s_*^2+\sin^2(\ell\theta_0)\bigr)^{(m+1)/2}}
\leq CKs_*^{-m}.
\]
The remaining terms in \eqref{eq:appC-HN-derivative-sum}
are of lower order. Consequently,
\[
|\nabla^mH_N^e(x,b)|
\leq CKs_*^{-m}=CK|x-b^*|^{-m}.
\]
Since
\[
|x-b^*|\geq\rho^{-1}-|x|
=\rho^{-1}(1-\rho|x|)\geq1-\rho|x|,
\]
the second inequality in \eqref{eq:appC-HN-derivatives} follows.

It remains to prove the square-integral estimates.
Using the same slab integration estimate as in Type I,
\begin{equation*}
\int_{\Sigma_K}|H_N^e(x,b)|^2\,dx
\leq CK\int_{d_*}^{C_0}t\log^2\frac Ct\,dt
\leq CK.
\end{equation*}
Similarly,
\begin{equation*}
\int_{\Sigma_K}|\nabla H_N^e(x,b)|^2\,dx
\leq CK\int_{d_*}^{C_0}\frac{dt}{t}
\leq CK\log K.
\end{equation*}
Here $C_0>0$ is independent of $K$.
The proof is complete.
\end{proof}

To get the precise expansion, we write
\[
K_0(x,p)=\frac{1}{B(x,p)},
\qquad
L(x,p)=\log\frac{2}{B(x,p)+A(x,p)},
\qquad
Q_0(x,p)=\frac12\bigl(|x|^2+|p|^2\bigr).
\]
Let $K_0^e,L^e$ and $Q_0^e$ denote their even extensions in the
first variable, as in \eqref{eq:appC-HN-even}. Then
\begin{equation}
H_N^e=-K_0^e-L^e-Q_0^e+\frac{14K}{5}.
\label{eq:appC-HN-even-decomposition}
\end{equation}

\begin{lemma}
\label{lem:appC-HN-diagonal}
Under \eqref{eq:outline-parameters}, one has
\[
H_N^e(b,b)
=-\frac{1}{2\rho}\bigl[S_1^+(K,d)+O(K)\bigr].
\]
\end{lemma}

\begin{proof}
Repeating the computation in \cite[Lemma~8.6]{SunWeiYangFull},
with the signs of the reflected terms changed according to the
even extension, gives
\[
K_0^e(b,b)
=\frac{1}{2\rho}
\bigl[S_1^+(K,d)+O(Kd^{-2}\alpha_b^2)\bigr].
\]

For the logarithmic part, the formulas for $B$ and the bound
$|\alpha_b|\leq\theta_0/2$ give
\begin{align}
B(R_{4j\theta_0}b,b)^2
&\geq c\bigl(d^2+\sin^2(2j\theta_0)\bigr),
\notag\\
B(R_{4j\theta_0}S_Kb,b)^2
&\geq c\bigl(d^2+\sin^2((2j+1)\theta_0)\bigr),
\qquad 0\leq j\leq K/2-1.
\label{eq:appC-HN-diagonal-B-bounds}
\end{align}
Since $A(x,p)>0$ and $B(x,p)\leq A(x,p)+B(x,p)\leq C$,
\[
|L(x,p)|\leq C\bigl(1+|\log B(x,p)|\bigr).
\]
Consequently,
\[
|L^e(b,b)|
\leq CK+C\log\frac2d
+C\sum_{\ell=1}^{K-1}\log\frac1{\sin(\ell\theta_0)}
\leq CK,
\]
where we used $d\geq K^{-1}$ and
\[
\prod_{\ell=1}^{K-1}\sin(\ell\theta_0)=\frac{K}{2^{K-1}}.
\]
The polynomial and constant terms are both $O(K)$, so
\eqref{eq:appC-HN-even-decomposition} yields the conclusion.
\end{proof}

\begin{lemma}
\label{lem:appC-HN-gradient-diagonal}
Under \eqref{eq:outline-parameters}, write
$w=(|w|e^{i\alpha_w},0)^T$ and $\widehat\alpha_b=K\alpha_b$.
One has
\[
\begin{aligned}
-w\cdot\nabla_xH_N^e(b,b)
={}&\frac{|w|}{4\rho^2}
\bigl[-S_1^+(K,d)+d\sqrt{1+d^2}\,S_3^+(K,d)\bigr]+O\!\left(|w|\left[
Kd^{-2}\alpha_w^2+K^{-1}d^{-4}\widehat\alpha_b^2
+K\log\frac2d\right]\right),
\end{aligned}
\]
and the same expansion applies to $-w\cdot\nabla_pH_N^e(b,b)$.
\end{lemma}

\begin{proof}
Repeating the computation in \cite[Lemma~8.7]{SunWeiYangFull},
with the signs of the reflected terms changed according to the
even extension, gives
\[
\begin{aligned}
w\cdot\nabla_xK_0^e(b,b)
={}&\frac{|w|}{4\rho^2}
\bigl[-S_1^+(K,d)+d\sqrt{1+d^2}\,S_3^+(K,d)\bigr]+O\!\left(|w|\left[
Kd^{-2}\alpha_w^2+K^{-1}d^{-4}\widehat\alpha_b^2
\right]\right),
\end{aligned}
\]
and the same expansion holds for $w\cdot\nabla_pK_0^e(b,b)$.

For the logarithmic part, $|\nabla_xB|=|p|$ and
$|\nabla_pB|=|x|$, and hence
\[
|\nabla_xL(x,p)|+|\nabla_pL(x,p)|
\leq\frac{|\nabla_xB|+|\nabla_pB|+|p|+|x|}{A+B}
\leq\frac{C}{B(x,p)}.
\]
Using \eqref{eq:appC-HN-diagonal-B-bounds} and
Table~\ref{tab:appC-d-sums}, we obtain
\begin{align*}
|w\cdot\nabla_xL^e(b,b)|+|w\cdot\nabla_pL^e(b,b)|
&\leq C|w|\sum_{\ell=0}^{K-1}
\frac{1}{\bigl(d^2+\sin^2(\ell\theta_0)\bigr)^{1/2}}=C|w|S_1^+(K,d)
\leq C|w|K\log\frac2d.
\end{align*}
The polynomial term contributes $O(|w|K)$, while the constant
term has zero derivative. Combining these estimates with
\eqref{eq:appC-HN-even-decomposition} proves the expansion.
\end{proof}

\begin{lemma}
\label{lem:appC-HN-hessian-diagonal}
Under \eqref{eq:outline-parameters}, one has
\[
\begin{aligned}
w^T\nabla_{x,p}^2H_N^e(b,b)w
={}&-\frac{|w|^2}{8\rho^3}
\bigl[S_1^+-(d+\sqrt{1+d^2})^2S_3^+
+3(d^2+d^4)S_5^+\bigr]\\
&+O\!\left(|w|^2\left[
Kd^{-2}\alpha_w^2+K^{-1}d^{-4}\widehat\alpha_b^2
+Kd^{-1}\right]\right),
\end{aligned}
\]
where $S_i^+=S_i^+(K,d)$ for $i=1,3,5$ and
$\widehat\alpha_b=K\alpha_b$. In particular,
\[
w^T\nabla_{x,p}^2H_N^e(b,b)w
=-\frac{K|w|^2}{4\pi\rho^3d^2}
+O\!\left(\delta^{-4}|w|^2Kd^{-1}\right).
\]
\end{lemma}

\begin{proof}
Repeating the computation in \cite[Lemma~8.8]{SunWeiYangFull},
with the signs of the reflected terms changed according to the
even extension, gives
\[
\begin{aligned}
w^T\nabla_{x,p}^2K_0^e(b,b)w
={}&\frac{|w|^2}{8\rho^3}
\bigl[S_1^+-(d+\sqrt{1+d^2})^2S_3^+
+3(d^2+d^4)S_5^+\bigr]\\
&+O\!\left(|w|^2\left[
Kd^{-2}\alpha_w^2+K^{-1}d^{-4}\widehat\alpha_b^2
\right]\right).
\end{aligned}
\]

For the logarithmic part, set $D=A+B$. Direct differentiation gives
\[
\nabla_{x,p}^2L
=-\frac{\nabla_{x,p}^2D}{D}
+\frac{\nabla_xD\otimes\nabla_pD}{D^2}.
\]
Since $D\geq B$, $|\nabla_xD|+|\nabla_pD|\leq C$ and
$|\nabla_{x,p}^2D|\leq CB^{-1}$, it follows that
\[
|\nabla_{x,p}^2L(x,p)|\leq CB(x,p)^{-2}.
\]
Therefore, by \eqref{eq:appC-HN-diagonal-B-bounds},
\begin{align*}
|w^T\nabla_{x,p}^2L^e(b,b)w|
&\leq C|w|^2\sum_{\ell=0}^{K-1}
\frac{1}{d^2+\sin^2(\ell\theta_0)}\\
&\leq C|w|^2\left(d^{-2}
+K\int_0^1\frac{dt}{d^2+t^2}\right)
\leq C|w|^2Kd^{-1},
\end{align*}
The polynomial and constant terms have zero mixed derivative,
so \eqref{eq:appC-HN-even-decomposition} gives the first expansion.

Finally, Table~\ref{tab:appC-d-sums} gives
\[
S_1^+-(d+\sqrt{1+d^2})^2S_3^+
+3(d^2+d^4)S_5^+
=\frac{2K}{\pi d^2}+O(Kd^{-1}).
\]
The last assertion follows.
\end{proof}

\section*{Acknowledgements}

The author is grateful to Wen Yang and Min Gao of the University
of Macau, and to Hong-Ge Chen of Central China Normal University,
for valuable insights that helped shape the ideas and approach
of this paper. The author also acknowledges the assistance of
OpenAI's ChatGPT in improving the exposition, locating relevant
literature, and checking calculations, particularly in  improving the rigor of the proof of
Proposition~\ref{prop:fdr-inner-parameter-dependence}.


\begin{thebibliography}{99}

\bibitem{BlagouchineMoreau2024}
I.~V.~Blagouchine and E.~Moreau,
\emph{On a finite sum of cosecants appearing in various problems},
J. Math. Anal. Appl. \textbf{539} (2024), no.~1, Art.~128515.
\href{https://doi.org/10.1016/j.jmaa.2024.128515}
{doi:10.1016/j.jmaa.2024.128515}.

\bibitem{Brezis2023}
H.~Brezis,
\emph{Some of my favorite open problems},
Rend. Lincei Mat. Appl. \textbf{34} (2023), no.~2, 307--335.
\href{https://doi.org/10.4171/RLM/1008}{doi:10.4171/RLM/1008}.

\bibitem{BrezisNirenberg1983}
H.~Brezis and L.~Nirenberg,
\emph{Positive solutions of nonlinear elliptic equations involving
critical Sobolev exponents},
Comm. Pure Appl. Math. \textbf{36} (1983), no.~4, 437--477.
\href{https://doi.org/10.1002/cpa.3160360405}
{doi:10.1002/cpa.3160360405}.

\bibitem{CaffarelliGidasSpruck1989}
L.~A.~Caffarelli, B.~Gidas, and J.~Spruck,
\emph{Asymptotic symmetry and local behavior of semilinear elliptic
equations with critical Sobolev growth},
Comm. Pure Appl. Math. \textbf{42} (1989), no.~3, 271--297.
\href{https://doi.org/10.1002/cpa.3160420304}
{doi:10.1002/cpa.3160420304}.

\bibitem{ChenLi1991}
W.~Chen and C.~Li,
\emph{Classification of solutions of some nonlinear elliptic equations},
Duke Math. J. \textbf{63} (1991), no.~3, 615--622.
\href{https://doi.org/10.1215/S0012-7094-91-06325-8}
{doi:10.1215/S0012-7094-91-06325-8}.

\bibitem{DelPinoMussoPacardPistoia2011}
M.~del Pino, M.~Musso, F.~Pacard, and A.~Pistoia,
\emph{Large energy entire solutions for the Yamabe equation},
J. Differential Equations \textbf{251} (2011), no.~9, 2568--2597.
\href{https://doi.org/10.1016/j.jde.2011.03.008}
{doi:10.1016/j.jde.2011.03.008}.

\bibitem{DelPinoMussoRomanWei2019}
M.~del Pino, M.~Musso, C.~Rom\'an, and J.~Wei,
\emph{Interior bubbling solutions for the critical Lin--Ni--Takagi
problem in dimension 3},
J. Anal. Math. \textbf{137} (2019), no.~2, 813--843.
\href{https://doi.org/10.1007/s11854-019-0008-8}
{doi:10.1007/s11854-019-0008-8}.

\bibitem{DevillanovaSolimini2002}
G.~Devillanova and S.~Solimini,
\emph{Concentration estimates and multiple solutions to elliptic
problems at critical growth},
Adv. Differential Equations \textbf{7} (2002), no.~10, 1257--1280.
\href{https://doi.org/10.57262/ade/1356651637}
{doi:10.57262/ade/1356651637}.

\bibitem{Ding1986}
W.~Ding,
\emph{On a conformally invariant elliptic equation on $\mathbb{R}^n$},
Comm. Math. Phys. \textbf{107} (1986), no.~2, 331--335.
\href{https://doi.org/10.1007/BF01209398}{doi:10.1007/BF01209398}.

\bibitem{Druet2002}
O.~Druet,
\emph{Elliptic equations with critical Sobolev exponents in dimension 3},
Ann. Inst. H. Poincar\'e C Anal. Non Lin\'eaire
\textbf{19} (2002), no.~2, 125--142.
\href{https://doi.org/10.1016/S0294-1449(02)00095-1}
{doi:10.1016/S0294-1449(02)00095-1}.

\bibitem{GuiWei1999}
C.~Gui and J.~Wei,
\emph{Multiple interior peak solutions for some singularly perturbed
Neumann problems},
J. Differential Equations \textbf{158} (1999), no.~1, 1--27.
\href{https://doi.org/10.1016/S0022-0396(99)80016-3}
{doi:10.1016/S0022-0396(99)80016-3}.

\bibitem{LinNi1988}
C.-S.~Lin and W.-M.~Ni,
\emph{On the diffusion coefficient of a semilinear Neumann problem},
in: S.~Hildebrandt, D.~Kinderlehrer, and M.~Miranda (eds.),
\emph{Calculus of Variations and Partial Differential Equations},
Lecture Notes in Mathematics, vol.~1340,
Springer, Berlin, 1988, pp.~160--174.
\href{https://doi.org/10.1007/BFb0082894}{doi:10.1007/BFb0082894}.

\bibitem{LinNiTakagi1988}
C.-S.~Lin, W.-M.~Ni, and I.~Takagi,
\emph{Large amplitude stationary solutions to a chemotaxis system},
J. Differential Equations \textbf{72} (1988), no.~1, 1--27.
\href{https://doi.org/10.1016/0022-0396(88)90147-7}
{doi:10.1016/0022-0396(88)90147-7}.

\bibitem{MussoSalazar2018}
M.~Musso and D.~Salazar,
\emph{Multispike solutions for the Brezis--Nirenberg problem
in dimension three},
J. Differential Equations \textbf{264} (2018), no.~11, 6663--6709.
\href{https://doi.org/10.1016/j.jde.2018.01.046}
{doi:10.1016/j.jde.2018.01.046}.

\bibitem{MussoWei2015}
M.~Musso and J.~Wei,
\emph{Nondegeneracy of nodal solutions to the critical Yamabe problem},
Comm. Math. Phys. \textbf{340} (2015), no.~3, 1049--1107.
\href{https://doi.org/10.1007/s00220-015-2462-1}
{doi:10.1007/s00220-015-2462-1}.

\bibitem{NiPanTakagi1992}
W.-M.~Ni, X.-B.~Pan, and I.~Takagi,
\emph{Singular behavior of least-energy solutions of a semilinear
Neumann problem involving critical Sobolev exponents},
Duke Math. J. \textbf{67} (1992), no.~1, 1--20.
\href{https://doi.org/10.1215/S0012-7094-92-06701-9}
{doi:10.1215/S0012-7094-92-06701-9}.

\bibitem{Rey1990}
O.~Rey,
\emph{The role of the Green's function in a non-linear elliptic equation
involving the critical Sobolev exponent},
J. Funct. Anal. \textbf{89} (1990), no.~1, 1--52.
\href{https://doi.org/10.1016/0022-1236(90)90002-3}
{doi:10.1016/0022-1236(90)90002-3}.

\bibitem{ReyWei2005}
O.~Rey and J.~Wei,
\emph{Arbitrary number of positive solutions for an elliptic problem
with critical nonlinearity},
J. Eur. Math. Soc. \textbf{7} (2005), no.~4, 449--476.
\href{https://doi.org/10.4171/JEMS/35}{doi:10.4171/JEMS/35}.

\bibitem{Salazar2019}
D.~Salazar,
\emph{Sign changing bubbling solutions for a critical Neumann problem
in 3D},
Nonlinear Anal. \textbf{188} (2019), 500--539.
\href{https://doi.org/10.1016/j.na.2019.06.018}
{doi:10.1016/j.na.2019.06.018}.

\bibitem{SchechterZou2010}
M.~Schechter and W.~Zou,
\emph{On the Br\'ezis--Nirenberg problem},
Arch. Ration. Mech. Anal. \textbf{197} (2010), no.~1, 337--356.
\href{https://doi.org/10.1007/s00205-009-0288-8}
{doi:10.1007/s00205-009-0288-8}.

\bibitem{SunWeiYangFull}
L.~Sun, J.-C.~Wei, and W.~Yang,
\emph{On Brezis' first open problem},
preprint, manuscript dated June~16, 2026.
An earlier version is available as
\emph{On Brezis' first open problem: A complete solution},
\href{https://arxiv.org/abs/2503.06904}{arXiv:2503.06904} (2025).
\href{https://doi.org/10.48550/arXiv.2503.06904}
{doi:10.48550/arXiv.2503.06904}.

\bibitem{SunWeiYangSmall}
L.~Sun, J.-C.~Wei, and W.~Yang,
\emph{On the 3-D Brezis--Nirenberg problem with small parameter},
Commun. Contemp. Math. \textbf{28} (2026), no.~5, Art.~2640003.
\href{https://doi.org/10.1142/S0219199726400031}
{doi:10.1142/S0219199726400031}.

\bibitem{TrefethenWeideman2014}
L.~N.~Trefethen and J.~A.~C.~Weideman,
\emph{The exponentially convergent trapezoidal rule},
SIAM Rev. \textbf{56} (2014), no.~3, 385--458.
\href{https://doi.org/10.1137/130932132}{doi:10.1137/130932132}.

\bibitem{WangWeiYan2010}
L.~Wang, J.~Wei, and S.~Yan,
\emph{A Neumann problem with critical exponent in nonconvex domains
and Lin--Ni's conjecture},
Trans. Amer. Math. Soc. \textbf{362} (2010), no.~9, 4581--4615.
\href{https://doi.org/10.1090/S0002-9947-10-04955-X}
{doi:10.1090/S0002-9947-10-04955-X}.

\bibitem{WangWeiYan2011}
L.~Wang, J.~Wei, and S.~Yan,
\emph{On Lin--Ni's conjecture in convex domains},
Proc. Lond. Math. Soc. (3) \textbf{102} (2011), no.~6, 1099--1126.
\href{https://doi.org/10.1112/plms/pdq051}{doi:10.1112/plms/pdq051}.

\bibitem{WeiXu2005}
J.~Wei and X.~Xu,
\emph{Uniqueness and a priori estimates for some nonlinear elliptic
Neumann equations in $\mathbb{R}^3$},
Pacific J. Math. \textbf{221} (2005), no.~1, 159--165.
\href{https://doi.org/10.2140/pjm.2005.221.159}
{doi:10.2140/pjm.2005.221.159}.

\bibitem{WeiYan2010}
J.~Wei and S.~Yan,
\emph{Infinitely many solutions for the prescribed scalar curvature
problem on $\mathbb{S}^N$},
J. Funct. Anal. \textbf{258} (2010), no.~9, 3048--3081.
\href{https://doi.org/10.1016/j.jfa.2009.12.008}
{doi:10.1016/j.jfa.2009.12.008}.

\bibitem{Zhu1999}
M.~Zhu,
\emph{Uniqueness results through a priori estimates.
I. A three-dimensional Neumann problem},
J. Differential Equations \textbf{154} (1999), no.~2, 284--317.
\href{https://doi.org/10.1006/jdeq.1998.3529}
{doi:10.1006/jdeq.1998.3529}.

\end{thebibliography}
\end{document}